\documentclass[a4paper, oneside]{amsart}
\usepackage{amsmath,amsthm,amsfonts,amssymb}
\usepackage{mathtools}
\usepackage[pdftex,bookmarks=true]{hyperref} 
\usepackage{graphicx}
\usepackage{subcaption}
\usepackage{enumitem}
\usepackage{tikz-cd}
\usepackage{bbm}
\usepackage{hyperref}
\usepackage[a4paper, textheight=24cm, textwidth=17cm, top=3cm]{geometry}

\newcommand{\R}{\mathbb{R}}
\newcommand{\Z}{\mathbb{Z}}

\makeatletter
\let\origsection\section
\renewcommand\section{\@ifstar{\starsection}{\nostarsection}}

\newcommand\nostarsection[1]
{\sectionprelude\origsection{#1}\sectionpostlude}

\newcommand\starsection[1]
{\sectionprelude\origsection*{#1}\sectionpostlude}

\newcommand\sectionprelude{%
  \vspace{1em}
}

\newcommand\sectionpostlude{%
  \vspace{1em}
}
\makeatother

\DeclareMathOperator{\wdim}{d}

\numberwithin{equation}{subsection}

\newtheoremstyle{referencedlemma}
  {12pt plus 4pt minus 2pt}   
  {12pt plus 4pt minus 2pt}   
  {\itshape}  
  {0pt}       
  {\bfseries} 
  {.}         
  {5pt plus 1pt minus 1pt} 
  {\thmname{#1}\thmnumber{ #2}\thmnote{ #3}} 

\theoremstyle{referencedlemma}
\newtheorem{lemref}[equation]{Lemma}

\theoremstyle{plain} 
\newtheorem{thm}{Theorem}

\newtheorem{lem}[equation]{Lemma}

\newtheorem*{theorem}{Theorem}

\theoremstyle{remark}

\title{Graded-fusion 2-categories and quantum homotopy invariants of 4-manifolds}
\author{Matthew Cellot}

\begin{document}

\begin{abstract}
We introduce 3-group-graded extensions of fusion 2-categories, where 3-groups are modeled by \mbox{2-crossed} modules.  From this data, we derive a state-sum invariant of 4-manifolds equipped with a homotopy class of maps to a homotopy~3-type, or equivalently of flat 3-bundles over 4-manifolds.  This invariant is nontrivial and generalizes the Douglas--Reutter invariant of 4-manifolds.  To construct our invariant, we encode homotopy classes of maps to the classifying space of a 3-group~$\sigma$ via $\sigma$-colorings of triangulations, and we generalize Pachner's theorem in this context. 
\end{abstract}

\maketitle

\section{Introduction}

\noindent Quantum topology is a branch of mathematics that emerged in the 1980s following groundbreaking work by Jones, Drinfeld, and Witten, which significantly renewed topology, particularly in low dimensions.  Drawing on ideas from quantum physics, their framework has led to the construction of new manifold invariants based on quantum groups, called quantum invariants.   One of the early successes in this direction is the Turaev--Viro--Barrett--Westbury invariant of closed 3-manifolds,  defined as a state sum using quantum 6j-symbols associated with a fusion category, see \cite{turaev_state_1992, barrett_invariants_1996}.  These quantum invariants fit together to form a 3-dimensional topological quantum field theory (TQFT)---that is, a functorial assignment of modules and linear maps to surfaces and 3-dimensional cobordisms, respectively.  

The Turaev--Viro state-sum approach was generalized by Turaev and Virelizier~\cite{turaev_3-dimensional_2012} to yield quantum invariants of 3-manifolds equipped with a (flat) $G$-bundle, for a discrete group $G$, using the data of a $G$-fusion category. Moreover, they showed that these invariants fit into a  homotopy quantum field theory (HQFT)---that is, a TQFT for manifolds equipped with a homotopy class of maps to a fixed target space, which in this case is the Eilenberg--MacLane space $K(G,1)$.  Sözer and Virelizier further generalized this construction to yield a 3-dimensional HQFT with target a connected homotopy 2-type, modeled by a crossed module $\chi$, using the data of a $\chi$-fusion category.  The resulting invariants are known as quantum homotopy invariants.

The Turaev--Viro state-sum approach has since been adapted to dimension 4 by Douglas and Reutter \cite{douglas_fusion_2018}.  They derive quantum invariants of closed 4-manifolds using the data of a fusion 2-category.  The Douglas--Reutter invariants generalize many 4-dimensional state-sum invariants including the Crane--Yetter--Kauffman invariant~\cite{crane_state-sum_1997}, the (twisted) Yetter--Dijkgraaf--Witten invariant \cite{yetter_tqfts_1993, martins_yetters_2007}, the Mackaay invariant \cite{mackaay_spherical_1999}, and the Cui invariant~ \cite{cui_four_2019}. Fusion 2-categories have since found applications to TQFTs via the cobordism hypothesis \cite{decoppet_dualizability_2024, decoppet_drinfeld_2025} and to the theory of braided fusion categories \cite{johnson-freyd_minimal_2023}. 

The aim of this paper is to adapt the Douglas--Reutter state-sum approach to obtain invariants of 4-manifolds equipped with a homotopy class of maps to a connected homotopy 3-type.   We model 3-types by 2-crossed modules, as introduced by Conduché \cite{conduche_modules_1984}.  A~2-crossed module $\sigma$ is given by a complex of groups
\begin{center}
\begin{tikzcd}L\arrow[r] & E \arrow[r] & H\end{tikzcd}
\end{center}
together with an action of $H$ on $L$ and $E$, and a pairing $E\times E \to L$, called the Peiffer lifting, all satisfying compatibility conditions.  We introduce the notion of a \emph{$\sigma$-graded linear monoidal 2-category} in which the objects, \mbox{1-morphisms}, and 2-morphisms are graded by~$H$, $E$, and $L$, respectively, in such a way that the tensor product is compatible with the action of $H$, and the interchange~2-isomorphism is compatible with the Peiffer lifting.  A \emph{$\sigma$-fusion 2-category} is a $\sigma$-graded linear monoidal 2-category with duals satisfying suitable semisimplicity and finiteness conditions.  Such structures should be viewed as graded extensions of fusion 2-categories; although $\sigma$-fusion 2-categories are not necessarily fusion 2-categories, they each contain a neutral component which is a fusion 2-category, see Section~\ref{Definition_Graded-fusion 2-categories}. 

A \emph{$\sigma$-manifold} is a closed oriented smooth 4-manifold equipped with a homotopy class of maps to the classifying space of the 2-crossed module $\sigma$.  This data is equivalent to that of a 4-manifold equipped with a flat 3-bundle for the 3-group modeled by $\sigma$, see \cite{faria_martins_fundamental_2011, jurco_nonabelian_2011, szabo_higher_2025, gagliardo_principal_2025, Cellot_thesis_2026}. We model a $\sigma$-manifold $(W,g)$ using the notion of a \emph{combinatorial $\sigma$-manifold}, that is a pair $(K,\phi)$, where~$K$ is a triangulation of $W$ and $\phi$ is a morphism from a 2-crossed module associated with $K$ to $\sigma$.  From any spherical $\sigma$-fusion~2-category $\mathcal{C}$ with invertible dimensions,  we derive a scalar invariant $\mathrm{HDR}_\mathcal{C}(W,g)$ as a sum over the labelings of the edges of $K$ by simple objects of $\mathcal{C}$ and of the triangles of~$K$ by simple 1-morphisms of $\mathcal{C}$, such that the gradings are compatible with~$\phi$. Our main result is the following:

\begin{theorem}
The state sum $\mathrm{HDR}_\mathcal{C}(W,g)$ is an invariant of $\sigma$-manifolds. 
\end{theorem}

This theorem appears as Theorem \ref{maintheorem} in Section \ref{Section. State sum}.  In the case where $\sigma$ is a trivial 2-crossed module,  so that its classifying space is contractible,  the invariant $\mathrm{HDR}_\mathcal{C}(W, W \to *)$ coincides with the Douglas--Reutter invariant of $W$ (see Section \ref{Ex.Douglas-Reutter}).  We show through examples that the quantum homotopy invariants $\mathrm{HDR}_\mathcal{C}$ are nontrivial and can even distinguish homotopy classes of phantom maps,  i.e. maps that induce trivial homomorphisms on homotopy groups (see Section \ref{Ex.nontrivial}).  We expect the invariants of Mochida \cite{mochidaInvariantsFlatConnections2026} and of Bridges--Cui \cite{bridges_involutory_2025} to be special cases of our construction when $\sigma = (1\to 1 \to G)$, for a group~$G$.  Our invariants are expected to be part of a 4-dimensional HQFT, the details of which are left for future work.

In order to prove the invariance of the state sum,  we introduce the fundamental 2-crossed module $\Pi_3(K)$ of a simplicial complex $K$ and the classifying space $B\sigma$ of a 2-crossed module $\sigma$.  Building on the work of Faria Martins \cite{martins_fundamental_2011},  we show the following:

\begin{theorem}[Homotopy Classification Theorem]
Let $K$ be a canonical simplicial complex and let $\sigma$ be a 2-crossed module of groups. There is a canonical bijection
$$
[\lvert K\rvert , B\sigma] \cong [\Pi_3(K), \sigma],
$$
where the left-hand side denotes homotopy classes of maps and the right-hand side denotes gauge-equivalence classes of morphisms of 2-crossed modules. 
\end{theorem}

This theorem appears as Theorem \ref{homotopy classification theorem} in Section \ref{Section. Classification of homotopies}.  Recall from \cite{lickorish_simplicial_1999} that two combinatorial manifolds are piecewise-linear homeomorphic if and only if they are related by a finite sequence of Pachner moves.  We use the Homotopy Classification Theorem to generalize this result:

\begin{theorem}[Colored Pachner Theorem]
Two combinatorial $\sigma$-manifolds are equivalent if and only if they are related by a finite sequence of $\sigma$-colored Pachner moves.
\end{theorem}

This theorem appears as Theorem \ref{Theorem - Colored Pachner} in Section \ref{Section-Pachner}.  The proof of the invariance of our state sum reduces to checking that it is invariant under $\sigma$-colored Pachner moves. 

This paper is organized as follows. In Section \ref{Section. 2-crossed modules}, we extend the definition of 2-crossed modules to the context of groupoids and we introduce the notion of a free 2-crossed module. In Section \ref{Section. Classification of homotopies}, we define the funda\-mental \mbox{2-crossed} module of a simplicial complex and the classifying space of a 2-crossed module, and we prove the Homotopy Classification Theorem.  Section \ref{Section. Graded monoidal 2-categories} is devoted to the definition of monoidal 2-categories graded by 2-crossed modules.  In Section \ref{Section. Graded-fusion 2-categories},  we define the notion of a graded-fusion 2-category. In Section \ref{Dimensions in graded-fusion 2-categories}, we introduce dimensions in graded-fusion 2-categories.  In Section \ref{Section. 10j-symbols},  we introduce the colored 10j-symbols. We use these in Section \ref{Section. State sum} to derive a state-sum invariant of $\sigma$-manifolds.  In Section \ref{Section-Pachner}, we prove the Colored Pachner Theorem.  Section \ref{Section. Proof of main theorem} is devoted to the proof of the invariance of the state sum of Theorem \ref{maintheorem}. Section \ref{Section_Proof of Lemma 8.3.1} contains a proof of Lemma \ref{Pachner_(1,5)}.

Throughout this paper, $\mathbbm{k}$ is a nonzero commutative ring. 

\section{Crossed modules and 2-crossed modules}\label{Section. 2-crossed modules}
\noindent The notion of a 2-crossed module was originally introduced by Conduché \cite{conduche_modules_1984} as a model for pointed homotopy~3-types.  In order to work in an unpointed setting, we extend the definition of a 2-crossed module to the context of groupoids. We then introduce free 2-crossed modules. 

\subsection{Notation} For any groupoid $H$, we denote by $H_0$ the set of objects of $H$, by $H(x,y)$ the set of morphisms from an object $x$ to an object $y$, and by $H(x)$ the group of endomorphisms of an object $x$. In a groupoid, the composition of a morphism $h\colon x \to y$ and a morphism $k\colon y\to z$ is denoted $hk\colon x\to z$.

\subsection{Crossed modules of groupoids}
In this section, we recall some definitions due to Brown and Higgins~\cite{brown_algebra_1981}.  

A groupoid is \emph{totally disconnected} if all of its morphisms are endomorphisms. If $E$ is a totally disconnected groupoid, we write
$$
E = (E(x))_{x\in E_0}.
$$

A (left) \emph{action} of a groupoid $H$ on a groupoid $E$ with the same set of objects as $H$ is a family of maps
$$
\{H(x,y) \times E(y) \to E(x), \ (h,e) \mapsto \prescript{h}{}{e}\}_{x,y\in H_0}
$$
such that 
$$
\prescript{k}{}(ef) = \prescript{k}{}e \prescript{k}{}f, \quad \quad \prescript{hk}{}e = \prescript{h}{}(\prescript{k}{}e), \quad\quad \prescript{1}{}e = e
$$
for all $x,y,z \in H_0$, $h\in H(x,y)$, $k \in H(y,z)$, and $e,f\in E(z)$. We say that a groupoid ${H}$ \emph{acts on} a groupoid ${E}$ if $H$ and $E$ have the same set of objects and there is an action of ${H}$ on ${E}$. 

A \emph{pre-crossed module} is a morphism $\partial\colon {E}\to {H}$ from a totally disconnected groupoid ${E}$ to a groupoid $H$ together with an action of ${H}$ on ${E}$ such that $\partial$ is the identity on objects and preserves the action of $H$, where~$H$ acts on itself by conjugation.

Let ${E}$ and ${E}'$ be totally disconnected groupoids with actions of ${H}$ and ${H}'$ respectively, and let~$\phi_1 : {H} \to {H}'$ be a morphism of groupoids. A morphism $\phi_2 : {E} \to {E}'$ is \emph{$\phi_1$-equivariant} if it coincides with $\phi_1$ on objects and satisfies
\begin{equation*}
\phi_2(\prescript{h}{}e) = \prescript{\phi_1(h)}{}\phi_2(e)
\end{equation*}
for all $x,y \in H_0$, $h\in H(x,y)$, and $e\in E(y)$.

A \emph{morphism} from a pre-crossed module $\partial \colon {E}\to {H}$ to a pre-crossed module $\partial'\colon {E}'\to {H}'$ is a pair of groupoid morphisms $\phi = (\phi_1\colon {H}\to {H}', \phi_2\colon {E}\to {E}')$ making the square 
\begin{center}
\begin{tikzcd}[row sep = large, column sep = large]
{E} \arrow[r, "\partial"]\arrow[d, "\phi_2"] & {H}\arrow[d, "\phi_1"]\\
{E}' \arrow[r, "\partial'"] & {H}'
\end{tikzcd}
\end{center}
commutative and such that $\phi_2$ is $\phi_1$-equivariant.

A \emph{crossed module} is a pre-crossed module $\partial\colon {E}\to {H}$ that also satisfies the \emph{Peiffer identity}:
$$
\prescript{\partial e}{}f = efe^{-1}
$$
for all $x \in H_0$ and $e,f \in E(x)$. A \emph{crossed module of groups} is a crossed module with set of objects a singleton.

\subsection{2-crossed modules of groupoids} 

A \emph{2-crossed module} is a tuple\begin{tikzcd}\sigma = (L \arrow[r, "\delta"] & E \arrow[r, "\partial"] & H, \ \omega)\end{tikzcd}where 
\begin{itemize}
\item $\partial \colon E \to H$ is a pre-crossed module
\item $L$ is a totally disconnected groupoid and $H$ acts on $L$
\item $\delta$ is a morphism of groupoids that is the identity on objects and preserves the action of $H$, 
\item the \emph{Peiffer lifting} $\omega$ is a collection of set maps
$$
\omega = (\omega_x \colon E(x)\times E(x) \to L(x))_{x\in E_0}
$$
that preserve the action of $H$,
\end{itemize}
satisfying the following conditions:
\begin{align*}
\partial \circ \delta(l) &= 1_x, \quad &\delta \left(\omega_x(e,f)\right) = efe^{-1}\prescript{\partial (e)}{}{f^{-1}}, \quad \omega_x(\delta (l), \delta (k)) &= lkl^{-1}k^{-1},\\
\omega_x(\delta (l), e)\omega_x(e, \delta (l)) &= l\prescript{\partial (e)}{}l^{-1}, \quad &\omega_x(ef,g) = \omega_x(e,fgf^{-1})\prescript{\partial (e)}{}\omega_x(f,g),\quad \omega_x(e,fg) &= \omega_x(e,f) (\prescript{\partial (e)}{}f)\triangleright \omega_x(e,g),
\end{align*}
for all $x\in H_0$, $e, f, g\in E(x)$, and $l,k\in L(x)$, where we set 
\begin{equation*}
e \triangleright l = l\omega_x(\delta(l)^{-1}, e).
\end{equation*}
Notice that $\triangleright$ defines an action of ${E}$ on ${L}$, and that together with this action, the morphism $\delta \colon {L}\to {E}$ defines a crossed module. A \emph{2-crossed module of groups} is a 2-crossed module with set of objects a singleton.

\subsection{Morphisms of 2-crossed modules}
A \emph{morphism} of 2-crossed modules $\phi \colon \sigma \to \sigma'$, where
$$
\begin{tikzcd}\sigma = ({L} \arrow[r, "\delta"] & {E} \arrow[r, "\partial"] & {H}, \ \omega)\end{tikzcd}\quad\text{and} \quad\begin{tikzcd}\sigma' = ({L}' \arrow[r, "\delta'"] & {E}' \arrow[r, "\partial'"] & {H}', \ \omega')\end{tikzcd}
$$
is a triple $\phi = (\phi_1\colon  {H} \to {H}', \phi_2 \colon  {E} \to {E}', \phi_3 \colon {L} \to {L}')$ of groupoid morphisms making the diagram
\begin{center}
\begin{tikzcd}[row sep=large,column sep=large]
{E}\times_{E_0}{E} \arrow[r, "\omega"]\arrow[d, "{(\phi_2, \phi_2)}"] & {L} \arrow[r, "\delta"] \arrow[d, "\phi_3"] & {E} \arrow[r, "\partial"] \arrow[d,"\phi_2"] & {H}\arrow[d, "\phi_1"]\\
{E'}\times_{E'_0} {E'}\arrow[r, "\omega'"] & {L'} \arrow[r, "\delta'"] & {E'} \arrow[r, "\partial'"] & {H'}
\end{tikzcd}
\end{center}
commutative and such that the morphisms $\phi_2$ and $\phi_3$ are $\phi_1$-equivariant. Here, given a totally disconnected groupoid $G$, we denote by $G\times_{G_0} G$ the totally disconnected groupoid with set of objects $G_0$ and with
$$
(G\times_{G_0} G)(x) = G(x) \times G(x),
$$
 for all $x\in G_0$. 

\subsection{Special cases}

1) If a 2-crossed module is of the form
$$
\begin{tikzcd}\sigma = (\overline{H_0} \arrow[r, "1"] & E \arrow[r, "\partial"] & H, \ \omega), \end{tikzcd}
$$
where $\overline{H_0}$ denotes the discrete groupoid with set of objects $H_0$, then the morphism $\partial \colon E\to H$ together with the action of $H$ on $E$ defines a crossed module.  Conversely, any crossed module induces a 2-crossed module of the above form. 

2) If a 2-crossed module is of the form
$$
\begin{tikzcd}\sigma = (L \arrow[r, "\delta"] & E \arrow[r] & 1, \ \omega),\end{tikzcd}
$$
where $1$ denotes the trivial groupoid, then the morphism $\delta\colon L\to E$ together with the action $\triangleright$ of $E$ on $L$ defines a braided crossed module of groups (in other words, the associated monoidal category is braided). 

3) Let 
$$
\begin{tikzcd}\sigma = (L \arrow[r, "\delta"] & E \arrow[r, "\partial"] & H, \ \omega),\end{tikzcd}
$$
be a 2-crossed module of groups such that $\mathrm{Ker}(\partial) / \mathrm{Im}(\delta) = 1$. Then
$$
1 \to \mathrm{Ker}(\delta) \to L \to E \to H \to \mathrm{Coker}(\partial) \to 1
$$
is an exact sequence of groups and determines a cohomology class in $H^4(\mathrm{Coker}(\partial); \mathrm{Ker}(\delta))$.  Such a 2-crossed module is equivalently characterized by the tuple
$$
(\mathrm{Coker}(\partial), \mathrm{Ker}(\delta), \rho, \gamma),
$$
where $\rho\colon \mathrm{Coker}(\partial) \to \mathrm{Aut}(\mathrm{Ker}(\delta))$ is induced by the action of $H$ on $L$, and the class $\gamma\in H^4(\mathrm{Coker}(\partial); \mathrm{Ker}(\delta))$ is induced by the above exact sequence, see details in \cite{eilenberg_homology_1949, conduche_modules_1984}. 

4) If a 2-crossed module is of the form
$$
\begin{tikzcd}\sigma = (L \arrow[r, "\delta"] & E \arrow[r, "\partial"] & H, 1),\end{tikzcd}
$$
where $1$ denotes the trivial Peiffer lifting, then\begin{tikzcd}L \arrow[r, "\delta"] & E \arrow[r, "\partial"] & H\end{tikzcd}is a crossed complex of length 3 (see~\cite{brown_nonabelian_2011}). Conversely, any crossed complex of length 3 induces a 2-crossed module with trivial Peiffer lifting. 

\subsection{Examples}

1) Given a 2-crossed module\begin{tikzcd}\sigma = (L \arrow[r, "\delta"] & E \arrow[r, "\partial"] & H, \ \omega),\end{tikzcd}the truncation $\partial \colon E \to H$ is a pre-crossed module. This functor has a left adjoint that associates to any pre-crossed module $\partial \colon E\to H$ a 2-crossed module
$$
\begin{tikzcd}(\{E, E\} \arrow[r] & E \arrow[r, "\partial"] & H, \ \omega),\end{tikzcd}
$$
where $\{E, E\} = (\{E,E\}(x))_{x\in H_0}$ and $\{E,E\}(x)$ denotes the free group on the set $\{(e,f) \mid e,f \in E(x)\}$ quotiented by the 2-crossed module relations. The details of this construction can be found in \cite{carrasco_coproduct_2016}.

2) To any crossed square (of groups)
$$
\begin{tikzcd}
L\arrow[d, "\lambda"]\arrow[r, "{\lambda'}"] & N \arrow[d, "\mu"]\\
M \arrow[r, "\nu"] & P
\end{tikzcd}
$$
with function $h\colon M\times N \to L$ (see details in \cite{ellis_crossed_1993}), we associate a 2-crossed module of groups
$$
\begin{tikzcd}
L\arrow[r, "{(\lambda, \lambda'^{-1})}"] & M\rtimes N \arrow[r, "\mu\nu"] & P
\end{tikzcd}
$$
with Peiffer lifting $\omega$ defined by
$$
\omega((m,n),(m',n')) = h(m, nn'n^{-1}).
$$
for all $m, m' \in M$ and $n, n'\in N$.

3) The main example of a 2-crossed module arises from a triad $(X,A,B)$ of topological spaces and $X_0 \subset A\cap B$ a non-empty set.  The \emph{fundamental 2-crossed module} of $(X, A, B; X_0)$ is the 2-crossed module
$$
\Pi_3(X,A,B; X_0) = (\pi_3(X, A, B; X_0) \to \pi_2(A, A\cap B; X_0) \to \pi_1(B; X_0), \ \omega),
$$
where $\pi_3(X,A,B; X_0) = (\pi_3(X,A,B; x))_{x\in X_0}$, and $\pi_3(X,A,B; x)$ is given by the homotopy classes of maps from a pointed 3-ball that send the southern hemisphere to $A$, the northern hemisphere to $B$, and the basepoint to $x$, the map
$$
\pi_3(X, A, B; X_0) \to \pi_2(A, A\cap B; X_0)
$$
is induced by the restriction to the southern hemisphere, and the map $\pi_2(A, A\cap B; X_0) \to \pi_1(B; X_0)$ is induced by the inclusion $A\cap B \subset B$.  The Peiffer lifting $\omega$ is defined via the generalized Whitehead product, see~\cite{martins_fundamental_2011, ellis_crossed_1993}.  The details of this construction for certain triads of simplicial complexes can be found in Section~\ref{Section_Fundamental 2-crossed module}.

\subsection{Free 2-crossed modules}\label{Free 2-crossed modules}

Let $H_0$ be a set. An \emph{$H_0$-set} is a family of sets $B = (B(x))_{x\in H_0}$ indexed by~$H_0$.  An \emph{$H_0$-map} from an $H_0$-set $B = (B(x))_{x\in H_0}$ to an $H_0$-set $C = (C(x))_{x\in H_0}$ is a family of set maps
$$
f = (f_x\colon B(x) \to C(x))_{x\in H_0}.
$$ 

A pre-crossed module $\partial \colon E \to H$ is \emph{totally free} if $H$ is a free groupoid and there is an $H_0$-set $B$, called a \emph{basis}, and an~$H_0$-map $f\colon B\to E$ such that the following universal property is satisfied: for every pre-crossed module of the form $\partial' \colon E'\to H$ and for every $H_0$-map $f'\colon B \to E'$ such that $\partial' \circ f' = \partial \circ f$, there exists a unique morphism $\phi \colon E \to E'$ such that $\phi \circ f = f'$ and $(\phi, \mathrm{Id}_H)$ is a morphism of pre-crossed modules.  We obtain the following commutative diagram
$$
\begin{tikzcd}[row sep = large, column sep = large]
B\arrow[r, "f"] \arrow[swap, dr, "{f'}"] & E \arrow[r,  "\partial"]\arrow[d, dashed, "\phi"] & H\arrow[d, equal] \\
& E' \arrow[r, "{\partial'}"] & H.
\end{tikzcd}
$$

A 2-crossed module\begin{tikzcd}\sigma = ({L} \arrow[r, "\delta"] & {E} \arrow[r, "\partial"] & {H}, \ \omega)\end{tikzcd}is \emph{free} if $\partial \colon E \to H$ is a totally free pre-crossed module.  

\subsection{The model category of 2-crossed modules of groups}
The \emph{homotopy groups} of a 2-crossed module of groups\begin{tikzcd}\sigma = (L\arrow[r, "\delta"] & E \arrow[r, "\partial"] & H, \ \omega)\end{tikzcd}are the groups
$$
\pi_1(\sigma) = \mathrm{Coker}(\partial)\quad\quad \pi_2(\sigma) = \mathrm{Ker}(\partial) / \mathrm{Im}(\delta) \quad\quad \pi_3(\sigma ) = \mathrm{Ker}(\delta).
$$
Cabello and Garzón \cite{cabello1994quillen} show that the category of 2-crossed modules of groups is a Quillen model category with \emph{weak equivalences} given by the morphisms that induce isomorphisms at the level of homotopy groups.  A \emph{fibration} is a morphism of 2-crossed modules
\begin{center}
\begin{tikzcd}[row sep=large,column sep=large]
{L} \arrow[r, "\delta"] \arrow[d, "\phi_3"] & {E} \arrow[r, "\partial"] \arrow[d,"\phi_2"] & {H}\arrow[d, "\phi_1"]\\
{L'} \arrow[r, "\delta'"] & {E'} \arrow[r, "\partial'"] & {H'}
\end{tikzcd}
\end{center}
such that $\phi_2$ and $\phi_3$ are surjective.  A \emph{cofibrant object} is a free 2-crossed module of groups.  The homotopy category of~2-crossed modules is equivalent to the homotopy category of connected homotopy 3-types (i.e., pointed connected CW complexes with trivial $n$-th homotopy groups for $n>3$).

\section{The Homotopy Classification Theorem}\label{Section. Classification of homotopies}

\noindent In this section we introduce the fundamental 2-crossed module $\Pi_3(K)$ of a canonical simplicial complex $K$, and we define the classifying space $B\sigma$ of a 2-crossed module of groups $\sigma$.  The main result of this section is:

\begin{thm}[Homotopy Classification Theorem]\label{homotopy classification theorem}
Let $K$ be a canonical simplicial complex and let $\sigma$ be a~2-crossed module of groups. There is a canonical bijection
$$
[\lvert K\rvert , B\sigma] \cong [\Pi_3(K), \sigma]
$$
where the left-hand side is the set of homotopy classes of maps from the underlying space of $K$ to the classifying space of $\sigma$, and the right-hand side is the set of gauge-equivalence classes of morphisms of 2-crossed modules.
\end{thm}

A similar result is shown in \cite{martins_fundamental_2011} for pointed homotopy classes of maps.   Theorem \ref{homotopy classification theorem} provides a useful combinatorial description of the set of homotopy classes of maps from $K$ to the classifying space of $\sigma$ and will be an essential ingredient in the proof of Theorem \ref{Theorem - Colored Pachner}.  Its proof builds on the work of Brown--Loday \cite{brown_van_1987}, Berger \cite{bergerDoubleLoopSpaces1999}, Ellis \cite{ellis_crossed_1993}, and Faria Martins~\cite{martins_fundamental_2011}.

\subsection{Simplicial complexes}\label{simplicialcomplex}
In this section, we recall classical notions relating to simplicial complexes. An \emph{$n$-simplex} (or simply a \emph{simplex}), for some nonnegative integer $n$, is the convex hull $(v_0, \ldots, v_n)$ of $n+1$ affinely independent points $v_0, \ldots v_n$ in $\mathbb{R}^N$, for some $N\geqslant n$.  The \emph{standard $n$-simplex} is the convex hull $\Delta^n$ of the standard basis vectors of $\R^{n+1}$.  An \emph{$m$-face} (or simply a \emph{face}) of an $n$-simplex $s$, for some nonnegative integer~$m\leqslant n$, is the convex hull of a subset of size $m+1$ of the~$n+1$ points that define $s$.  In particular, an $m$-face of a simplex is an $m$-simplex.  A \emph{vertex} is a $0$-simplex, an \emph{edge} is a $1$-simplex, a \emph{triangle} is a $2$-simplex, and a \emph{tetrahedron} is a $3$-simplex.  

A \emph{geometric simplicial complex} (or simply a \emph{simplicial complex}) is a set $K$ of simplices in $\R^N$, for some nonnegative integer $N$, such that every face of a simplex in $K$ is also a simplex in $K$,  the intersection of any two simplices in $K$ is either empty or a face of both of these simplices, and $K$ is locally finite, i.e.,  for every point $x$ in a simplex in $K$, there is a neighborhood of $x$ in $\R^N$ that intersects a finite number of simplices in~$K$.  

Let $K$ denote a simplicial complex in $\R^N$. We denote by $K^{(i)}$ the set of all $i$-simplices in~$K$.  A \emph{simplicial subcomplex} of $K$ is a simplicial complex $L$ such that every simplex in $L$ is a simplex in $K$.  The \emph{underlying space} of $K$ is the union $\lvert K \rvert$ in $\R^N$ of all of the simplices in $K$
$$
\lvert K \rvert = \bigcup_{s\in K}s,
$$
with the topology induced by the topology of $\R^N$.  Notice that a simplicial complex is finite if and only if its underlying space is compact.  The \emph{$n$-skeleton} of a simplicial complex $K$ is the union $K_n$ in $\R^N$ of all of the $i$-simplices in $K$ such that $0 \leqslant i\leqslant n$
$$
K_n = \bigcup_{i\leqslant n}\bigcup_{s\in K^{(i)}} s. 
$$
This induces a filtration of topological spaces $K_0 \subset K_1 \subset K_2 \subset \cdots \subset \lvert K \rvert$. 

A \emph{subdivision} of a simplicial complex $K$ is a simplicial complex $K'$ such that every simplex in $K'$ is a subset of a simplex in $K$, and every simplex in $K$ is a finite union of simplices in $K'$.  Notice that a simplicial complex and any of its subdivisions have the same underlying space.  A \emph{piecewise-linear map} from a simplicial complex~$K$ to a simplicial complex $L$ is a continuous map $f\colon \lvert K\rvert \to \lvert L \rvert$ such that there exists a subdivision $K'$ of $K$ for which $f$ is affine when restricted to each simplex in $K'$.

Let $S$ be a subset of a simplicial complex $K$.  The \emph{star} of $S$ is the set $\operatorname{St}(S)$ of simplices in $K$ that have a face in $S$.  The \emph{closure} of $S$ is the smallest simplicial subcomplex $\mathrm{Cl}(S)$ of $K$ containing $S$.  The \emph{link} of $S$ is
$$
\operatorname{Lk}(S) = \operatorname{Cl}(\operatorname{St}(S))\setminus \operatorname{St}(S).
$$

\subsection{Canonical simplicial complexes}\label{Canonical simplicial complexes}A \emph{canonical simplicial complex} is a simplicial complex $K$ together with the choice,  for every triangle $t$ of $K$,  of a distinguished vertex $p_t$ and of a decomposition of $t$ into subspaces~$t^+$ and~$t^-$ as follows.  Let~$D$ be a closed disk embedded in $t$ such that $D\cap \partial t = \{p_t\}$.  We set 
$$
t^- = \partial t \cup D \quad \quad t^+ = t \setminus \mathring{D},
$$
where $\mathring{D}$ denotes the interior of $D$.  

\begin{figure}[h]
\centering
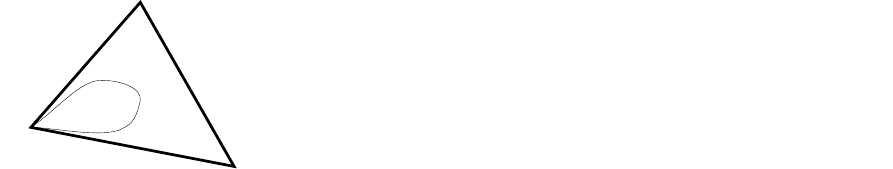
\end{figure}

Let $K$ be a canonical simplicial complex. Denote by $K_2^-$ and $K_2^+$ the subspaces of $K_2$ defined by
\begin{align*}
K_2^- &= K_1 \cup \left(\bigcup\limits_{t \in K^{(2)}} t^-\right), &
K_2^+ &= K_1 \cup \left(\bigcup\limits_{t \in K^{(2)}} t^+\right).
\end{align*}
Notice that $K_2 = K_2^- \cup K_2^+$, and that the inclusions $K_1 \xhookrightarrow{} K_2^-$ and $K_1\xhookrightarrow{} K_2^+$ induce homotopy equivalences.

\subsection{Homotopy groupoids} The \emph{fundamental groupoid} of a topological space $X$ is the groupoid $\pi_1(X)$ with set of objects $X$ and with morphisms from $x$ to $y$ given by the homotopy classes relative to the boundary of paths from $x$ to $y$, where the composition is induced by the concatenation of paths. The \emph{fundamental groupoid of~$X$ with chosen basepoints in $A$}, denoted by $\pi_1(X;A)$, for any subspace $A\subset X$, is the full subcategory of~$\pi_1(X)$ with set of objects $A$.

The \emph{$n$-th relative homotopy groupoid} of a triple of topological spaces $A\subset Y \subset X$ for $n\geqslant 2$ is the totally disconnected groupoid $\pi_n(X, Y;A)$ with set of objects $A$ and with sets of morphisms:
\begin{equation*}
\pi_n(X,Y;A)(x) = \pi_n(X,Y,x)
\end{equation*}
for all $x\in A$, where $\pi_n(X,Y,x)$ denotes the $n$-th relative homotopy group of the pair $Y\subset X$ with basepoint~\mbox{$x\in Y$}.  Notice that there is an action of the fundamental groupoid $\pi_1(Y;A)$ on the $n$-th relative homotopy groupoid $\pi_n(X,Y;A)$ for all $n\geqslant 2$.

The \emph{$n$-th triad homotopy groupoid} of a quadruple of spaces $Y \subset X$, $Z\subset X$, $A\subset Y\cap Z$ for~$n\geqslant 3$ is the totally disconnected groupoid $\pi_n(X,Y,Z;A)$ with set of objects $A$ and sets of morphisms:
\begin{equation*}
\pi_n(X,Y,Z;A)(x) = \pi_n(X,Y,Z,x)
\end{equation*}
for all $x\in A$, where $\pi_n(X,Y,Z,x)$ denotes the $n$-th triad homotopy group, see \cite{blakers_homotopy_1949}. Notice that there is an action of the fundamental groupoid $\pi_1(Y\cap Z; A)$ on the $n$-th triad homotopy groupoid $\pi_n(X,Y,Z;A)$ for all~$n\geqslant 3$.

\subsection{The fundamental 2-crossed module of a canonical simplicial complex}\label{Section_Fundamental 2-crossed module}
In this section, we adapt the definition of Faria Martins \cite{martins_fundamental_2011} of the fundamental 2-crossed module of groups of a reduced CW complex to the context of simplicial complexes and groupoids. The \emph{fundamental 2-crossed module} of a canonical simplicial complex $K$ is the 2-crossed module:
\[
\begin{tikzcd}
\Pi_3(K) = (\pi_3(\lvert K\rvert , K_2^-, K_2^+; K_0) \arrow[r, "d_2"] & \pi_2(K_2^-, K_2^-\cap K_2^+; K_0) \arrow[r, "d_1"] & \pi_1(K_1; K_0), \ \omega)
\end{tikzcd}
\]
where the maps $d_1$, $d_2$, and $\omega$ are defined as follows. The map $d_1 \colon \pi_2(K_2^-, K_2^-\cap K_2^+; K_0) \to \pi_1(K_1; K_0)$ is the composition of the connecting homomorphism
$$
\pi_2(K_2^-, K_2^-\cap K_2^+; K_0)\to \pi_1(K_2^-\cap K_2^+; K_0)
$$
from the long exact sequence of the pair $(K_2^-, K_2^-\cap K_2^+)$ with the morphism $\pi_1(K_2^- \cap K_2^+; K_0) \to \pi_1(K_2^+; K_0)$ induced by the inclusion~\mbox{$K_2^-\cap K_2^+ \subset K_2^+$} and the isomorphism $\pi_1(K_2^+; K_0) \cong \pi_1(K_1; K_0)$ induced by the inclusion $K_1 \subset K_2^+$. The map $d_2 \colon \pi_3(\lvert K\rvert, K_2^-, K_2^+; x) \to \pi_2(K_2^-, K_2^-\cap K_2^+; x)$ is the connecting homomorphism from the long exact sequence of the triad $(\lvert K \rvert, K_2^-, K_2^+)$.  The map
$$
\omega = \left(\omega_x \colon \pi_2(K_2^-, K_2^-\cap K_2^+; x)\times \pi_2(K_2^-, K_2^-\cap K_2^+; x)\to  \pi_3(\lvert K\rvert, K_2^-, K_2^+; x)\right)
$$
is defined via the generalized Whitehead product for the triad $(\lvert K \rvert, K_2^-, K_2^+)$, see details in \cite[Definition~3.1]{martins_fundamental_2011}.  We deduce from the discussion in \cite[Theorem 2]{ellis_crossed_1993} and from \cite{brown_van_1987} that the fundamental 2-crossed module of a canonical simplicial complex is free in the sense of Section \ref{Free 2-crossed modules}.  

We deduce from the freeness of the fundamental 2-crossed module that if $K$ and $K'$ are canonical simplicial complexes with the same underlying simplicial complex, then their fundamental 2-crossed modules are isomorphic.

\subsection{Colorings}

Let $\sigma$ be a 2-crossed module of groups. A \emph{$\sigma$-coloring} of a canonical simplicial complex $K$ is a morphism of 2-crossed modules $\Pi_3(K) \to \sigma$.  A \emph{$\sigma$-colored simplicial complex} is a pair $(K,\phi)$, where $K$ is a canonical simplicial complex and $\phi$ is a $\sigma$-coloring of $K$.

\subsection{Ordered simplices}\label{Section_Ordered simplices}Let $K$ be a canonical simplicial complex. An \emph{order} of an $i$-simplex of $K$, for~$i\neq 2$, is  a choice of linear order on its vertices. An \emph{order} of a triangle of~$K$ is a choice of cyclic order on its vertices. An $i$-simplex of $K$ is \emph{ordered} if it is endowed with an order.

A \emph{pointing} of an ordered $i$-simplex $s$ of $K$, for $i\geqslant 3$, is the choice, for every triangle $t$ of $s$, of a homotopy class of paths in the boundary of $t$ from the vertex of $t$ that is minimal for the order of $s$ to the distinguished vertex of $t$. A \emph{pointed ordered} $i$-simplex, for $i\geqslant 3$, is an ordered $i$-simplex together with a choice of pointing. The \emph{pointed vertex} of a pointed $i$-simplex, for $i\geqslant 3$, is the vertex that is minimal for the order.

\subsection{Choices of ordered simplices of canonical simplicial complexes}\label{Subsection_Choices of ordered simplices}
A \emph{choice of ordered simplices} of a canonical simplicial complex~$K$ is a sequence $B_* = (B_n)_{n\geqslant 1}$ such that
\begin{itemize}
\item for $i = 1, 2$, $B_i$ is a subset of the ordered $i$-simplices of $K$ containing exactly one choice of order for each $i$-simplex,
\item for $i\geqslant 3$, $B_i$ is a subset of the pointed ordered $i$-simplices of $K$ containing exactly one choice of pointing and order for each $i$-simplex.
\end{itemize}

Let $B_* = (B_n)_{n\geqslant 1}$ be a choice of ordered simplices of a canonical simplicial complex $K$. There is a canonical injective map $B_1\to \pi_1(K_1; K_0)$ that sends an ordered edge to its homotopy class.  Therefore, we will often identify $B_1$ with its image in $\pi_1(K_1; K_0)$.  There is an injective map $B_2\to  \pi_2(K_2^-, K_2^-\cap K_2^+; K_0)$ that sends an ordered triangle with distinguished vertex $p\in K_0$ to its homotopy class in~\mbox{$\pi_2(K_2^-, K_2^-\cap K_2^+; p)$}, and we use this to identify $B_2$ with its image in $ \pi_2(K_2^-, K_2^-\cap K_2^+; K_0)$.  There is an injective map~\mbox{$B_3\to \pi_3(K_3, K_2^-, K_2^+; K_0)$} that sends a pointed ordered tetrahedron $\tau$ with pointed vertex $p$ to its homotopy class in $\pi_3(K_3, K_2^-, K_2^+;p)$, and we use this to identify $B_3$ with its image in $\pi_3(K_3, K_2^-, K_2^+;K_0)$.  Similarly, we identify $B_4$ with its image in~$\pi_4(K_4, K_3; K_0)$.   

Notice that $\pi_1(K_1;K_0)$ is a free groupoid with basis $B_1$, and that~\mbox{$d_1\colon \pi_2(K_2^-, K_2^-\cap K_2^+; K_0)\to \pi_1(K_1; K_0)$} is a totally free pre-crossed module with basis $B_2$.  Therefore, by freeness, a $\sigma$-coloring of $K$ is uniquely determined by its values on $B_*$, see \cite{brown_van_1987}.

Conversely, a triple of maps $c = (c_1 \colon B_1 \to H, c_2\colon B_2 \to E, c_3\colon B_3\to L)$ induces a $\sigma$-coloring of $K$ if it verifies the \emph{triangle identity}:
$$
\partial (c_2(t)) = \widehat{c_1}(d_1(t))
$$
for all $t\in B_2$, where $\widehat{c_1}\colon \pi_1(K_1; K_0)\to H$ is the unique morphism of groupoids that extends $c_1$, the \emph{tetrahedron identity}:
$$
\delta (c_3(\tau)) = \widehat{c_2}(d_2(\tau))
$$
for all $\tau \in B_3$, where $\widehat{c_2}\colon  \pi_2(K_2^-, K_2^-\cap K_2^+; K_0)\to E$ is the unique $\widehat{c_1}$-equivariant morphism of groupoids that extends $c_2$, and the \emph{4-simplex identity}:
$$
1  =\widehat{c_3}(d_3(s))
$$
for all $s\in B_4$, where $\widehat{c_3}\colon \pi_3(K_3, K_2^-, K_2^+; K_0)\to L$ denotes the unique morphism of groupoids that coincides with $c_3$ such that $(\widehat{c_1}, \widehat{c_2}, \widehat{c_3})$ is a morphism of 2-crossed modules.

\subsection{The classifying space of a 2-crossed module of groups}\label{Section_Classifying space}

In order to define the classifying space of a 2-crossed module of groups, we first recall the notion of a canonical CW complex.  We refer the reader to Ellis \cite{ellis_crossed_1993} and Faria Martins \cite{martins_fundamental_2011} for a more detailed presentation of canonical CW complexes.  We denote
$$
D = \{z\in \mathbb{C} \mid \lvert z - 1 \rvert \leqslant 1\}\quad \quad D^- = \{z\in \mathbb{C} \mid \lvert z - \frac{1}{2}\rvert \leqslant \frac{1}{2}\} \quad \quad D^+ = D \setminus \mathrm{int}(D^-).
$$
This defines a triad $(D, D^-, D^+)$ with basepoint $z = 0$, as illustrated in Figure \ref{fig:disk triad}.

\begin{figure}[h]
\centering
\begingroup%
  \makeatletter%
  \providecommand\color[2][]{%
    \errmessage{(Inkscape) Color is used for the text in Inkscape, but the package 'color.sty' is not loaded}%
    \renewcommand\color[2][]{}%
  }%
  \providecommand\transparent[1]{%
    \errmessage{(Inkscape) Transparency is used (non-zero) for the text in Inkscape, but the package 'transparent.sty' is not loaded}%
    \renewcommand\transparent[1]{}%
  }%
  \providecommand\rotatebox[2]{#2}%
  \newcommand*\fsize{\dimexpr\f@size pt\relax}%
  \newcommand*\lineheight[1]{\fontsize{\fsize}{#1\fsize}\selectfont}%
  \ifx\svgwidth\undefined%
    \setlength{\unitlength}{61.64958563bp}%
    \ifx\svgscale\undefined%
      \relax%
    \else%
      \setlength{\unitlength}{\unitlength * \real{\svgscale}}%
    \fi%
  \else%
    \setlength{\unitlength}{\svgwidth}%
  \fi%
  \global\let\svgwidth\undefined%
  \global\let\svgscale\undefined%
  \makeatother%
  \begin{picture}(1,0.97623124)%
    \lineheight{1}%
    \setlength\tabcolsep{0pt}%
    \put(0,0){\includegraphics[width=\unitlength,page=1]{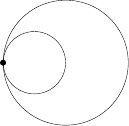}}%
    \put(0.19558052,0.46950297){\makebox(0,0)[lt]{\lineheight{1.25}\smash{\begin{tabular}[t]{l}$D^-$\end{tabular}}}}%
    \put(0.64324987,0.46950297){\makebox(0,0)[lt]{\lineheight{1.25}\smash{\begin{tabular}[t]{l}$D^+$\end{tabular}}}}%
  \end{picture}%
\endgroup%

\caption{The triad structure on the disk}
\label{fig:disk triad}
\end{figure}

We also define a triad structure on the sphere $S^2$, taking $(S^2)^+$ to be the northern hemisphere and $(S^2)^-$ to be the southern hemisphere, with a basepoint on the equator, as illustrated in Figure \ref{fig:sphere triad}. 

\begin{figure}[h]
\centering
\begingroup%
  \makeatletter%
  \providecommand\color[2][]{%
    \errmessage{(Inkscape) Color is used for the text in Inkscape, but the package 'color.sty' is not loaded}%
    \renewcommand\color[2][]{}%
  }%
  \providecommand\transparent[1]{%
    \errmessage{(Inkscape) Transparency is used (non-zero) for the text in Inkscape, but the package 'transparent.sty' is not loaded}%
    \renewcommand\transparent[1]{}%
  }%
  \providecommand\rotatebox[2]{#2}%
  \newcommand*\fsize{\dimexpr\f@size pt\relax}%
  \newcommand*\lineheight[1]{\fontsize{\fsize}{#1\fsize}\selectfont}%
  \ifx\svgwidth\undefined%
    \setlength{\unitlength}{61.68975169bp}%
    \ifx\svgscale\undefined%
      \relax%
    \else%
      \setlength{\unitlength}{\unitlength * \real{\svgscale}}%
    \fi%
  \else%
    \setlength{\unitlength}{\svgwidth}%
  \fi%
  \global\let\svgwidth\undefined%
  \global\let\svgscale\undefined%
  \makeatother%
  \begin{picture}(1,0.97559562)%
    \lineheight{1}%
    \setlength\tabcolsep{0pt}%
    \put(0,0){\includegraphics[width=\unitlength,page=1]{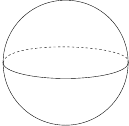}}%
    \put(0.34646858,0.67812078){\makebox(0,0)[lt]{\lineheight{1.25}\smash{\begin{tabular}[t]{l}$(S^2)^+$\end{tabular}}}}%
    \put(0.35237325,0.17909077){\makebox(0,0)[lt]{\lineheight{1.25}\smash{\begin{tabular}[t]{l}$(S^2)^-$\end{tabular}}}}%
    \put(0,0){\includegraphics[width=\unitlength,page=2]{Sphere_triad.pdf}}%
  \end{picture}%
\endgroup%

\caption{The triad structure on the sphere}
\label{fig:sphere triad}
\end{figure}

If $X$ is a CW complex, we denote by $X_n$ its $n$-skeleton, for $n\geqslant 0$. A CW complex $X$ is \emph{canonical} if it has a unique 0-cell $*$, taken as its basepoint, the attaching map $S^1 \to X_1$ of every 2-cell is pointed, thus defining a triad $(X_2, X_2^-, X_2^+)$ induced by the triad structure on $D$, and every 3-cell is attached via a pointed triad map~$(S^2, (S^2)^-, (S^2)^+) \to (X_2, X_2^-, X_2^+)$. 

The \emph{fundamental 2-crossed module} of a canonical CW complex $X$, as defined by Faria Martins \cite{martins_fundamental_2011}, is the 2-crossed module of groups
$$
\begin{tikzcd}
\Omega_3(X) = (\pi_3(X, X_2^-, X_2^+, *) \arrow[r, "\delta"] & \pi_2(X_2^-, X_2^-\cap X_2^+, *) \arrow[r, "\partial"]& \pi_1(X_1, *), \ \omega)
\end{tikzcd}
$$
where the boundary maps $\delta$ and $\partial$ and the Peiffer lifting $\omega$ are defined analogously to the maps in Section \ref{Section_Fundamental 2-crossed module}. For instance, if $S^2$ denotes the 2-sphere with the triad structure described above, then we get
$$
\begin{tikzcd}
\Omega_3(S^2) = (\Z \arrow[r, "0"] & \Z \arrow[r, "0"] & 0, \ \omega),
\end{tikzcd}
$$
where the Peiffer lifting is defined by $\omega(n,m) = nm$ for all $n,m \in \Z$.  The following realization result is a direct consequence of the work of Faria Martins \cite{martins_fundamental_2011}:

\begin{lem}
If $\sigma$ is a free 2-crossed module of groups, then there exists a canonical CW complex $X$ whose fundamental 2-crossed module $\Omega_3(X)$ is isomorphic to $\sigma$. 
\end{lem}

For an arbitrary~2-crossed module of groups $\sigma$ that does not satisfy this freeness condition, we replace $\sigma$ up to weak equivalence using the cofibrant replacement functor.  Following the Quillen model structure on the category of 2-crossed modules of groups defined by Cabello and Garzón \cite{cabello1994quillen}, we associate to a 2-crossed module of groups $\sigma$ a cofibrant replacement $p \colon Q(\sigma) \ \tilde{\to}\ \sigma$, where~$p$ is a trivial fibration. A \emph{classifying space} of a~2-crossed module of groups $\sigma$ is a canonical CW complex~$B\sigma$ such that $\pi_n(B\sigma) = 0$ for all~$n>3$ together with an isomorphism $\Omega_3(B\sigma) \cong Q(\sigma)$.  The following is a consequence of the above discussion.

\begin{lem}
Every 2-crossed module of groups admits a classifying space, and every connected CW complex such that $\pi_n(X) = 0$ for all $n>3$ is homotopy equivalent to the classifying space of a 2-crossed module. 
\end{lem}

\subsection{Geometric realizations of a coloring}\label{geometric realization}
Let $K$ be a canonical simplicial complex and let $X$ be a canonical CW complex.  A \emph{bifiltered map} from $K$ to $X$ is a continuous map~$f\colon \lvert K\rvert \to X$ such that 
$$
f(K_2^-)\subset X_2^- \quad \quad f(K_2^+) \subset X_2^+, \quad \quad f(K_n) \subset X_n\  \text{for all $n\geqslant 0$}. 
$$
Any bifiltered map $K\to X$ clearly induces a morphism of 2-crossed modules $\Pi_3(K) \to \Omega_3(X)$.  Now fix $\sigma$ a~\mbox{2-crossed} module of groups.  The \emph{induced $\sigma$-coloring} of a bifiltered map $f$ from $K$ to a classifying space~$B\sigma$ is the morphism $\Pi_3(f) \colon \Pi_3(K) \to \sigma$ obtained by composing the induced morphism~$\Pi_3(K) \to \Omega_3(B\sigma)$ with the trivial fibration $\Omega_3(B\sigma)\  \tilde{\to}\  \sigma$. A \emph{geometric realization} of a $\sigma$-coloring $\phi$ of $K$ is a bifiltered map $B\phi\colon K\to B\sigma$ such that the induced $\sigma$-coloring $\Pi_3(f)$ is equal to $\phi$ in the homotopy category of 2-crossed modules of groups. 

\begin{lem}\label{Lemma Realizability} 
Every $\sigma$-coloring of $K$ admits a geometric realization, and any map from $K$ to the classifying space of $\sigma$ is homotopic to the geometric realization of a $\sigma$-coloring of $K$. 
\end{lem}
\begin{proof}
If $K$ is connected, it follows from \cite{martins_fundamental_2011} and the model structure on the category of 2-crossed modules of groups that if $\sigma$ is a 2-crossed module of groups and $x\in K_0$, then any morphism $\phi\colon \Pi_3(K, x) \to \sigma$ admits a bifiltered map $f\colon K \to B\sigma$ that induces $\phi$. If $K$ is not connected, we apply the above construction to every connected component of $K$.  

Now let $f$ be a map from $K$ to a classifying space $B\sigma$ of $\sigma$.  By the bifiltered approximation theorem in~\cite{martins_fundamental_2011},~$f$ is homotopic to a bifiltered map $\bar{f}$ that induces a morphism of 2-crossed modules
$$
\Omega_3(\bar{f})\colon \Pi_3(K,x) \to \Omega_3(B\sigma).
$$
By \cite[Lemma 3.16]{martins_fundamental_2011}, $f$ is homotopic to the geometric realization of $\Omega_3(\bar{f})$. 
\end{proof}

\subsection{Quadratic derivations}

In this section, we extend the notion of a quadratic derivation defined by Faria Martins \cite{martins_fundamental_2011} to the context of groupoids.  Let
$$
\begin{tikzcd}\sigma = (L\arrow[r, "\delta"] & E\arrow[r, "\partial"] & H, \ \omega)\end{tikzcd} \quad\text{and} \quad \begin{tikzcd}\sigma' = (L'\arrow[r, "\delta'"] & E'\arrow[r, "\partial'"] & H', \ \omega')\end{tikzcd}
$$
be 2-crossed modules, and let $\phi = (\phi_1, \phi_2, \phi_3)$ be a morphism of 2-crossed modules from $\sigma$ to $\sigma'$.  A \emph{quadratic $\phi$-derivation} is a triple of set maps $(r\colon H_0\to H', s\colon H\to E', t\colon E\to L')$ such that 
$$
r(x) \in H'(\phi(x)),\quad\quad s(h) \in E'(\phi(y)),\quad\quad t(e) \in L'(\phi(x)),
$$
for all $x,y\in H_0$, $h\in H(x,y)$, and $e\in E(x)$, satisfying the following conditions:
$$
s(hk) = \left(\prescript{\phi_1(k)^{-1}}{}{s(h)}\right)s(k),
$$
for all $x, y, z\in H_0$, $h\in H(x,y)$, and $k \in H(y,z)$,  
$$
t(ef) = \left(\phi_2(f)s(d_1 (f))\right)^{-1}\triangleright \left(\omega(\phi_2 (f), \prescript{\phi_1(d_1f)^{-1}}{}{s(d_1(e))^{-1}}) t(e)\right)t(f),
$$
for all $x\in H_0$, $e,f\in E(x)$, and
$$
t(\prescript{h}{}{e}) = \prescript{\phi_1(h)}{}{\left(\left(\left(s(h)s(d_1(e))^{-1}\right)\triangleright \omega\big(\phi_2(e)^{-1}, s(h)^{-1}\big)^{-1}\right) \omega\big(s(h), s(d_1(e))^{-1}\phi_2(e)^{-1}\big)\prescript{\partial (s(h))}{}{t(e)} \right)},
$$
for all $x,y \in H_0$, $h\in H(x,y)$, and $e\in E(y)$.

The \emph{action} of a quadratic $\phi$-derivation $(r,s,t)$ is the triple
$$
\phi\cdot (r,s,t) = (\phi_1' \colon H\to H', \phi_2'\colon E \to E', \phi_3'\colon L\to L'),
$$
such that 
$$
\phi'_1(h) = r(x)^{-1}\phi_1(h)\partial' (s(h))r(y), \quad \quad \phi'_2(e) = \prescript{r(x)^{-1}}{}{\big(}\phi_2(e) s(\partial(e))\delta'( t(e))\big), \quad \quad \phi'_3(l) = \prescript{r(x)^{-1}}{}{\big(} \phi_3(l) t(\delta(l))\big),
$$
for all $x, y \in H_0$, $h\in H(x,y)$, $e\in E(x)$,  and $l \in L(x)$.  A simple computation shows that $\phi\cdot (r,s,t)$ is a morphism of 2-crossed modules (see  \cite[Theorem 2.13]{martins_fundamental_2011}).  We say that two morphisms $\phi$ and $\phi'$ from $\sigma$ to~$\sigma'$ are \emph{homotopic} if there exists a quadratic $\phi$-derivation $(r,s,t)$ such that
$$
\phi\cdot (r,s,t) = \phi'.
$$
When $H$ is a free groupoid, this defines an equivalence relation (see \cite{gohla_pointed_2013} for the group case). 

\subsection{The gauge groupoid}\label{gauge-groupoid} Let $K$ be a canonical simplicial complex with $B_*$ a choice of ordered simplices,  let\begin{tikzcd}\sigma = (L\arrow[r, "\delta"] & E \arrow[r, "\partial"] & H, \ \omega)\end{tikzcd}be a 2-crossed module of groups, and let $\phi$ be a $\sigma$-coloring of $K$. It follows from the freeness of $\Pi_3(K)$ that a quadratic $\phi$-derivation $(r,s,t)$ is uniquely determined by the restrictions~\mbox{$(r\colon K_0\to H,  s_{\vert B_1}\colon B_1 \to E, t_{\vert B_2}\colon B_2\to L)$}. Moreover, any triple of set maps
$$
(r'\colon K_0\to H, s'\colon B_1\to E, t'\colon B_2\to L)
$$ defines a quadratic $\phi$-derivation.  

In order to define the composition of quadratic derivations, we introduce a map defined by Gohla and Faria Martins \cite{gohla_pointed_2013}. Let $s$ and $s'$ denote maps $\pi_1(K_1; K_0)\to E$ such that 
$$
s(hk) = \left(\prescript{\phi_1(k)^{-1}}{}{s(h)}\right)s(k)\quad \quad \text{and}\quad \quad s'(hk) = \left(\prescript{\partial (s(k))^{-1}\phi_1(k)^{-1}}{}{s'(h)}\right)s'(k)
$$
for all $x,y,z \in K_0$, $h\in \pi_1(K_1; K_0)(x,y)$, and $k\in \pi_1(K_1; K_0)(y,z)$.  The \emph{connecting map} of $s$ and $s'$ is the map 
$$
\omega^{(s,s')}\colon \pi_1(K_1; K_0)\to L
$$
defined by $\omega^{(s, s')}(1_x) = 1$, $\omega^{(s,s')}(b) = 1$ for all $b\in B_1$, and 
$$
\omega^{(s,s')}(hk) = \omega^{(s,s')}(k)\left(s'(k)^{-1}\triangleright\left( \prescript{\phi_1(k)^{-1}}{}{\omega\left(\prescript{\phi_1(k)}{}{s(k)^{-1}}, \delta\left( \omega^{(s,s')}(h)\right) s'(h)^{-1}\right)}\prescript{\partial (s(k))^{-1}\phi_1(k)^{-1}}{}{\omega^{(s,s')}(h)}  \right)\right)
$$
for all $x,y,z \in K_0$, $h\in \pi_1(K_1; K_0)(x,y)$, and $k\in \pi_1(K_1; K_0)(y,z)$. 

\begin{lem}\label{Lemma composition}
If $(r,s,t)$ is a quadratic $\phi$-derivation and $(r',s',t')$ is a quadratic $\phi\cdot (r,s,t)$-derivation, then
$$
(\phi\cdot (r,s,t))\cdot (r',s',t') = \phi \cdot (r'', s'', t'')
$$
where $(r'', s'', t'')$ is the unique quadratic $\phi$-derivation satisfying
$$
r''(x) = r(x)r'(x)\quad \quad s''(b_1) = s(b_1)\prescript{r(y)}{}{s'(b_1)}\quad \quad t''(b_2) = \omega^{(s, \prescript{r}{}{s'})}(\partial (b_2)) \left(\prescript{r(x)}{}{s'(\partial (b_2))^{-1}}\triangleright t(b_2)\right)\prescript{r(x)}{}{t'(b_2)}
$$
for all $x,y,z \in K_0$, $b_1\in B_1(x,y)$, $b_2 \in B_2(x)$, where $\prescript{r}{}{s'}$ denotes the map defined by $\prescript{r}{}{s'}(b) = \prescript{r(y)}{}{s'(b)}$ for all~$x,y\in K_0$, $b\in B_1(x,y)$, and 
$$
s'(hk) =\left( \prescript{\partial (s(k))^{-1}\phi_1(k)^{-1}r(y)}{}{s'(h)}\right)\prescript{r(z)}{}{s'(k)}
$$
for all $x,y, z\in K_0$, $h\colon x\to y$, and $k\colon y\to z$. 
\end{lem}
\begin{proof}
Let $b_1\in B_1(x,y)$. We have
\begin{align*}
(\phi\cdot (r,s,t) )\cdot (r',s',t') (b_1) &= r'(x)^{-1}r(x)^{-1}\phi_1(b_1)\partial (s(b_1)) r(y)\partial \big(\prescript{r(y)}{}{s'(b_1)}\big)r'(y)\\
&= \phi\cdot (r'', s'', t'')(b_1)
\end{align*}
Let $b_2 \in B_2(x)$. We have
\begin{align*}
(\phi\cdot (r,s,t))\cdot (r',s',t') (b_2) &= \prescript{r'(x)^{-1}}{}{\left(\prescript{r(x)^{-1}}{}{\left(\phi_2(b_2)s(d_1 (b_2))\delta (t(b_2))\right)}s'(d_1 (b_2))\delta (t'(b_2))\right) }\\
&= \phi \cdot (r'', s'', t'')(b_2)
\end{align*}
Let $b_3 \in \pi_3(K, K_2^-, K_2^+; x)$. We have
\begin{align*}
(\phi\cdot (r,s,t))\cdot (r',s',t'))(b_3) &= \prescript{r'(x)^{-1}}{}{\left(\prescript{r(x)^{-1}}{}{\big (}\phi_3(b_3) t(d_2 (b_3))\big ) t'(d_2 (b_3))\right)}\\
&= \phi \cdot (r'', s'', t'').\qedhere
\end{align*}
\end{proof}

Let $\mathcal{G}(K,\sigma)$ be the category with objects the $\sigma$-colorings of $K$ and with morphisms from a $\sigma$-coloring $\phi$ to a $\sigma$-coloring $\phi'$ given by the quadratic~$\phi$-derivations~$(r,s,t)$ such that
$$
\phi' = \phi \cdot (r,s,t).
$$
The composition of morphisms $(r,s,t) \colon \phi \to \phi'$ and $(r',s',t')\colon \phi' \to \phi''$ is the quadratic~$\phi$-derivation
$$
(r, s, t)\cdot (r',s',t') \colon \phi \to \phi''
$$
defined in Lemma \ref{Lemma composition}.

\begin{lem}
The category $\mathcal{G}(K,\sigma)$ is a groupoid. 
\end{lem}

We call $\mathcal{G}(K,\sigma)$ the \emph{gauge groupoid} of $K$ with target $\sigma$.  Two $\sigma$-colorings of $K$ are \emph{gauge equivalent} if they are in the same connected component of the gauge groupoid. 

\begin{proof}
The composition of morphisms is clearly unital. Let us check that it is associative.  Let
$$
(r,s,t) \colon \phi \to \phi', \quad (r',s',t')\colon \phi' \to \phi'', \quad (r'',s'',t'') \colon \phi'' \to \phi'''
$$
be quadratic derivations. We set 
$$
((r,s,t)\cdot (r',s',t')) \cdot (r'',s'',t'') = (r_1, s_1, t_1)\quad\text{and} \quad
(r,s,t)\cdot ((r',s',t')\cdot (r'',s'',t'')) = (r_2, s_2, t_2).
$$
Straightforward computations show that $r_1(x) = r_2 (x)$ for all $x\in K_0$ and that $s_1(b_1) = s_2(b_1)$ for all $b_1\in B_1$.  

For any morphism $\psi \colon \Pi_3(K)\to \sigma$, any quadratic $\psi$-derivation~\mbox{$(\rho, \sigma, \tau)$}, and any quadratic \mbox{$\psi \cdot (\rho, \sigma, \tau)$-deri}\-vation~\mbox{$(\rho', \sigma', \tau')$}, we denote by $\sigma\otimes \sigma'$ the map $\pi_1(K_1; K_0) \to E$ defined by $(\sigma\otimes \sigma')(b_1) = \sigma(b_1) \prescript{\rho(y)}{}{\sigma'(b_1)}$ for all $x,y\in K_0$ and~$b_1 \in B_1(x,y)$ satisfying
$$
(\sigma\otimes \sigma')(hk) = \left(\prescript{\psi(k)^{-1}}{}{(\sigma\otimes\sigma')(h)}\right)(\sigma\otimes \sigma')(k)
$$
for all $x,y,z\in K_0$, $h\colon x\to y$, and $k \colon y \to z$. Let $x\in K_0$ and $b_2 \in B_2(x)$. 

\begin{align*}
t_1(b_2) &= \omega^{(s\otimes s', \prescript{rr'}{}{s''})}(d_1 (b_2))\left(\prescript{r(x)r'(x)}{}{s''(d_1(b_2))}^{-1}\triangleright \omega^{(s, \prescript{r}{}{s'})}(d_1(b_2))\right)\\ 
& \quad \left( \left( \prescript{r(x)r'(x)}{}{s''(d_1(b_2))^{-1}}\prescript{r(x)}{}{s'(d_1(b_2))^{-1}}\right) \triangleright t(b_2)\right) \left(\prescript{r(x)r'(x)}{}{s''(d_1(b_2))^{-1}}\triangleright \prescript{r(x)}{}{t'(b_2)}\right)\prescript{r(x)r'(x)}{}{t''(b_2)},
\end{align*}
and 
\begin{align*}
t_2(b_2) &= \omega^{(s, \prescript{r}{}{(s'}\otimes s'') )}(d_1(b_2))\prescript{r(x)}{}{\omega^{(s', \prescript{r'}{}{s''})}(d_1(b_2))}\left(\left( \prescript{r(x)r'(x)}{}{s''(d_1(b_2))^{-1}}\prescript{r(x)}{}{s'(d_1(b_2))}\right) \triangleright t(b_2)\right)\\
& \quad \left( \prescript{r(x)r'(x)}{}{s''(d_1(b_2))^{-1}}\triangleright \prescript{r(x)}{}{t'(b_2)}\right)\prescript{r(x)r'(x)}{}{t''(b_2)}.
\end{align*}
Therefore, $t_1(b_2) = t_2(b_2)$ is equivalent to the equation
$$
\omega^{(s\otimes s', \prescript{rr'}{}{s''})}(d_1 (b_2))\left(\prescript{r(x)r'(x)}{}{s''(d_1(b_2))}^{-1}\triangleright \omega^{(s, \prescript{r}{}{s'})}(d_1(b_2))\right) = \omega^{(s, \prescript{r}{}{(s'}\otimes s''))}(d_1(b_2))\prescript{r(x)}{}{\omega^{(s', \prescript{r'}{}{s''})}(d_1(b_2))}.
$$
By \cite[Equation (89)]{gohla_pointed_2013}, it suffices to show that 
$$
\prescript{r(x)}{}{\omega^{(s', \prescript{r'}{}{s})}}(d_1(b_2)) = \omega^{(\prescript{r}{}{s'}, \prescript{rr'}{}{s''})}(d_1(b_2)).
$$
This equality can easily be shown by induction on the length of $d_1(b_2)$ written in the basis $B_1$.  

Let us show that all morphisms in the gauge groupoid are isomorphisms. Let $(r,s,t) \colon \phi \to \phi'$. A right inverse of $(r,s,t)$ is given by the $\phi'$-derivation $(\bar{r}, \bar{s}, \bar{t})$ with 
$$
\bar{r}(x) = r(x)^{-1},\quad \quad 
\bar{s}(b_1) = \prescript{r(y)^{-1}}{}{s(b_1)^{-1}},\quad \quad 
\bar{t}(b_2) =\left( \prescript{r^{-1}}{}{\bar{s}(d_1(b_2))^{-1}}\triangleright \prescript{r(x)^{-1}}{}{t(b_2)^{-1}}\right)\prescript{r(x)^{-1}}{}{\omega^{(s, s^{-1})}(d_1(b_2))}
$$
for all $x, y \in K_0$, $b_1 \in B_1(x,y)$, and $b_2\in B_2(x)$. Similarly, a left inverse of $(r,s,t)$ is given by the quadratic $\phi'$-derivation $(\bar{r}, \bar{s}, \tilde{t})$ with 
$$
\tilde{t}(b_2) = \prescript{r(x)^{-1}}{}{s(d_1(b_2))}\triangleright \big (\omega^{(\bar{s}, \prescript{r^{-1}}{}{s})}(d_1(b_2))\prescript{r(x)^{-1}}{}{t(b_2)^{-1}}\big ) 
$$
for all $x\in K_0$ and $b_2\in B_2(x)$. The associativity tells us that in fact $\bar{t} = \tilde{t}$. 
\end{proof}

\subsection{Proof of the Homotopy Classification Theorem}

Let $K$ be a canonical simplicial complex, let~$B_*$ be a choice of ordered simplices of $K$, and let $\sigma$ be a 2-crossed module of groups with classifying space $B\sigma$. Without loss of generality, we can suppose that $K$ is connected because 
$$
[\lvert K\rvert , B\sigma] = \prod_{i}\, [\lvert K_{(i)}\rvert , B\sigma] \quad \text{and} \quad [\Pi_3(K), \sigma] = \prod_i [\Pi_3(K_{(i)}), \sigma],
$$
where $K_{(i)}$ denotes the $i$-th connected component of $K$. 

Let $\phi$ and $\phi'$ be $\sigma$-colorings of $K$.  By definition, $\phi$ and $\phi'$ are gauge equivalent if and only if there exists a quadratic~$\phi$-derivation $(r,s,t)$ such that $\phi' = \phi\cdot (r,s,t)$.  We will say that $\phi$ and $\phi'$ are \emph{pointed gauge-equivalent} if there is a point $x\in K_0$ such that the restrictions $\phi_x \colon \Pi_3(K,x)\to \sigma$ and $\phi_x'\colon \Pi_3(K, x)\to \sigma$ are homotopic in the sense of~\cite[Definition 2.14]{martins_fundamental_2011}. Notice that $(r,s,t) = (1,s,t)\cdot (r,1,1)$, and therefore $\phi$ and  $\phi'$ are gauge equivalent if and only if there is a map $r\colon K_0\to H$ such that $\phi$ and $\phi'\cdot (r,1,1)$ are pointed gauge-equivalent.   

For any $h\in H$, and for any morphism $\psi \colon \Pi_3(K,x)\to \sigma$, we denote by $\psi\cdot h\colon  \Pi_3(K,x)\to \sigma$ the morphism defined by
$$
(\psi\cdot h)_1(a) = h^{-1}\psi_1(a) h\quad \quad (\psi\cdot h)_2(b) = \prescript{h^{-1}}{}{\psi_2(b)} \quad \quad (\psi\cdot h)_3(c) = \prescript{h^{-1}}{}{\psi_3(c)},
$$
for all $a\in \pi_1(K_1;x)$, $b\in \pi_2(K_2^-, K_2^-\cap K_2^+; x)$, and $c\in \pi_3(K,K_2^-, K_2^+; x)$.  This defines a right group action of~$H$ on the set of morphisms $\Pi_3(K,x) \to \sigma$.  Notice that for any $\sigma$-coloring $\phi$ of $K$,  for any map $r\colon K_0 \to H$, and for any $x\in K_0$, we have
$$
(\phi\cdot (r,1,1))_x = \phi_x \cdot r(x).
$$
Therefore, two $\sigma$-colorings $\phi$ and $\phi'$ of $K$ are gauge equivalent if and only if there exist $x\in K_0$ and $h\in H$ such that $\phi_x$ and $\phi'_x\cdot h$ are homotopic in the sense of \cite[Definition 2.14]{martins_fundamental_2011}.  By \cite[Lemma 3.12]{martins_fundamental_2011}, we know that $\phi_x$ and $\phi'_x\cdot h$ admit geometric realizations, denoted by $B(\phi_x)$ and $B(\phi'_x\cdot h)$ respectively. By \cite[Lemma~3.16]{martins_fundamental_2011}, we deduce that $\phi$ and $\phi'$ are gauge equivalent if and only if there is an element $h\in H$ such that~$B(\phi_x)$ and $B(\phi'_x\cdot h)$ are pointed homotopic maps. 

Let $[\lvert K\rvert ,B\sigma]_*$ denote the set of pointed homotopy classes of maps from $(\lvert K\rvert , x)$ to $(B\sigma, *)$.  The fundamental group of $B\sigma$ acts on $[\lvert K\rvert,B\sigma]_*$ and
$$
[\lvert K\rvert ,B\sigma] \cong [\lvert K\rvert , B\sigma]_*/\pi_1(B\sigma; *).
$$
Moreover, recall that $\pi_1(B\sigma, *) \cong H/\partial(E)$. A careful study of the construction of a geometric realization of a morphism $\phi'_x \colon \Pi_3(K,x)\to \sigma$ (see \cite[Proposition 9]{ellis_crossed_1993}) shows that for any $h\in H$,  $B(\phi'_x\cdot h)$ and $B(\phi'_x)\cdot [h]$ are pointed homotopic.  We deduce that two $\sigma$-colorings $\phi$ and $\phi'$ of $K$ are gauge equivalent if and only if there exists $h\in H$ such that $B\phi'\cdot h$ and $B\phi$ are pointed homotopic,  where $B\phi$ and $B\phi'$ denote their geometric realizations. We conclude that $\phi$ and $\phi'$ are gauge equivalent if and only if $B\phi'$ and $B\phi$ are homotopic. 

\section{Graded monoidal 2-categories}\label{Section. Graded monoidal 2-categories}

\noindent The notion of a monoidal category graded by a 2-group, or equivalently by a crossed module, was introduced by Sözer and Virelizier \cite{sozer_monoidal_2023, sozer_hopf_2024} in order to construct 3-dimensional homotopy quantum field theories with target a homotopy 2-type. We categorify their construction: we define the notion of a monoidal 2-category graded by a 3-group, or equivalently by a 2-crossed module. Recall that $\mathbbm{k}$ is a nonzero commutative ring. 

\subsection{Linear categories}For any objects $X$, $Y$ in a category $C$, we denote by~$\mathrm{Hom}_C(X,Y)$ the set of morphisms from $X$ to $Y$, and we denote by $1_X$ the identity of $X$.  A category $C$ is \emph{$\mathbbm{k}$-linear} (or simply \emph{linear}) if its morphism sets carry the structure of $\mathbbm{k}$-modules such that the composition of morphisms is bilinear.  A \emph{linear functor}~$F\colon C \to D$ between $\mathbbm{k}$-linear categories is a functor such that for all objects $X$, $Y$ of $C$, the induced map
$$
\mathrm{Hom}_C(X,Y) \to \mathrm{Hom}_D(F(X), F(Y))
$$
is a morphism of $\mathbbm{k}$-modules.  A $\mathbbm{k}$-linear category $C$ is \emph{$\mathbbm{k}$-additive} (or simply \emph{additive}) if any finite (possibly empty) family of objects has a direct sum in $C$.  A \emph{zero object} in a linear category $C$ is an object $0$ that is both initial and terminal.  Any additive category has a zero object that is unique up to isomorphism.  An object $X$ in an additive category $C$ is a zero object if and only if $\mathrm{Hom}_C(X, X)$ is the zero $\mathbbm{k}$-module.  

\subsection{Categories graded by a set}\label{Section_Set-graded}

Let $E$ denote a set. An \emph{$E$-graded category} (over $\mathbbm{k}$) is a $\mathbbm{k}$-additive category~$C$ endowed with a family $\{C_e\}_{e\in E}$ of full $\mathbbm{k}$-additive subcategories such that every object of $C$ is a direct sum of a finite family of objects in~$\bigcup_{e\in E} C_e$, and 
$$
\mathrm{Hom}_C(X,Y) = \{0\}
$$
for all $X\in C_e$ and $Y\in C_f$ such that $e\neq f\in E$. In that case, we write 
$$
C = \bigoplus_{e\in E}C_e.
$$
An object $X$ of $C$ is \emph{homogeneous of degree $e$} if $X\in C_e$ for some $e\in E$.  When $X$ is nonzero, such an $e$ is uniquely determined, denoted $\vert X \rvert$, and is called the \emph{degree} of $X$.  If $E$ is a group, the category $C_1$ is called the \emph{neutral component} of $C$. 

\subsection{Hom-graded categories}\label{Hom-graded categories}
In this section, we recall some definitions related to Hom-graded categories and we refer the reader to \cite{sozer_monoidal_2023} for more details.  Let $L$ be a group. An \emph{$L$-Hom-graded category} (over~$\mathbbm{k}$) is a category enriched over the monoidal category of $L$-graded $\mathbbm{k}$-modules and $\mathbbm{k}$-linear grading-preserving homomorphisms. In particular, in such a category $C$, the set of morphisms between two objects $X$,$Y$ is an $L$-graded $\mathbbm{k}$-module:
$$
\mathrm{Hom}_C(X,Y) = \bigoplus_{l\in L}\mathrm{Hom}_C^l(X,Y).
$$
The morphisms in $\mathrm{Hom}_C^l(X,Y)$ are \emph{homogeneous of degree $l$}.  The \emph{$1$-subcategory} of an $L$-Hom-graded category~$C$ is the category $C^1$ with the same objects as $C$ and with morphisms
$$
\mathrm{Hom}_{C^1}(X,Y) = \mathrm{Hom}_C^1(X,Y),
$$
for all objects $X$ and $Y$. 

Let $C$ be an $L$-Hom-graded category, and let $l\in L$. An object $Z$ of $C$ is an \emph{$l$-direct sum} of a finite family~$\{X_i\}_{i\in I}$ of objects of $C$ if there is a family $(p_i\colon Z \to X_i, q_i\colon X_i \to Z)_{i\in I}$ of morphisms such that $p_i$ is homogeneous of degree $l^{-1}$, $q_i$ is homogeneous of degree $l$,
$$
1_Z = \sum_{i\in I} q_i\circ p_i, \quad \text{and} \quad  1_{X_i} = p_i \circ q_i,
$$
for all $i\in I$, and $p_i \circ q_j = 0$ for all $i\neq j \in I$.  We denote such an $l$-direct sum by
$$
Z = \bigoplus_{i\in I}^l X_i.
$$
Such an $l$-direct sum, if it exists, is unique up to a homogeneous isomorphism of degree 1.  An $L$-Hom-graded category is \emph{$L$-additive} if any finite (possibly empty) family of objects has an $l$-direct sum for all $l\in L$. 

\subsection{Linear 2-categories} A 2-category is \emph{$\mathbbm{k}$-linear} (or simply \emph{linear}) if its 2-morphism sets are \mbox{$\mathbbm{k}$-mod}ules such that horizontal and vertical composition of 2-morphisms is bilinear. A \emph{linear 2-functor} between linear 2-categories is a 2-functor that is linear at the level of Hom-categories and such that the horizontal composition is bilinear.  

A \emph{linear monoidal 2-category} is a monoidal 2-category whose underlying 2-category is linear such that the left and right tensor products are linear. 

\subsection{Direct sums in linear 2-categories} An object $Z$ in a linear 2-category $\mathcal{C}$ is a \emph{direct sum} of a finite family $\{X_i\}_{i\in I}$ of objects of~$\mathcal{C}$ if there are \emph{projection} 1-morphisms $\{p_i \colon Z \to X_i\}_{i\in I}$, \emph{inclusion} \mbox{1-morphisms}~\mbox{$\{q_i \colon X_i\to Z\}_{i\in I}$}, and 2-iso\-morphisms~$\{\varphi\colon p_i\circ q_i\ \tilde{\Rightarrow}\ 1_{X_i}\}_{i\in I}$ such that $p_i\circ q_j$ is a zero 1-morphism for all $i\neq j \in I$ and the direct sum~$\bigoplus_{i\in I}(q_i \circ p_i)$ exists and is isomorphic to $1_Z$ in $\mathcal{C}(Z,Z)$. 

\subsection{Zero objects and zero 1-morphisms}

Let $\mathcal{C}$ be a linear 2-category. A \emph{zero 1-morphism} in $\mathcal{C}$ is a 1-morphism~\mbox{$0\colon X \to Y$} that is a zero object in the linear category $\mathcal{C}(X,Y)$.  A \emph{zero object} in $\mathcal{C}$ is an object~$0$ such that for every object $Y$ of $\mathcal{C}$, the categories $\mathcal{C}(0,Y)$ and $\mathcal{C}(Y,0)$ are terminal. Notice that an object $X$ of~$\mathcal{C}$ is a zero object if and only if $1_X$ is a zero 1-morphism.  

\subsection{Hom-graded 2-categories}
Let $\delta \colon L \to E$ be a crossed module of groups with the action of $E$ on $L$ denoted by $\triangleright$.  A \emph{$\delta$-Hom-graded 2-category} (over $\mathbbm{k}$) is a $\mathbbm{k}$-linear 2-category $\mathcal{C}$ such that the category $\mathcal{C}(X,Y)$ is~$L$-Hom-graded for all objects $X$ and $Y$,  the 1-subcategory $\mathcal{C}(X,Y)^1$ of $\mathcal{C}(X,Y)$ is $E$-graded for all objects $X$ and $Y$, and $\mathcal{C}$ satisfies the following conditions, where a 1-morphism $f\colon X\to Y$ is said to be \emph{homogeneous} if it is homogeneous in $\mathcal{C}(X,Y)^1$:
\begin{itemize}
\item for all homogeneous 1-morphisms $f, g\colon X\to Y$ and for all $l\in L$ such that $\delta(l)\lvert f \rvert \neq \lvert g \rvert$,
$$
\mathrm{Hom}_{\mathcal{C}(X,Y)}^l(f,g) = \{0\},
$$
\item  if $f\colon X \to Y$ and $g\colon Y\to Z$ are homogeneous 1-morphisms, then the composite $g\circ f$ is homogeneous of degree $\lvert g\circ f\rvert =  \lvert g \rvert \lvert f\rvert$,
\item all nonzero identity 1-morphisms are homogeneous of degree 1,
\item the horizontal composition of 2-morphisms is grading preserving: $$
    \circ \colon \mathrm{Hom}_{\mathcal{C}(Y,Z)}^l(h,i) \times \mathrm{Hom}_{\mathcal{C}(X,Y)}^k(f,g) \to \mathrm{Hom}_{\mathcal{C}(X,Z)}^{l(\lvert f \rvert \triangleright k)}(h\circ f, i\circ g)
$$
for all objects $X, Y,Z$, for all $f,g\colon X\to Y$, $h,i\colon Y\to Z$, and $l,k\in L$, whenever $f$ is homogeneous.
\end{itemize}

Notice that a one-object $\delta$-Hom-graded 2-category is a $\delta$-graded monoidal category in the sense of Sözer and Virelizier \cite{sozer_monoidal_2023}.

\subsection{Equivalences in Hom-graded 2-categories}

Let $\delta \colon L \to E$ be a crossed module of groups with the action of $E$ on $L$ denoted by $\triangleright$, and let $(l,e) \in L\times E$. An~\emph{$(l,e)$-equivalence} in a $\delta$-Hom-graded 2-category~$\mathcal{C}$ is an equivalence $(f,g,\eta, \varepsilon)$ such that:
\begin{itemize}
\item $f\colon X\to Y$ is a homogeneous 1-morphism of degree $e$,
\item $g\colon Y\to X$ is a homogeneous 1-morphism of degree $\delta(l)e^{-1}$,
\item $\eta\colon 1_X\Rightarrow g\circ f$ is a homogeneous 2-isomorphism of degree $l$,
\item $\varepsilon\colon f\circ g \Rightarrow 1_Y$ is a homogeneous 2-isomorphism of degree $e\triangleright l^{-1}$.
\end{itemize}

\subsection{Direct sums in Hom-graded 2-categories}\label{Direct sums}
Let $\delta \colon L\to E$ be a crossed module of groups with the action of $E$ on $L$ denoted by $\triangleright$, and let~$(l,e) \in L\times E$. An object $Z$ of a $\delta$-Hom-graded 2-category $\mathcal{C}$ is an~\emph{$(l,e)$-direct sum} of a finite family of objects~$\{X_i\}_{i\in I}$ if there is a family $(p_i, q_i, \varphi_i)_{i\in I}$ such that for all $i,j\in I$:
\begin{itemize}
\item  $p_i\colon Z\to X_i$ is a homogeneous 1-morphism of degree $e^{-1}\delta (l)^{-1}$,
\item $q_i\colon X_i\to Z$ is a homogeneous 1-morphism of degree $e$, 
\item $\varphi_i\colon p_i\circ q_i\Rightarrow 1_{X_i}$ is a homogeneous 2-isomorphism of degree $e^{-1}\triangleright l$,
\item $p_i\circ q_j$ is a zero 1-morphism when $i\neq j$, and
\end{itemize} 
$$
1_Z = \bigoplus_{i\in I}^{l}(q_i \circ p_i),
$$
in the $L$-Hom-graded category $\mathcal{C}(Z,Z)$. 
Such an $(l,e)$-direct sum, if it exists, is denoted by
\begin{equation*}
Z = \boxplus_{i\in I}^{(l,e)}X_i.
\end{equation*}
If in addition $\mathcal{C}$ is locally $L$-additive (meaning that each Hom-category $\mathcal{C}(X,Y)$ is $L$-additive, see Section \ref{Hom-graded categories}) then such an $(l,e)$-direct sum is unique up to~$(1,1)$-equivalence. 

\subsection{Subcategories of Hom-graded 2-categories}\label{Subsection_1 and (1,1)-subcategories}Let $\delta\colon L \to E$ be a crossed module of groups. The \emph{\mbox{$1$-sub}\-category} of a $\delta$-Hom-graded 2-category~$\mathcal{C}$ is the 2-category $\mathcal{C}^1$ whose objects are the objects of $\mathcal{C}$ and such that~\mbox{$\mathcal{C}^1(X,Y) = \mathcal{C}(X,Y)^1$} is the 1-subcategory of $\mathcal{C}(X,Y)$ for all objects $X$ and $Y$.

The \emph{$(1,1)$-subcategory} of a $\delta$-Hom-graded 2-category~$\mathcal{C}$ is the 2-category $\mathcal{C}^1_1$ whose objects are the objects of~$\mathcal{C}$ and such that $\mathcal{C}^1_1(X,Y) = \mathcal{C}^1(X,Y)_1$ is the neutral component of the $E$-graded category $\mathcal{C}^1(X,Y)$ for all objects $X$ and $Y$. 

\subsection{Linear monoidal 2-categories}
We work in the setting of \emph{semistrict monoidal 2-categories} (or simply \emph{monoidal \mbox{2-cate}gories}), meaning that the underlying 2-category is strict and the monoidal associators and unitors are strict, but there may be a nontrivial interchange 2-isomorphism between two ways of taking the monoidal product of 1-morphisms. Such categories appear in the work of Gordon-Power-Street \cite{gordon_coherence_1995} and are sometimes called ``Gray monoids''.  Recall that such monoidal 2-categories consist of a 2-category $\mathcal{C}$, an identity object $\mathbbm{1}$, left and right tensor 2-functors $X\otimes -$, $- \otimes X$ for every object $X$, and interchange 2-isomorphisms
$$
\phi_{f,g} \colon (f\otimes Y') \circ (X \otimes g) \ \tilde{\Rightarrow} \ (X'\otimes g) \circ (f\otimes Y),
$$
for all 1-morphisms $f\colon X \to X'$, $g\colon Y \to Y'$, subject to certain compatibility conditions.  We refer the reader to the presentation given by Douglas--Reutter \cite{douglas_fusion_2018}.  A \emph{$\mathbbm{k}$-linear monoidal 2-category} is a (semistrict) monoidal 2-category whose underlying 2-category is $\mathbbm{k}$-linear such that for any object $X$ the left and right tensor 2-functors are linear.

Any weakly monoidal (weak) bicategory may be strictified to obtain a monoidal 2-category, see \cite{gordon_coherence_1995}. Similarly, $\mathbbm{k}$-linear weakly monoidal bicategories can be strictified to obtain $\mathbbm{k}$-linear monoidal \mbox{2-cate}\-gories. Therefore, we permit ourselves to work in the semistrict setting without loss of generality. 

\subsection{Monoidal 2-categories graded by a group}\label{Section_group-graded}
Let $\mathcal{C}$ be a $\mathbbm{k}$-linear monoidal 2-category. A \emph{full subcategory} of $\mathcal{C}$ is a $\mathbbm{k}$-linear~2-category $\mathcal{D}$ such that every object of $\mathcal{D}$ is an object of $\mathcal{C}$ and $\mathcal{D}(X,Y) = \mathcal{C}(X,Y)$ for all objects $X$ and~$Y$ of~$\mathcal{D}$, where the compositions in $\mathcal{D}$ are induced by the compositions in $\mathcal{C}$.

Let $H$ be a group. An \emph{$H$-graded 2-category} (over $\mathbbm{k}$) is a $\mathbbm{k}$-linear monoidal 2-category $\mathcal{C}$ endowed with a family $\{\prescript{}{h}{\mathcal{C}}\}_{h\in H}$ of full subcategories such that 
\begin{itemize}
\item every object of $\mathcal{C}$ is a direct sum of a finite family of objects in $\bigcup_{h\in H}\prescript{}{h}{\mathcal{C}}$, 
\item if $X\in \prescript{}{h}{\mathcal{C}}$ and $Y\in \prescript{}{k}{\mathcal{C}}$ with $h\neq k\in H$, then $\mathcal{C}(X,Y)$ is the terminal category,
\item if $X\in \prescript{}{h}{\mathcal{C}}$ and $Y\in \prescript{}{k}{\mathcal{C}}$, then $X\otimes Y \in \prescript{}{hk}{\mathcal{C}}$,
\item $\mathbbm{1} \in \prescript{}{1}{\mathcal{C}}$.
\end{itemize}
The monoidal 2-category $\prescript{}{1}{\mathcal{C}}$ is called the \emph{neutral component} of $\mathcal{C}$. We say that an object $X\in \mathcal{C}$ is \emph{homogeneous} if $X$ is nonzero and $X\in \prescript{}{h}{\mathcal{C}}$ for some $h\in H$. Such an $h$ is uniquely determined by $X$, denoted $\lvert X \rvert$, and is called the \emph{degree} of $X$. 

\subsection{Monoidal 2-categories graded by a 2-crossed module} Let
$$
\begin{tikzcd}
\sigma = (L \arrow{r}{\delta} & E \arrow{r}{\partial} & H,\ \omega)
\end{tikzcd}
$$
be a~2-crossed module of groups.  A \emph{$\sigma$-graded monoidal 2-category} (or simply a \emph{$\sigma$-graded 2-category}) is a $\mathbbm{k}$-linear monoidal 2-category $\mathcal{C}$ that is~\mbox{$\delta$-Hom-graded}, whose $(1,1)$-subcate\-gory~$\mathcal{C}^1_1$ is $H$-graded, and that satisfies the following conditions, where an object $X$ of $\mathcal{C}$ is said to be \emph{homogeneous} if it is homogeneous in $\mathcal{C}_1^1$: 
\begin{itemize}
    \item if $X$ is a homogeneous object and $f$ is a homogeneous 1-morphism, then $X\otimes f$ is homogeneous of degree $\prescript{\lvert X\rvert }{}{\lvert f \rvert}$, and $f\otimes X$ is homogeneous of degree $\lvert f \rvert$;
    \item if $X$ is a homogeneous object and $\alpha$ is a homogeneous 2-morphism,  then $ X\otimes \alpha$ is homogeneous of degree~$\prescript{\lvert X \rvert}{}{\lvert \alpha\rvert}$, and $ \alpha\otimes X$ is homogeneous of degree $\lvert \alpha\rvert$;
    \item if $f\colon X\to X'$ and $g\colon Y\to Y'$ are homogeneous 1-morphisms and $X$ is a homogeneous object, then the interchanger 2-isomorphism
    \begin{equation*}
    \phi_{f,g} \colon (f\otimes Y')\circ (X\otimes g) \Rightarrow (X'\otimes g)\circ (f\otimes Y)
    \end{equation*}
    is homogeneous of degree $\lvert \phi_{f,g}\rvert = \omega(\lvert f \rvert, \prescript{\lvert X\rvert }{}{\lvert g \rvert })^{-1}$;
    \item if $f\colon X\to Y$ is a homogeneous 1-morphism, then $\partial(\lvert f \rvert ) \lvert X \rvert = \lvert Y \rvert$ whenever $X$ and $Y$ are homogeneous objects. 
\end{itemize}

\subsection{Subcategories of graded 2-categories}\label{Section-neutral component} Let $\sigma$ be a 2-crossed module of groups.  The \emph{neutral component} of a $\sigma$-graded 2-category $\mathcal{C}$ is the $\mathbbm{k}$-linear monoidal 2-category $\prescript{}{1}{\mathcal{C}}_1^1$ whose objects are the objects of the neutral component of $\mathcal{C}^1_1$ and their $(1,1)$-direct sums (see Section \ref{Subsection_1 and (1,1)-subcategories}), and such that $\prescript{}{1}{\mathcal{C}_1^1}(X,Y) = \mathcal{C}_1^1(X,Y)$ for all objects~$X$ and~$Y$ of~$\prescript{}{1}{\mathcal{C}_1^1}$. 

\subsection{Spherical 2-categories}\label{Subsection_Spherical 2-categories}A \emph{planar pivotal 2-category} is a $\mathbbm{k}$-linear 2-category  together with the functorial choice, for every 1-morphism $f\colon X \to Y$,  of a right adjoint $f^*\colon Y\to X$, of a \emph{unit} $\eta_f \colon 1_X\Rightarrow f^* \circ f$,  and of a \emph{counit}~$\varepsilon_f \colon f \circ f^*\Rightarrow 1_Y$,  such that the adjoint is involutive $f^{**} = f$ and the left and right mates of all 2-morphisms agree, see \cite[Definition 2.2.1]{douglas_fusion_2018}.  A \emph{monoidal planar pivotal 2-category} is a $\mathbbm{k}$-linear monoidal~\mbox{2-cate}\-gory equipped with a planar pivotal structure that is compatible with the tensor product, see~\cite[Definition~2.2.3]{douglas_fusion_2018}. 

A \emph{pivotal 2-category} is a monoidal planar pivotal 2-category together with the choice, for every object $X$, of a right dual $X^\#$, of \emph{folds} $i_X \colon \mathbbm{1}\to X^\#\otimes X$ and $e_X \colon X \otimes X^\#\to \mathbbm{1}$, and of \emph{cusps}
$$
C_X \colon (e_X \otimes X) \circ (X \otimes i_X) \Rightarrow 1_X \quad \text{and} \quad D_X \colon 1_{X^\#} \Rightarrow (X^\#\otimes e_X) \circ (i_X\otimes X^\#),
$$
satisfying compatibility conditions, see \cite[Definition 2.2.4]{douglas_fusion_2018}.

Let $\mathcal{D}$ be a pivotal 2-category, let $f\colon X \to Y$ be a 1-morphism in $\mathcal{D}$, and let~$\alpha\colon f \Rightarrow f$ be a 2-endomorphism.  The \emph{back 2-spherical trace} of $\alpha$ is the 2-endomorphism $\mathrm{Tr}_B(\alpha) \colon 1_{\mathbbm{1}} \Rightarrow 1_\mathbbm{1}$ defined by 
$$
\mathrm{Tr}_B(\alpha) = \varepsilon_{e_Y} \cdot \left( 1_{e_Y} \circ \left( \left(\varepsilon_f \cdot \left( \alpha \circ 1_{f^*} \right)\cdot \eta_{f^*} \right)\otimes Y^{\#} \right) \circ 1_{i_{Y^\#}} \right) \cdot \eta_{i_{Y^\#}}.
$$
The \emph{front 2-spherical trace} of $\alpha$ is the 2-endomorphism $\mathrm{Tr}_F(\alpha) \colon 1_{\mathbbm{1}} \Rightarrow 1_\mathbbm{1}$ defined by 
$$
\mathrm{Tr}_F(\alpha) = \varepsilon_{e_{Y^\#}} \cdot \left( 1_{e_{Y^\#}} \circ \left(Y^\#\otimes  \left(\varepsilon_f \cdot \left( \alpha \circ 1_{f^*} \right)\cdot \eta_{f^*} \right) \right) \circ 1_{i_{Y}} \right) \cdot \eta_{i_{Y}}.
$$
A \emph{spherical 2-category} is a pivotal 2-category such that the front 2-spherical trace and the back 2-spherical trace of any 2-endomorphism $\alpha$ are equal, then denoted $\mathrm{Tr}(\alpha)$ and called the \emph{trace} of $\alpha$. 

\subsection{Spherical graded 2-categories}\label{Subsection_Spherical graded 2-categories}
Let
$$
\begin{tikzcd}
\sigma = (L \arrow{r}{\delta} & E \arrow{r}{\partial} & H,\ \omega)
\end{tikzcd}
$$
be a~2-crossed module of groups.  A \emph{pivotal $\sigma$-graded 2-category} is a pivotal 2-category that is $\sigma$-graded, such that 
\begin{itemize}
\item the dual $X^\#$ of a homogeneous object $X$ is homogeneous of degree $\lvert X^\#\rvert = \lvert X \rvert^{-1}$,
\item the adjoint $f^*$ of a homogeneous 1-morphism $f$ is homogeneous of degree $\lvert f^*\rvert = \lvert f \rvert^{-1}$,
\item the units, counits, folds,  and cusps are all homogeneous of degree $1$.
\end{itemize}
Notice that the 1-subcategory, the $(1,1)$-subcategory, and the neutral component of a pivotal~$\sigma$-graded \mbox{2-category} are pivotal 2-categories (see Sections \ref{Subsection_1 and (1,1)-subcategories} and \ref{Section-neutral component}). 

A \emph{spherical} $\sigma$-graded 2-category is a pivotal~$\sigma$-graded 2-category whose 1-subcategory is spherical. In other words, in a spherical $\sigma$-graded 2-category,  the front 2-spherical trace and the back 2-spherical trace of any degree-1 homogeneous 2-endomorphism $\alpha$ are equal, then denoted $\mathrm{Tr}(\alpha)$ and called the \emph{trace} of $\alpha$.  

\subsection{Example}\label{Ex: graded 2-category associated to a 2-crossed module}
Any 2-crossed module of groups
$$
\begin{tikzcd}\sigma = (L\arrow[r, "\delta"] & E\arrow[r, "\partial"] & H, \ \omega)\end{tikzcd}
$$
induces a $\mathbbm{k}$-linear monoidal 2-category $\widehat{\mathcal{G}}_\sigma$ with set of objects the group $H$, with 1-morphisms
$$
\widehat{\mathcal{G}}_\sigma(h,k) = \{e\in E \mid \partial (e) = kh^{-1}\}
$$
for any objects $h$ and $k$, and with 2-morphisms
$$
\mathrm{Hom}_{\widehat{\mathcal{G}}_\sigma(h,k)}(e,f) = \mathbbm{k}\{l \in L \mid \delta( l) = fe^{-1}\}
$$
for any objects $h, k$ and for all $e,f \colon h \to k$.  The tensor product of two objects of $\widehat{\mathcal{G}}_\sigma$ is given by their product in $H$, the composition of 1-morphisms $e\colon h\to k$ and $f\colon k \to l$ is given by their product $fe\colon h\to l$ in $E$, and the (vertical) composition of 2-morphisms $l \colon e\Rightarrow f$ and $m\colon f\Rightarrow g$ is given by their product $ml\colon e\Rightarrow g$ in~$L$.  Let $\mathcal{G}_\sigma$ be the locally additive completion of $\widehat{\mathcal{G}}_\sigma$.  In other words,~$\mathcal{G}_\sigma$ has the same objects as $\widehat{\mathcal{G}}_\sigma$, the 1-morphisms of $\mathcal{G}_\sigma$ from an object $h$ to an object $k$ are given by finite~(possibly empty) sequences of 1-morphisms in $\widehat{\mathcal{G}}_\sigma(h,k)$, and the 2-morphisms of $\mathcal{G}_\sigma$ are matrices of~2-morphisms of $\widehat{\mathcal{G}}_\sigma$. We call $\mathcal{G}_\sigma$ the \emph{linearization} of $\sigma$. The linearization of $\sigma$ is a spherical $\sigma$-graded 2-category, where the dual of an object $h\in H$ is the object $h^{-1}$ and the adjoint of a 1-morphism $e\in E$ is the 1-morphism~$e^{-1}$.

\subsection{Example}\label{pushforward}Let
$$
\begin{tikzcd}\sigma = (L\arrow[r, "\delta"] & E \arrow[r, "\partial"] & H, \ \omega)\end{tikzcd} \quad\text{and} \quad \begin{tikzcd}\sigma' = (L'\arrow[r, "\delta'"] & E' \arrow[r, "\partial'"] & H', \ \omega')\end{tikzcd}
$$
be 2-crossed modules of groups,  let $\mathcal{C}$ be a $\sigma$-graded 2-category, and let
$$
\phi = (\phi_1\colon H \to H', \phi_2\colon E\to E', \phi_3\colon L \to L')
$$
be a morphism of 2-crossed modules from $\sigma$ to $\sigma'$. The \emph{pushforward} of $\mathcal{C}$ along $\phi$ is the $\sigma'$-graded 2-category~$\phi_*(\mathcal{C})$ with the same underlying linear monoidal~2-category as $\mathcal{C}$ and with $\sigma'$-grading defined as follows. For all $l'\in L'$ and for all 1-morphisms~$f, g\colon X\to Y$:
$$
\mathrm{Hom}_{\phi_*(\mathcal{C})(X,Y)}^{l'}(f,g) =\!\!\! \bigoplus_{l\in \phi_3^{-1}(l')}\!\! \!\mathrm{Hom}_{\mathcal{C}(X,Y)}^l(f,g).
$$ 
For all $e'\in E'$, the degree-$e'$ homogeneous 1-morphisms in $\phi_*(\mathcal{C})$ are the homogeneous 1-morphisms in $\mathcal{C}$ of degree $e\in \phi_2^{-1}(e')$ and their 1-direct sums. For all $h'\in H'$, the degree-$h'$ homogeneous objects in $\phi_*(\mathcal{C})$ are the homogeneous objects in $\mathcal{C}$ of degree $h\in \phi_1^{-1}(h')$. Notice that if $\mathcal{C}$ is spherical, then so is $\phi_*(\mathcal{C})$.

\subsection{Example}For any crossed module $\chi\colon E\to H$, Gainutdinov, Runkel, and Wang~\cite{gainutdinov_constructions_2026} define a notion of~$\chi$-crossed braided monoidal category. A $\chi$-crossed braided monoidal category is an $E$-graded monoidal category~$C = \bigoplus_{e\in E}C_e$ together with an $H$-action on $C$ such that $h\in H$ maps $C_e$ to $C_{(\prescript{h}{}{e})}$, and a $\chi$-crossed braiding
$$
\{c_{U,V} \colon U\otimes V \to (\chi(\lvert U\rvert)\cdot V)\otimes U\}_{U,V \in C},
$$
satisfying compatibility conditions. To any $\chi$-crossed braided category $C$ we associate a $\sigma$-graded 2-category $\mathcal{C}$, with
$$
\begin{tikzcd}
\sigma = (1\arrow[r] & E\arrow[r, "\chi"] & H, \ 1).
\end{tikzcd}
$$
The objects of $\mathcal{C}$ are the elements of $H$, and 
$$
\mathcal{C}(h,k) =\!\!\!\!\! \bigoplus_{e \in \chi^{-1}(kh^{-1})}\!\!\!\!\! C_e
$$
for all $h,k\in H$. The composition of 1-morphisms in $\mathcal{C}$ is given by their tensor product in $C$, and the interchange~2-isomorphism in $\mathcal{C}$ is given by the braiding in $C$.

\section{Graded-fusion 2-categories}\label{Section. Graded-fusion 2-categories}

\noindent Fusion 2-categories were introduced by Douglas and Reutter \cite{douglas_fusion_2018} as a categorification of the notion of a fusion category.  In this section, we introduce graded-fusion 2-categories as a categorification of the notion of a graded-fusion category defined by Sözer and Virelizier \cite{sozer_monoidal_2023}.  Throughout this section, 
\begin{center}
\begin{tikzcd}\sigma =(L \arrow[r, "\delta"] & E \arrow[r, "\partial"] & H, \ \omega)\end{tikzcd}
\end{center}
is a 2-crossed module of groups. Recall that $\mathbbm{k}$ denotes a nonzero commutative ring.  

\subsection{Semisimple categories}

An object $X$ in a $\mathbbm{k}$-linear category $D$ is \emph{simple} if $\mathrm{Hom}_D(X,X)$ is a free \mbox{$\mathbbm{k}$-mod}ule of rank 1 (with basis~$\{1_X\}$).  A $\mathbbm{k}$-linear category $D$ is \emph{semisimple} if every object of $D$ is a direct sum of a finite family of simple objects, and for any non-isomorphic simple objects $X$ and $Y$ of $D$,  $\mathrm{Hom}_D(X,Y) = \{0\}$.  

\subsection{Graded-semisimple categories}
Let $L$ be a group. Recall that the 1-subcategory of an $L$-Hom-graded category $C$ is the category $C^1$ with the same objects as $C$ and only the degree-1 homogeneous morphisms, see Section \ref{Hom-graded categories}. An $L$-Hom-graded category $C$ is \emph{$L$-semisimple} if $C^1$ is semisimple, and for all $l\in L$, every object of~$C$ is an $l$-direct sum of a finite family of simple objects of~$C^1$.  Notice that for $L = 1$, we recover the definition of a semisimple category. However, in general, an~\mbox{$L$-semi}simple category need not be semisimple.  In particular, a simple object in the 1-subcategory $C^1$ of an $L$-Hom-graded category $C$ is not necessarily simple in $C$. 

\subsection{Simple objects and 1-morphisms}

We say that a 1-morphism $f\colon X \to Y$ in a $\mathbbm{k}$-linear~2-category~$\mathcal{C}$ is \emph{simple} if $\mathrm{Hom}_{\mathcal{C}(X,Y)}(f,f)$ is a free $\mathbbm{k}$-module of rank 1 (with basis $\{1_f\}$).  We say that an object $X$ in a~$\mathbbm{k}$-linear 2-category is \emph{simple} if $1_X$ is simple. 

\subsection{Semisimple 2-categories} A $\mathbbm{k}$-linear 2-category $\mathcal{D}$ is \emph{semisimple} if 
\begin{itemize}
\item it is \emph{locally semisimple}, that is for all objects $X$ and $Y$ of $\mathcal{D}$, the category $\mathcal{D}(X,Y)$ is semisimple,
\item it is \emph{locally additive}, that is for all objects $X$ and $Y$ of $\mathcal{D}$, the category $\mathcal{D}(X,Y)$ is additive,
\item every 1-morphism of $\mathcal{D}$ admits a left adjoint and a right adjoint,
\item every object of $\mathcal{D}$ is a direct sum of a finite family of simple objects.
\end{itemize}
Notice that a presemisimple 2-category in the sense of \cite[Definition 1.2.7]{douglas_fusion_2018} is a semisimple 2-category in the above sense. 

Let $\mathcal{D}$ be a semisimple 2-category and let $X$ and $Y$ be objects of $\mathcal{D}$. A set $I$ of simple 1-morphisms of~$\mathcal{D}(X,Y)$ is \emph{representative} if every simple 1-morphism in $\mathcal{D}(X,Y)$ is isomorphic to a unique element of $I$.  Each~1-morphism in $\mathcal{D}(X,Y)$ is then a direct sum of a finite family of 1-morphisms in $I$, and $\mathrm{Hom}_{\mathcal{D}(X,Y)}(i,j) = 0$ for any distinct elements $i,j\in I$. 

A set $J$ of simple objects of a semisimple 2-category $\mathcal{D}$ is \emph{representative} if every simple object of $\mathcal{D}$ is equivalent to a unique element of $J$.  Each object of $\mathcal{D}$ is then a direct sum of a finite family of objects in $J$.  

Contrary to the case of semisimple categories,  non-equivalent simple objects of a semisimple 2-category $\mathcal{D}$ may have nonzero 1-morphisms between them, see for instance \cite[Example 1.4.14]{douglas_fusion_2018}. We say that two simple objects $X$ and $Y$ in $\mathcal{D}$ are \emph{disconnected} if $\mathcal{D}(X,Y)$ is the terminal category, and we say that they are \emph{connected} otherwise. A set $K$ of disconnected simple objects of $\mathcal{D}$ is \emph{representative} if every simple object of $\mathcal{D}$ is connected to a unique element of $K$.

\subsection{Graded-semisimple 2-categories}\label{Section_Graded-semisimple 2-categories} Recall from Section \ref{Subsection_1 and (1,1)-subcategories} that the 1-subcategory of a $\sigma$-graded \mbox{2-cate}\-gory~$\mathcal{C}$ is the linear monoidal 2-category $\mathcal{C}^1$ with the same objects and 1-morphisms as $\mathcal{C}$, and with only degree-1 homogeneous 2-morphisms.  A $\sigma$-graded 2-category $\mathcal{C}$ is \emph{$\sigma$-semisimple} (over $\mathbbm{k}$) if 
\begin{itemize}
\item it is \emph{locally $L$-semisimple}, that is for all objects $X$ and $Y$ of $\mathcal{C}$, the category $\mathcal{C}(X,Y)$ is $L$-semisimple,
\item it is \emph{locally $L$-additive}, that is for all objects $X$ and $Y$ of $\mathcal{C}$, the category $\mathcal{C}(X,Y)$ is $L$-additive (see Section \ref{Hom-graded categories}),
\item every 1-morphism of $\mathcal{C}$ admits a left adjoint and a right adjoint in $\mathcal{C}^1$,
\item for all $(l,e)\in L\times E$, every object of $\mathcal{C}$ is a $(l,e)$-direct sum of a finite family of simple objects of $\mathcal{C}^1$ (see Section~\ref{Direct sums}). 
\end{itemize}
Note that for $\sigma$ trivial, we recover the definition of a semisimple 2-category.  The $1$-subcategory $\mathcal{C}^1$, the $(1,1)$-subcategory $\mathcal{C}_1^1$, and the neutral component $\prescript{}{1}{\mathcal{C}_1^1}$ of a $\sigma$-semisimple 2-category $\mathcal{C}$ are all semisimple (see Sections~\ref{Subsection_1 and (1,1)-subcategories} and \ref{Section-neutral component}).  However,  in general, a $\sigma$-semisimple~2-category need not be semisimple.  In particular, a simple object in~$\mathcal{C}^1$ is not necessarily simple in $\mathcal{C}$.  

\begin{lem}
Let $\mathcal{C}$ be a $\sigma$-semisimple 2-category and let $\mathcal{C}^1$ denote its 1-subcategory. Any simple object in~$\mathcal{C}^1$ is $(1,1)$-equivalent to a homogeneous object, and any simple 1-morphism in $\mathcal{C}^1$ is homogeneous.
\end{lem}
\begin{proof}
Let $X$ denote a simple object in $\mathcal{C}^1$.  We know that $X$ is a $(1,1)$-direct sum of a finite family of homogeneous objects~$\{X_i\}_{i\in I}$. Let $(p_i \colon X\to X_i, \ q_i\colon X_i\to X)_{i\in I}$ denote the 1-morphisms of this direct sum.  We have 
$$
1_X = \bigoplus^1_{i\in I} (q_i\circ p_i),
$$
 but $1_X$ is simple in $\mathcal{C}^1$, therefore there is a unique $i_0 \in I$ such that $q_{i_0}\circ p_{i_0}$ is nonzero.  For $i\neq i_0$, $q_i \circ p_i = 0$, therefore $p_i\circ q_i \circ p_i = 0$. Moreover, for $i\neq i_0$, $1_{X_i} \cong p_i\circ q_i$, therefore $p_i = 0$, and therefore $X_i$ is a zero object.  We deduce that $X$ is $(1,1)$-equivalent to a homogeneous object $X_{i_0}$. 
 
Let $f$ be a simple 1-morphism in $\mathcal{C}^1$.  We know that $f$ is a 1-direct sum of a finite family of homogeneous 1-mor\-phisms~$\{f_i\}_{i\in I}$.  By simplicity of $f$, there is a unique $i_0\in I$ such that $f\cong f_{i_0}$. Therefore $f$ is a 1-direct sum of a finite family of homogeneous 1-morphisms of degree $\lvert  f_{i_0}\rvert$, and $f$ is homogeneous. 
\end{proof}

Let $\mathcal{C}$ be a $\sigma$-semisimple 2-category.  Recall that for all objects $X$ and $Y$, the category $\mathcal{C}^1(X,Y)$ is graded by (the set) $E$ (see Section \ref{Section_Set-graded}), and so decomposes as
$$
\mathcal{C}^1(X,Y) = \bigoplus_{e\in E} \mathcal{C}^1_e(X,Y),
$$
where $\mathcal{C}^1_e(X,Y) = \mathcal{C}^1(X,Y)_e$ is the category whose objects are the degree-$e$ homogeneous 1-morphisms from $X$ to $Y$, and whose morphisms are the homogeneous 2-morphisms of degree 1. 
Notice that~$\mathcal{C}^1_e(X,Y)$ is semisimple for all~\mbox{$e\in E$.} Recall that the $(1,1)$-subcategory $\mathcal{C}_1^1$ of a~$\sigma$-semisimple 2-category $\mathcal{C}$ is $H$-graded (see Section~\ref{Section_group-graded}), and for all $h\in H$, the 2-category $\prescript{}{h}{\mathcal{C}_1^1}$ is semisimple. 

We list some useful properties of $\sigma$-semisimple 2-categories:
\begin{lemref}[\cite{sozer_monoidal_2023}]\label{Lemma SV}
Let $\mathcal{C}$ be a $\sigma$-semisimple 2-category (or a locally $L$-semisimple 2-category). Then,
\begin{enumerate}
\item For any simple 1-morphism $f$ of $\mathcal{C}^1$ and for all $l\in L$, there is a simple 1-morphism $g$ of $\mathcal{C}^1$ and a degree-$l$~2-isomorphism from $f$ to $g$.
\item For any simple 1-morphisms $f,g\colon X\to Y$ of $\mathcal{C}^1$ and for all $l\in L$,  $\mathrm{Hom}^l_{\mathcal{C}(X,Y)}(f,g) \cong \mathbbm{k}$ if there is a degree-$l$ homogeneous 2-isomorphism from $f$ to $g$,  and $\mathrm{Hom}^l_{\mathcal{C}(X,Y)}(f,g) = 0$ otherwise.
\item For any 1-morphisms $f,g \colon X \to Y$ and for all $l,k \in L$,  there is a $\mathbbm{k}$-linear isomorphism
$$
\operatorname{Hom}_{\mathcal{C}(X,Y)}^{kl}(f,g)\cong \bigoplus\limits_{i\in I}\operatorname{Hom}_{\mathcal{C}(X,Y)}^{k}(i, g)\otimes \operatorname{Hom}_{\mathcal{C}(X,Y)}^{l}(f,i),
$$
where $I$ is any representative set of simple 1-morphisms of $\mathcal{C}^1(X,Y)$.
\end{enumerate}

\end{lemref}

A $\sigma$-semisimple 2-category $\mathcal{C}$ is \emph{finite} if 
\begin{itemize}
\item for all objects $X$, $Y$ of $\mathcal{C}$ and for all $e\in E$,  any representative set of simple 1-morphisms of~$\mathcal{C}^1_e(X,Y)$ is finite,
\item for all $e\in \mathrm{Ker}(\partial)$, any representative set of simple 1-morphisms of~$\mathcal{C}_e^1(\mathbbm{1}, \mathbbm{1})$ is non-empty,
\item for all $h\in H$,  any representative set of simple objects of $\prescript{}{h}{\mathcal{C}_1^1}$ is finite and non-empty.
\end{itemize}
We can rephrase the above definition as follows. A $\sigma$-semisimple 2-category $\mathcal{C}$ is \emph{finite} if for all objects $X$ and $Y$ of $\mathcal{C}$, and for all $e\in E$,  the set of degree-$1$ homogeneous 2-isomorphism classes of degree-$e$ simple 1-morphisms of~$\mathcal{C}^1$ from $X$ to $Y$ is finite; for all $e\in \mathrm{Ker}(\partial)$, the set of degree-$e$ homogeneous endomorphisms of $\mathbbm{1}$ is non-empty; and for all~$h\in H$, the set of $(1,1)$-equivalence classes of degree-$h$ homogeneous simple objects of $\mathcal{C}^1$ is finite and non-empty.

\subsection{Graded-fusion 2-categories}\label{Definition_Graded-fusion 2-categories}

A \emph{$\sigma$-fusion 2-category} (over $\mathbbm{k}$) is a finite $\sigma$-semisimple 2-category $\mathcal{C}$ such that all objects have left and right duals in the $(1,1)$-subcategory of $\mathcal{C}$, and the monoidal unit is simple in $\mathcal{C}^1$.  A \emph{spherical~$\sigma$-fusion 2-category} is a $\sigma$-fusion 2-category that is spherical as a $\sigma$-graded 2-category.

\begin{lem}\label{nonzero_Hom_categories} Let  $X$, $Y$ be objects in a $\sigma$-fusion 2-category $\mathcal{C}$,  let $e\in E$,  and $f\in \mathrm{Ker}(\partial)$. If~$\mathcal{C}^1_e(X,Y)$ contains a nonzero 1-morphism, then so does $\mathcal{C}^1_{ef}(X,Y)$.  In particular, for any simple object $X$ of~$\mathcal{C}^1$ and for any~\mbox{$e\in \mathrm{Ker}(\partial)$,} there is a simple 1-morphism in~$\mathcal{C}^1_e(X,X)$.
\end{lem}

\begin{proof}
We first suppose $Y$ is simple in $\mathcal{C}^1$.  Let $g\colon X\to Y$ be a nonzero degree-$e$ homogeneous 1-morphism, and let~$h\colon \mathbbm{1}\to \mathbbm{1}$ be a degree-$f$ homogeneous 1-morphism.  We can suppose that $g$ and $h$ are simple in $\mathcal{C}^1$. Then the~1-morphism $(h\otimes Y) \circ g \colon X\to Y$ is homogeneous of degree $ef$.  By the proof of \cite[Proposition~1.2.19]{douglas_fusion_2018}, the equation $(h\otimes Y) \circ g = 0$ implies that $h\otimes Y = 0$.  Let $Y^\#$ denote a right dual of $Y$ and let~$e \colon Y \otimes Y^\# \to \mathbbm{1}$ denote the associated fold.  We have
\begin{align*}
h\otimes Y = 0 & \Rightarrow h \otimes Y \otimes Y^\# = 0\\
& \Rightarrow (1 \otimes e) \circ (h\otimes Y \otimes Y^\#) = 0\\
& \Rightarrow h\circ e = 0.
\end{align*}
By the proof of \cite[Proposition 1.2.19]{douglas_fusion_2018} $h\circ e = 0$ implies that $h = 0$,  and therefore $(h\otimes Y) \circ g$ is nonzero. 

We no longer suppose that $Y$ is simple, and we write $Y =\boxplus^{(1,1)}_{i\in I} Y_i $, where $I$ is a finite set and $Y_i$ is simple in $\mathcal{C}^1$ for all $i\in I$.  We denote by $p_i \colon Y\to Y_i$ and $q_i\colon Y_i \to Y$ the projection and inclusion 1-morphisms for the~$(1,1)$-direct sum.  Let $g\colon X\to Y$ be a degree-$e$ simple 1-morphism, and let~$h\colon \mathbbm{1}\to \mathbbm{1}$ be a degree-$f$ simple 1-morphism. Choose any~$i\in I$ such that $p_i \circ g \colon X \to Y_i$ is nonzero.  By the above discussion, since $Y_i$ is simple,  the 1-morphism
$$
(h\otimes Y_i) \circ p_i \circ g\colon X \to Y_i
$$ 
is nonzero and homogeneous of degree $ef$, and therefore $q_i \circ (h\otimes Y_i) \circ p_i \circ g\colon X \to Y$ is nonzero and homogeneous of degree $ef$. 
\end{proof}

Notice that the 2-crossed module\begin{tikzcd}\sigma = (L\arrow[r, "\delta"] & E \arrow[r, "\partial"] & H, \ \omega)\end{tikzcd}induces a crossed module $(\delta\colon L \to \mathrm{Ker}(\partial))$, where $\mathrm{Ker}(\partial)$ acts on $L$ via $\triangleright$.  If $X$ is a simple object of the 1-subcategory $\mathcal{C}^1$ of a $\sigma$-fusion 2-category $\mathcal{C}$, then by Lemma \ref{nonzero_Hom_categories}, the endomorphism category $\mathcal{C}(X,X)$ is a $(\delta \colon L \to \mathrm{Ker}(\partial))$-fusion category in the sense of Sözer and Virelizier, see~\cite[Section 4.10]{sozer_monoidal_2023}.  In particular, $\mathcal{C}_1^1(X,X)$ is a fusion category. Moreover, up to idempotent and additive completion, the neutral component of a $\sigma$-fusion 2-category is a fusion 2-category in the sense of Douglas and Reutter when $\mathbbm{k}$ is an algebraically closed field of characteristic 0, see \cite[Definition~2.1.6]{douglas_fusion_2018}.  However, in general, a $\sigma$-fusion 2-category is not necessarily a fusion~2-category.

\subsection{The linearization of a 2-crossed module}\label{Ex. associated fusion}The linearization $\mathcal{G}_\sigma$ of $\sigma$ (see Section \ref{Ex: graded 2-category associated to a 2-crossed module}) is a spherical $\sigma$-fusion 2-category.  A representative set of simple objects of $(\mathcal{G}_\sigma)^1$ is given by $H$, and for all $h,k\in H$, a representative set of simple 1-morphisms of $(\mathcal{G}_\sigma)^1(h,k)$ is given by $\{e\in E \mid \partial(e) = kh^{-1}\}$.

\subsection{The pushforward of a $\sigma$-fusion 2-category}\label{Ex. pushforward fusion} Let
$$
\begin{tikzcd}\sigma' = (L'\arrow[r, "\delta'"] & E'\arrow[r, "\partial'"] & H,' \ \omega')\end{tikzcd}
$$
be a 2-crossed module of groups, and let $\phi = (\phi_1, \phi_2, \phi_3)$ be a morphism from $\sigma$ to $\sigma'$. Recall from Section~\ref{pushforward} the definition of the push\-forward~$\phi_*(\mathcal{C})$ of a $\sigma$-graded 2-category $\mathcal{C}$ along $\phi$. Notice that if $\phi \colon \sigma\to \sigma'$ is a morphism of 2-crossed modules such that 
\begin{itemize}
\item $\phi_1 \colon H \to H'$, $\phi_2\colon E \to E'$, and $\phi_3\colon L \to L'$ are surjective,
\item $\mathrm{Ker}(\phi_1)$ and $\mathrm{Ker}(\phi_2)$ are finite,
\item $\mathrm{Ker}(\phi_3) \cap \mathrm{Ker}(\delta) = 1$,
\end{itemize}
then the pushforward $\phi_*(\mathcal{C})$ of a $\sigma$-fusion 2-category $\mathcal{C}$ is a $\sigma'$-fusion 2-category.

\section{Dimensions in graded-fusion 2-categories}\label{Dimensions in graded-fusion 2-categories}
\noindent In this section we define the (quantum) dimensions of objects and 1-morphisms in spherical $\sigma$-fusion \mbox{2-categories}.  We use these to derive a notion of categorical dimension for spherical $\sigma$-fusion 2-categories and we prove some useful properties of dimensions. Throughout this section, let
$$
\begin{tikzcd}
\sigma = (L \arrow{r}{\delta} & E \arrow{r}{\partial} & H,\ \omega)
\end{tikzcd}
$$
be a~2-crossed module of groups, and let $\mathcal{C}$ be a spherical $\sigma$-fusion 2-category over $\mathbbm{k}$.  Recall from Section \ref{Subsection_1 and (1,1)-subcategories} that the~\mbox{1-subcategory} $\mathcal{C}^1$ of $\mathcal{C}$ has the same objects and 1-morphisms as $\mathcal{C}$, and only degree-1 homogeneous 2-morphisms. Recall that $\mathbbm{1}$ is a simple object in $\mathcal{C}^1$, and therefore we identify~$\mathrm{Hom}_{\mathcal{C}(\mathbbm{1}, \mathbbm{1})}^1(1_\mathbbm{1}, 1_\mathbbm{1})$ and  $\mathbbm{k}$. 

\subsection{Dimensions of objects and 1-morphisms}\label{Dimensions of objects and 1-morphisms}
The \emph{dimension} of a 1-morphism $f$ in $\mathcal{C}$ is 
$$
\dim (f) = \mathrm{Tr}(1_f) \in \mathbbm{k},
$$ 
where $\mathrm{Tr}(1_f)$ denotes the trace of $1_f$ (see Sections \ref{Subsection_Spherical 2-categories} and \ref{Subsection_Spherical graded 2-categories}). The \emph{dimension} of an object $X$ in $\mathcal{C}$ is
$$
\dim(X) = \dim(1_X) = \mathrm{Tr}(1_{1_X}) \in \mathbbm{k}.
$$

\begin{lem}\label{nonzero dimension of simple 1-morphisms}
The dimensions of simple objects and simple~1-morphisms in $\mathcal{C}^1$ are invertible in $\mathbbm{k}$. 
\end{lem}

\begin{proof}
First, we show that if $f\colon X \to Y$ denotes a simple 1-morphism in $\mathcal{C}^1$ with $Y$ a simple object in $\mathcal{C}^1$, then the right planar trace $\mathrm{tr}_R(1_f)$ (as defined in \cite{douglas_fusion_2018}) is invertible in $\mathbbm{k}$. We adapt the proof of \cite[Lemma~2.3.10]{douglas_fusion_2018}. By definition, $\mathrm{tr}_R(1_f) = \varepsilon_f\cdot \eta_{f^*} = \lambda 1_{1_Y}$ for some $\lambda \in \mathbbm{k}$. Note that by adjunction and by simplicity of~$f$:
$$
\mathrm{Hom}_{\mathcal{C}(Y,Y)}^1(f\circ f^* , 1_Y) \cong \mathrm{Hom}_{\mathcal{C}(X,Y)}^1(f,f) \cong \mathbbm{k},
$$
and $\{\varepsilon_f\}$ is a basis of $\mathrm{Hom}_{\mathcal{C}(Y,Y)}^1(f\circ f^* , 1_Y)$. Similarly,  $\mathrm{Hom}_{\mathcal{C}(Y,Y)}^1( 1_Y, f\circ f^*) = \mathbbm{k}\{\eta_{f^*}\}$. Moreover, by local \mbox{$L$-semi}\-simplicity of $\mathcal{C}$ and by simplicity of $Y$,  $1_Y$ is a direct summand in the 1-direct sum decomposition of~$f\circ f^*$ into simple 1-morphisms.  Therefore, there are 2-morphisms $\alpha \in \mathrm{Hom}_{\mathcal{C}(Y,Y)}^1(f\circ f^* , 1_Y)$ and~$\beta \in \mathrm{Hom}_{\mathcal{C}(Y,Y)}^1(1_Y, f\circ f^*)$ such that $\alpha\cdot \beta = 1_{1_Y}$. We write $\alpha = x\varepsilon_f$ and $\beta = y \eta_{f^*}$, for some scalars $x, y \in \mathbbm{k}$. We have
$$
1_{1_Y} = \alpha\cdot \beta = (xy) \varepsilon_f\cdot \eta_{f^*} = (xy \lambda) 1_{1_Y},
$$
therefore $\lambda$ is invertible.

Now let $f\colon X\to Y$ be a simple 1-morphism in $\mathcal{C}^1$, where $X$ and $Y$ are not necessarily simple.  By the duality between $Y$ and $Y^\#$, we have 
\begin{equation*}
\operatorname{End}_{\mathcal{C}(X\otimes Y^\#, \mathbbm{1})}^1(e_Y\circ (f\otimes Y^\#)) \cong \operatorname{End}_{\mathcal{C}(X,Y)}^1(f) \cong \mathbbm{k}.
\end{equation*}
Therefore,  $e_Y\circ (f\otimes Y^\#) \colon X\otimes Y^\#\to \mathbbm{1}$ is a simple 1-morphism in $\mathcal{C}^1$.  By the above discussion, the right trace~\mbox{$\operatorname{tr}_R(1_{e_Y\circ (f\otimes Y^\#)})$} is invertible, and the dimension of $f$ is 
$$
\operatorname{Tr}(1_f) = \operatorname{tr}_R (e_Y\circ (f\otimes Y^\#)) = \operatorname{tr}_R (1_{e_Y\circ (f\otimes Y^\#)}).
$$
Therefore $\dim(f)$ is invertible in $\mathbbm{k}$. As a consequence, if $X$ is a simple object in $\mathcal{C}^1$, then $\dim(X) = \dim(1_X)$ is invertible in $\mathbbm{k}$. 
\end{proof}

For any 1-morphisms $f, g\colon X\to Y$ in $\mathcal{C}$ and for all $l\in L$,  the \emph{trace pairing} is the map
$$
\langle \cdot, \cdot\rangle \colon \operatorname{Hom}_\mathcal{C}^l(f,g) \otimes \operatorname{Hom}_\mathcal{C}^{l^{-1}}(g,f) \to \mathbbm{k} \quad \quad \alpha\otimes \beta \mapsto \langle \alpha, \beta\rangle = \mathrm{Tr}(\alpha\cdot \beta).
$$

\begin{lem}\label{The trace pairing is nondegenerate}
The trace pairing $\langle \cdot, \cdot \rangle$ is nondegenerate.
\end{lem}

\begin{proof}
We adapt the proof of \cite[Proposition 2.3.14]{douglas_fusion_2018}. Let $\{s_i\}_{i\in I}$ be a representative set of simple 1-morphisms in $\mathcal{C}^1(X,Y)$.  By local $L$-semisimplicity of $\mathcal{C}$,  we can write
$$
f = \bigoplus^l_{i\in I}s_{k_i} \quad \text{and} \quad g = \bigoplus^1_{j\in J} s_{l_j}.
$$
We denote by $(\pi_i, \iota_i)_{i\in I}$ the projection and inclusion 2-morphisms for $f$, and by $(\pi_j', \iota_j')_{j\in J}$ the projection and inclusion~2-morphisms for $g$.  

Suppose $\alpha \colon f \Rightarrow g$ is a degree-$l$ homogeneous 2-morphism such that for any degree-$l^{-1}$ homogeneous 2-morphism $\beta \colon g \Rightarrow f$, we have $\langle \alpha, \beta\rangle = 0$.  Every degree-$l^{-1}$ homogeneous 2-morphism $g\Rightarrow f$ is a linear combination of the 2-morphisms $\iota_i \cdot \pi_j'$ such that $k_i = l_j$.  By assumption, $0  = \operatorname{Tr}(\alpha \cdot \iota_i \cdot \pi_j') = \operatorname{Tr}(\pi_j' \cdot \alpha \cdot \iota_i)$. By simplicity of $s_{k_i} = s_{l_j}$,  $$ \pi_j' \cdot \alpha \cdot \iota_i = \lambda_{i,j}1_{s_{k_i}}$$ for some scalar $\lambda_{i,j} \in \mathbbm{k}$. Therefore $0 = \lambda_{i,j}\operatorname{Tr}(1_{s_{k_i}}) = \lambda_{i,j}\operatorname{dim}(s_{k_i})$.  From Lemma \ref{nonzero dimension of simple 1-morphisms}, we deduce that~$\lambda_{i,j} = 0$ for all $i\in I$ and $j\in J$.  Therefore,
\begin{align*}
\alpha &= 1_g \cdot \alpha \cdot 1_f
= \left(\sum_{j\in J}\iota_j' \cdot \pi_j'\right) \cdot \alpha \cdot \left(\sum_{i\in I} \iota_i \cdot \pi_i\right)
= \sum_{i\in I, \ j \in J}\iota_j' \cdot \left( \pi_j'\cdot \alpha \cdot \iota_i \right) \cdot \pi_i
= 0.\qedhere
\end{align*}
\end{proof}

\subsection{Weighted dimensions of objects}\label{Weighted dimension}

Recall that $\mathcal{C}^1$ denotes the 1-subcategory of $\mathcal{C}$, see Section \ref{Subsection_1 and (1,1)-subcategories}. Recall that for all objects $X$, $Y$ of $\mathcal{C}$ and $e\in E$,  $\mathcal{C}_e^1(X,Y)$ is the category whose objects are the degree-$e$ homogeneous 1-morphisms from $X$ to $Y$, and whose morphisms are homogeneous 2-morphisms of degree 1, see Section \ref{Section_Graded-semisimple 2-categories}.  For simple objects $X$, $Y$ of $\mathcal{C}^1$ and $e\in E$, the \emph{dimension} of the category $\mathcal{C}_e^1(X,Y)$ is
$$
\dim(\mathcal{C}_e^1(X,Y)) = \frac{1}{\dim(X)\dim(Y)}\sum_{f} \dim(f)^2,
$$
where $f$ runs over a representative set of simple 1-morphisms in $\mathcal{C}_e^1(X,Y)$.  It is well defined; see Lemma \ref{nonzero dimension of simple 1-morphisms}.

The \emph{weighted dimension} of a simple object $X$ of $\mathcal{C}^1$ is
$$
\wdim(X) =\dim(X) \dim(\mathcal{C}^1_1(X,X))\operatorname{n}(X) \in \mathbbm{k},
$$
where $\operatorname{n}(X)$ denotes the number of $(1,1)$-equivalence classes of simple objects that are connected to $X$ in $\mathcal{C}^1_1$.  

\subsection{Categorical dimensions}\label{Dimensions of 2-categories}

Recall that the $(1,1)$-subcategory of $\mathcal{C}$ is $H$-graded (see Section \ref{Section_group-graded}) and that~$\prescript{}{h}{\mathcal{C}_1^1}$ denotes the linear 2-category with degree-$h$ homogeneous objects and degree-1 homogeneous \mbox{1-mor}\-phisms and 2-morphisms. Let $h\in H$ such that the dimensions $\dim(\mathcal{C}_1^1(X,X))\in \mathbbm{k}$ are invertible for all degree-$h$ homogeneous simple objects $X$.  The \emph{dimension} of the 2-category $\prescript{}{h}{\mathcal{C}_1^1}$ is 
$$
\dim(\prescript{}{h}{\mathcal{C}_1^1}) = \sum_{X}\dim(\mathcal{C}_1^1(X,X))^{-1}\in \mathbbm{k},
$$
where $X$ runs over a representative set of disconnected simple objects in $\prescript{}{h}{\mathcal{C}^1_1}$. From \cite[Proposition~1.2.27]{douglas_fusion_2018}, we deduce that if $X$ and $Y$ are simple objects of $\mathcal{C}^1$ that are connected in $\mathcal{C}_1^1$, then
$$
\dim(\mathcal{C}_1^1(X,X)) = \dim(\mathcal{C}_1^1(Y,Y)).
$$
Therefore, the dimension of $\prescript{}{h}{\mathcal{C}_1^1}$ does not depend on the choice of representative set of disconnected simple objects of $\prescript{}{h}{\mathcal{C}_1^1}$. 

We say that $\mathcal{C}$ has \emph{invertible dimensions} if
\begin{itemize}
\item the dimensions $\dim({\mathcal{C}_1^1}(X,X))$ are invertible in $\mathbbm{k}$ for all simple objects $X$ of $\mathcal{C}^1$,
\item the categorical dimension $\dim(\prescript{}{1}{\mathcal{C}_1^1})$ is invertible in $\mathbbm{k}$,
\item for every simple object $X$ of $\mathcal{C}^1$, the number of $(1,1)$-equivalence classes of simple objects of $\mathcal{C}^1$ that are connected to $X$ in $\mathcal{C}_1^1$ is invertible in $\mathbbm{k}$.
\end{itemize}
By \cite[Theorem 2.3]{etingofFusionCategories2005},  any spherical $\sigma$-fusion 2-category has invertible dimensions when $\mathbbm{k}$ is an algebraically closed field of characteristic zero.

\subsection{Example} Let $\mathcal{G}_\sigma$ denote the linearization of $\sigma$ (see Sections \ref{Ex: graded 2-category associated to a 2-crossed module} and~\ref{Ex. associated fusion}).  For any $h\in H$, 
$$
\dim((\mathcal{G}_\sigma)_1^1(h,h)) = 1, \quad\quad  \text{and}\quad\quad\dim(\prescript{}{1}{(\mathcal{G}_\sigma)_1^1}) = 1.
$$

\subsection{Dimension formulas}\label{Section_Dimension_formulas} In this section, we suppose that for every simple object $X$ in $\mathcal{C}$, the dimension $\dim(\mathcal{C}_1^1(X,X))$ is invertible in $\mathbbm{k}$ and that the categorical dimension $\dim(\prescript{}{1}{\mathcal{C}_1^1})$ is invertible in $\mathbbm{k}$.  We prove useful formulas relating the dimensions of objects and~1-morphisms in $\mathcal{C}$.  These are adapted from the lemmas in \cite[Appendix B]{douglas_fusion_2018}. 

\begin{lemref}[\cite{douglas_fusion_2018}]\label{Lemma_B1}
Let $g\colon X\to Y$ and $f\colon Y \to Z$ be 1-morphisms in $\mathcal{C}$, and assume $Y$ is simple in $\mathcal{C}^1$. Then
$$
\dim(f\circ g) = \dim(f)\dim(g)\dim(Y)^{-1}.
$$
\end{lemref}

\begin{lem}\label{Lemma_B2}
Let $f\colon X\to Y$ be a 1-morphism in $\mathcal{C}$. Then, for all $l\in L$,
$$
\sum_{h} \dim(\mathrm{Hom}_{\mathcal{C}(X,Y)}^l(h,f))\dim(h) = \dim(f),
$$
where $h$ runs over a representative set of simple 1-morphisms in $\mathcal{C}^1(X,Y)$. 
\end{lem}
Notice that we allow an infinite sum in Lemma \ref{Lemma_B2} because $\dim(\mathrm{Hom}_{\mathcal{C}(X,Y)}^l(h,f))= 0$ for all but finitely many $h$. 

\begin{proof}
We write $ f = \bigoplus^l_{j\in J} f_j$ with $f_j$ simple 1-morphisms in $\mathcal{C}^1(X,Y)$ for all $j\in J$.  We denote by
$$
\pi_j \colon f \Rightarrow f_j, \quad \text{and} \quad \iota_j \colon f_j \Rightarrow f
$$
the projection and inclusion 2-morphisms for this $l$-direct sum. Let $h$ be a simple 1-morphism in $\mathcal{C}^1(X,Y)$. Recall that 
$$
\mathrm{Hom}_{\mathcal{C}(X,Y)}^l(h,f) \cong \mathrm{Hom}_{\mathcal{C}(X,Y)}^l(h, \bigoplus_{j\in J}^l f_j) \cong  \bigoplus_{j\in J} \mathrm{Hom}_{\mathcal{C}(X,Y)}^1(h,f_j).
$$ 
We denote by $J_h\subset J$ the set of indices such that $f_j$ is isomorphic to $h$ in $\mathcal{C}^1(X,Y)$.  Notice that 
$$
\dim(\mathrm{Hom}_{\mathcal{C}(X,Y)}^l(h,f)) = \lvert J_h\rvert.
$$
Therefore, 
\begin{align*}
\dim(f) = \mathrm{Tr}(1_f) &= \mathrm{Tr}(\sum_{j\in J} \iota_j \cdot \pi_j)\\
&= \sum_{j\in J}\mathrm{Tr}(\pi_j\cdot \iota_j)\\
&= \sum_{j\in J} \dim(f_j)\\
& = \sum_{h} \lvert J_h\rvert \dim(h),
\end{align*}
where in the last equality $h$ runs over a representative set of simple 1-morphisms in $\mathcal{C}^1(X,Y)$. 
\end{proof}

\begin{lem}\label{Lemma_B3}
Let $f\colon X\to Y$ be a 1-morphism in $\mathcal{C}$. Then, for all $e\in E$,
$$
\sum_W \sum_{g\colon W\to X} \dim(f\circ g)\dim(g) \wdim(W)^{-1} = \dim(f),
$$
where $W$ runs over a representative set of homogeneous simple objects in $\mathcal{C}^1$, and $g$ runs over a representative set of degree-$e$ simple 1-morphisms in $\mathcal{C}^1(W,X)$. 
\end{lem}

Notice that the sum over the $W$ in Lemma \ref{Lemma_B3} is infinite, but this is not an issue because
$$
\sum_{g\colon W\to X} \dim(f\circ g)\dim(g) \wdim(W)^{-1}= 0
$$
for all but finitely many $W$. 

\begin{proof}
Let $W$ be a degree-$h$ homogeneous simple object in $\mathcal{C}^1$.  We decompose $X$ into a~$(1,e)$-direct sum
$$
X =  \boxplus_{i\in I}^{(1,e)}X_i,
$$
with $p_i \colon X \to X_i$ homogeneous of degree $e^{-1}$, $q_i\colon X_i \to X$ homogeneous of degree $e$, and $\phi_i \colon p_i\circ q_i\ \tilde{ \Rightarrow} \ 1_{X_i}$ homogeneous of degree 1, for all $i\in I$.  If $g \colon W\to X$ is a degree-$e$ simple 1-morphism in $\mathcal{C}^1$, then there exists a unique index $\alpha(g) \in I$ such that $p_{\alpha(g)} \circ g \neq 0$. In fact, there is a bijection
\begin{align*}
\left\{\text{iso classes of simples in $\mathcal{C}_e^1(W,X)$}\right\} & \leftrightarrow  \bigsqcup_{i\in I}\left \{\text{iso classes of simples in $\mathcal{C}_1^1(W,X_i)$}\right\}
\end{align*}
given by sending a degree-$e$ simple 1-morphism $g$ in $\mathcal{C}^1(W,X)$ to $p_{\alpha(g)} \circ g$,  and a degree-$1$ simple 1-morphism $h$ in $\mathcal{C}^1(W, X_i)$ to $q_i\circ h$.  Therefore, 
\begin{align*}
\sum_{g\colon W \to X} \dim(f\circ g) \dim(g) &= \sum_{i\in I} \sum_{h\colon W\to X_i} \dim(f\circ q_i \circ h) \dim(q_i\circ h)\\
&= \sum_{i\in I} \sum_{h\colon W\to X_i} \frac{\dim(q_i) \dim(h)}{\dim(X_i)}\dim(f\circ q_i\circ h),
\end{align*}
where the second equality follows from Lemma \ref{Lemma_B1}. Therefore,
\begin{align*}
\sum_{g\colon W \to X} \dim(f\circ g) \dim(g) &= \sum_{i \in I_W} \frac{\dim(q_i)}{\dim(X_i)}\dim(f\circ q_i) \dim(W) \dim(\mathcal{C}^1_1(W,W))
\end{align*}
where $I_W \subset I$ is the set of indices such that $X_i$ and $W$ are connected.  Notice that $q_i$ is adjoint to~$p_i$, therefore 
$$
\dim(q_i) = \dim(p_i),
$$
and 
$$
\frac{\dim(q_i)}{\dim(X_i)}\dim(f\circ q_i) = \dim(f\circ q_i) \dim(p_i) \dim(X_i)^{-1} = \dim (f\circ q_i \circ p_i).
$$
We conclude that
\begin{align*}
\sum_{W}\sum_{g\colon W\to X} \frac{\dim(f\circ g)\dim(g)}{\wdim(W)} 
&= \sum_{W} \frac{1}{\mathrm{n}(W)}\sum_{i\in I_W} \dim(f\circ q_i\circ p_i)\\
&= \sum_{i\in I} \dim(f\circ q_i\circ p_i)\\
&= \dim(f)
\end{align*}
where the last equality results from the fact that for any 1-morphisms $h$ and $k$, $\dim(h\oplus k) = \dim(h) +\dim(k)$. 
\end{proof}

\begin{lem}\label{Lemma_B4}
For all $e\in E$, 
$$
\sum_{X}\sum_{f\colon X\to Y} \frac{\dim(f)^2}{\wdim(X)} = \dim(Y),
$$
where $X$ runs over a representative set of homogeneous simple objects in $\mathcal{C}^1$, and $f$ runs over a representative set of degree-$e$ simple 1-morphisms in $\mathcal{C}^1(X,Y)$. 
\end{lem}

\begin{proof}
Apply Lemma \ref{Lemma_B3} with $f = 1_X$. 
\end{proof}

\begin{lem}\label{Lemma_B5}
Let $Z$ be a simple object in $\mathcal{C}^1$.  Then, for all $h,k\in H$ and for all $e\in E$ such that $$
\partial(e) = \lvert Z \rvert k^{-1}h^{-1},
$$
we have
$$
\sum_{X, Y} \sum_{f\colon X\otimes Y\to Z} \frac{\dim(f)^2}{\dim(Z)\wdim(X)\wdim(Y)} = \dim(\prescript{}{1}{\mathcal{C}^1_1}),
$$
where $X$ runs over a representative set of degree-$h$ homogeneous simple objects in $\mathcal{C}^1$, $Y$ runs over a representative set of degree-$k$ homogeneous simple objects in $\mathcal{C}^1$, and $f$ runs over a representative set of degree-$e$ simple~1-morphisms in $\mathcal{C}^1(X\otimes Y, Z)$. 
\end{lem}
\begin{proof}
By sphericality and by Lemma \ref{Lemma_B4},  the left-hand side of this equation is equal to 
$$
\sum_{X,Y} \sum_{f\colon X\to Z\otimes Y^\#}  \frac{\dim(f)^2}{\dim(Z)\wdim(X)\wdim(Y)} = \sum_Y \frac{\dim(Z\otimes Y^\#)}{\dim(Z) \wdim(Y)} = \sum_Y \frac{1}{\mathrm{n}(Y) \dim(\mathcal{C}_1^1(Y,Y))} = \dim(\prescript{}{k}{\mathcal{C}_1^1}).
$$
Similarly, the left-hand side is also equal to
$$
\sum_{X,Y} \sum_{f\colon Y\to X^\#\otimes Z}  \frac{\dim(f)^2}{\dim(Z)\wdim(X)\wdim(Y)} = \dim(\prescript{}{h}{\mathcal{C}_1^1}).
$$
As a consequence, we have that $\dim(\prescript{}{h}{\mathcal{C}_1^1}) = \dim(\prescript{}{1}{\mathcal{C}_1^1})$ for all $h\in H$.  This proves the result. 
\end{proof}

\begin{lem}\label{Lemma_B6}
Let $f\colon X\to Y$ be a 1-morphism in $\mathcal{C}$, let $e\in E$, and let $l\in L$. Then
$$
\sum_Z \sum_{h\colon X\to Z} \sum_{g\colon Z \to Y} \dim(\mathrm{Hom}_{\mathcal{C}(X,Y)}^l(g\circ h, f) )\frac{\dim(g) \dim(h)}{\wdim(Z)} = \dim(f),
$$
where $Z$ runs over a representative set of simple objects of $\mathcal{C}^1$, $h$ runs over a representative set of simple 1-morphisms of $\mathcal{C}^1(X,Z)$, and $g$ runs over a representative set of degree-$e$ simple 1-morphisms in $\mathcal{C}^1(Z,Y)$. 
\end{lem}

\begin{proof}
By pivotality, 
$$
\mathrm{Hom}_{\mathcal{C}(X,Y)}^l(g\circ h, f) \cong \mathrm{Hom}_{\mathcal{C}(X,Z)}^{e^{-1}\triangleright l}(h, g^*\circ f).
$$
Using Lemma \ref{Lemma_B2}, we have
$$
\sum_{h\colon X\to Z}\dim(\mathrm{Hom}_{\mathcal{C}(X,Z)}^{e^{-1}\triangleright l}(h, g^*\circ f))\dim(h) = \dim(g^*\circ f) = \dim(f^*\circ g).
$$
Therefore, using Lemma \ref{Lemma_B3}, the left-hand side of the equation becomes
$$
\sum_Z \sum_{g\colon Z \to Y}\frac{\dim(g)\dim(f^*\circ g)}{\wdim(Z)} = \dim(f^*),
$$
and $\dim(f^*) = \dim(f)$. 

\end{proof}

\begin{lem}\label{Lemma_B7}
Let $f\colon X\to Y$ be a 1-morphism in $\mathcal{C}$, let $e\in E$ and $l\in L$.  Then,
$$
\sum_Z  \sum_{h\colon X\to Z} \sum_{g\colon Z\to Y} \frac{\dim(g)\dim(h)}{\wdim(Z)}\sum_{\gamma} \gamma\cdot \hat{\gamma} = 1_f,
$$
where $Z$ runs over a representative set of simple objects of $\mathcal{C}^1$, $h$ runs over a representative set of simple 1-morphisms in $\mathcal{C}^1(X,Z)$, and $g$ runs over a representative set of degree-$e$ simple 1-morphisms in $\mathcal{C}^1(Z,Y)$. Moreover, $\gamma$ runs over a basis of the free module $\mathrm{Hom}^l_{\mathcal{C}(X,Y)}(g\circ h, f)$, and $\hat{\gamma}$ is the dual of $\gamma$ in the basis of \mbox{$\mathrm{Hom}^{l^{-1}}_{\mathcal{C}(X,Y)}(f, g\circ h)$} with respect to the trace pairing. 
\end{lem}

\begin{proof}
We adapt the proof of \cite[Corollary B.7]{douglas_fusion_2018}.  Note that $\sum_{\gamma} \gamma\cdot \hat{\gamma}$ is independent of the choice of basis of~$\mathrm{Hom}^l_{\mathcal{C}(X,Y)}(g\circ h, f)$. Let $\{s_i\}_{i\in I}$ be a representative set of simple 1-morphisms in $\mathcal{C}^1(X,Y)$.  For each $i\in I$, let~$\{\alpha^i_j\}_{j\in J_i}$ be a basis of $\mathrm{Hom}_{\mathcal{C}(X,Y)}^1(s_i, f)$, and let $\{\beta^i_k\}_{k\in K_i}$ be a basis of $\mathrm{Hom}_{\mathcal{C}(X,Y)}^l(g\circ h, s_i)$.  By Lemma~\ref{Lemma SV}, the set 
$$
\mathcal{B} = \sqcup_{i\in I}\{\alpha^i_j\cdot \beta^i_k \mid j\in J_i, \ k \in K_i \}
$$
is a basis of $\mathrm{Hom}_{\mathcal{C}(X,Y)}^l(g\circ h, f)$.  

Since $s_i$ is simple, there are nondegenerate pairings
$$
\mathrm{Hom}_{\mathcal{C}(X,Y)}^1(f, s_i) \otimes \mathrm{Hom}_{\mathcal{C}(X,Y)}^1(s_i, f)\to \mathbbm{k} \quad \text{and} \quad  \mathrm{Hom}_{\mathcal{C}(X,Y)}^l(g\circ h, s_i) \otimes \mathrm{Hom}_{\mathcal{C}(X,Y)}^{l^{-1}}(s_i, g\circ h)\to \mathbbm{k},
$$
given by the vertical composition of 2-morphisms, where we identify $\mathrm{Hom}_{\mathcal{C}(X,Y)}^1(s_i, s_i)$ with $\mathbbm{k}$.  For these pairings, we denote by~$\{\widehat{\alpha}^i_j\}_{j\in J_i}$ and $\{\widehat{\beta}^i_k\}_{k\in K_i}$ the dual bases of $\{\alpha^i_j\}_{j\in J_i}$ and $\{\beta^i_k\}_{k\in K_i}$, respectively. Direct calculation shows that 
$$
\bigsqcup_{i\in I} \left\{\frac{1}{\dim(s_i)}\widehat{\beta}^i_k\cdot \widehat{\alpha}^i_j \mid k\in K_i, \ j\in J_i\right\}
$$
is a dual basis of $\mathcal{B}$ for the nondegenerate trace pairing (see Lemma \ref{The trace pairing is nondegenerate})
$$
\mathrm{Hom}_{\mathcal{C}(X,Y)}^l(g\circ h, f) \otimes \mathrm{Hom}_{\mathcal{C}(X,Y)}^{l^{-1}}(f, g\circ h)\to \mathbbm{k}.
$$
With this choice of basis, the left-hand side of the equation becomes
$$
\sum_Z  \sum_{h\colon X\to Z} \sum_{g\colon Z\to Y} \frac{\dim(g)\dim(h)}{\wdim(Z)}\sum_{i\in I}\frac{1}{\dim(s_i)}\sum_{j\in J_i}\sum_{k\in K_i}\alpha^i_j\cdot  \beta^i_k\cdot  \widehat{\beta}^i_k\cdot  \widehat{\alpha}^i_j.
$$
Notice that $\sum_{k\in K_i} \beta^i_k\cdot  \widehat{\beta}^i_k = \dim(\mathrm{Hom}_{\mathcal{C}(X,Y)}^l(g\circ h, s_i))1_{s_i}$. Therefore, by Lemma \ref{Lemma_B6}, we get that the left-hand side of the equation is equal to 
$$
\sum_{i\in I}\frac{1}{\dim(s_i)}\sum_{j\in J_i} \alpha^i_j\cdot\widehat{\alpha}^i_j \dim(s_i),
$$
and $\sum_{i\in I}\sum_{j\in J_i} \alpha^i_j\cdot\widehat{\alpha}^i_j = 1_f$. 
\end{proof}

\section{Colored 10j-symbols}\label{Section. 10j-symbols}

\noindent In this section,  let $\sigma$ be a 2-crossed module of groups (see Section \ref{Section. 2-crossed modules}),  let~$(K,\phi)$ be a $\sigma$-colored simplicial complex (see Section \ref{Section. Classification of homotopies}), and let $\mathcal{C}$ be a spherical $\sigma$-fusion 2-category (see Section \ref{Section. Graded-fusion 2-categories}). We introduce a notion of~$\mathcal{C}$-state of $(K,\phi)$ that associates to each edge of $K$ a simple object of $\mathcal{C}$ and to each triangle of $K$ a simple~1-morphism of~$\mathcal{C}$ in a way that is compatible with the coloring $\phi$.  From every such state, we derive colored~10j-symbols that generalize the 10j-symbols for spherical fusion 2-categories defined by Douglas and Reutter~\cite{douglas_fusion_2018}.  This enables us to obtain a scalar,  called the 10j-action of the canonical associated state. 

\subsection{Simplicial skeletons of spherical graded-fusion 2-categories}
Recall that $\mathcal{C}^1$ denotes the 1-subcategory of $\mathcal{C}$, that is the linear monoidal 2-category with the same objects and 1-morphisms as $\mathcal{C}$ and only the 2-morphisms of degree 1, see Section \ref{Subsection_1 and (1,1)-subcategories}.  The spherical $\sigma$-fusion 2-cate\-gory~$\mathcal{C}$ determines a 2-truncated semisimplicial set $\Delta(\mathcal{C})$ with $\Delta(\mathcal{C})_0 = \{*\}$,  $\Delta(\mathcal{C})_1 = \{\text{homogeneous simple objects of }\mathcal{C}^1\}$, and 
$$
\Delta(\mathcal{C})_2 = \{(X,Y,Z, f) \mid X, Y, Z \in \Delta(\mathcal{C})_1,  f \text{ is a simple 1-morphism in }\mathcal{C}^1(X\otimes Y, Z)\}.
$$
A \emph{simplicial skeleton} of a spherical $\sigma$-fusion 2-category $\mathcal{C}$ is a subsemisimplicial set $\mathcal{S} \subset \Delta(\mathcal{C})$ such that:
\begin{itemize}
\item the tensor unit $\mathbbm{1}$ is an element of $\mathcal{S}_1$,
\item if $X\in \mathcal{S}_1$, then $X^\#\in \mathcal{S}_1$,
\item every element of $\Delta(\mathcal{C})_1$ is $(1,1)$-equivalent to exactly one element of $\mathcal{S}_1$, 
\item for every element $X \in \mathcal{S}_1$, $(X,\mathbbm{1}, X, 1_X)$ and $(\mathbbm{1}, X, X, 1_X)$  are elements of $\mathcal{S}_2$,
\item if $(X,Y,Z, f) \in \mathcal{S}_2$, then $(Z, Y^\#, X, (X\otimes e_Y) \circ (f^*\otimes Y^\#)) \in \mathcal{S}_2$,
\item for all $X, Y, Z\in \mathcal{S}_1$, and simple $f\in \mathcal{C}^1(X\otimes Y, Z)$, there exists a unique $(X,Y,Z,g)\in \mathcal{S}_2$ such that $f$ and $g$ are isomorphic in $\mathcal{C}^1$.
\end{itemize}

\subsection{States}

Let $E(K)$ denote the set of ordered edges of $K$, and let $T(K)$ denote the set of ordered triangles of $K$, see Section \ref{Section_Ordered simplices}.  A \emph{$\mathcal{C}$-state} of $(K,\phi)$ is a pair of set maps 
$$
\Gamma = (\Gamma_1 \colon E(K) \to \Delta(\mathcal{C})_1, \ \Gamma_2\colon T(K) \to \Delta(\mathcal{C})_2)
$$ such that for every oriented edge $\langle v_0, v_1\rangle$ of $K$, the object $\Gamma_1\langle v_0, v_1\rangle$ is homogeneous of degree $\phi_1\langle v_0, v_1\rangle$ and 
$$
\Gamma_1\langle v_1, v_0\rangle = \Gamma_1\langle v_0, v_1\rangle^\#,
$$
and for every oriented triangle $\langle v_0, v_1, v_2\rangle$ of $K$ with distinguished vertex $v_0$,  the 1-morphism $\Gamma_2\langle v_0, v_1, v_2\rangle$ goes from $\Gamma_1\langle v_0,v_1\rangle \otimes \Gamma_1\langle v_1, v_2\rangle$ to $\Gamma_1\langle v_0, v_2\rangle$,  is homogeneous of degree $\phi_2\langle v_0, v_1, v_2\rangle$, and verifies
$$
\Gamma_2\langle v_0, v_2, v_1\rangle = \left( \Gamma_1\langle v_0, v_1\rangle \otimes e_{\Gamma_1\langle v_1, v_2\rangle}\right) \circ \left( \Gamma_2\langle v_0, v_1, v_2\rangle^*\otimes \Gamma_1\langle v_1, v_2\rangle^\#\right).
$$
We say that a $\mathcal{C}$-state of $(K,\phi)$ \emph{takes values} in a simplicial skeleton $\mathcal{S}$ of $\mathcal{C}$ if $\Gamma_1(E(K))\subset \mathcal{S}_1$ and $\Gamma_2(T(K))\subset \mathcal{S}_2$.

\subsection{Change of pointing for an oriented triangle}
Let $\mathcal{S}$ be a simplicial skeleton of $\mathcal{C}$, and let $\Gamma$ be a $\mathcal{C}$-state of $(K,\phi)$ taking values in $\mathcal{S}$.  For any ordered triangle $\langle v_0,v_1,v_2\rangle$ of $K$ and for any 1-morphism
$$f\colon \Gamma_1\langle v_0,v_1\rangle \otimes \Gamma_1\langle v_1, v_2\rangle \to \Gamma_1\langle v_0,v_2\rangle$$
in $\mathcal{S}_2$, we denote by $\prescript{\Gamma_1\langle v_1, v_0\rangle}{}{f}\colon \Gamma_1\langle v_1, v_2\rangle \otimes \Gamma_1\langle v_2, v_0\rangle \to \Gamma_1\langle v_1, v_0\rangle$ the 1-morphism in $\mathcal{S}_2$ that is degree-1 isomorphic to 
\begin{align*}
(\Gamma_1\langle v_1, v_0\rangle \otimes e_{\Gamma_1\langle v_0, v_2\rangle}) \circ (\Gamma_1\langle v_1, v_0\rangle \otimes f \otimes \Gamma_1\langle v_2, v_0\rangle)\circ ( i_{\Gamma_1\langle v_0,v_1\rangle} \otimes \Gamma_1\langle v_1, v_2\rangle \otimes \Gamma_1\langle v_2, v_0\rangle).
\end{align*}
Similarly, we denote by $\prescript{\Gamma_1\langle v_2, v_0\rangle}{}{f}\colon  \Gamma_1\langle v_0, v_2\rangle \otimes \Gamma_1\langle v_2, v_1\rangle \to \Gamma_1\langle v_0, v_1\rangle$ the 1-morphism in $\mathcal{S}_2$ that is degree-1 isomorphic to 
\begin{align*}
(e_{\Gamma_1\langle v_2, v_0\rangle }\otimes \Gamma_1\langle v_2, v_1\rangle )\circ (\Gamma_1\langle v_2,v_0\rangle \otimes f \otimes \Gamma_1\langle v_2, v_1\rangle ) \circ (\Gamma_1\langle v_2,v_0\rangle \otimes \Gamma_1\langle v_0,v_1\rangle \otimes i_{\Gamma_1\langle v_2, v_1\rangle}).
\end{align*}
Notice that
$$
\prescript{\Gamma_1\langle v_0, v_1\rangle \Gamma_1\langle v_1, v_0\rangle}{}{f} = f = \prescript{\Gamma_1\langle v_0, v_2\rangle \Gamma_1\langle v_2, v_0\rangle}{}{f}.
$$
This allows the definition of the 1-morphism~$\prescript{\Gamma_1(\gamma)}{}{f}$ for $\gamma$ the homotopy class of a path in the boundary of~$\langle v_0, v_1, v_2\rangle$ that ends at $v_0$ by decomposing $\gamma$ into edges.

\subsection{Associator state spaces}
Let $\Gamma$ be a $\mathcal{C}$-state  of $(K,\phi)$ taking values in a simplicial skeleton~$\mathcal{S}$. For any pointed ordered tetrahedron $(\langle v_0, v_1, v_2, v_3\rangle, \gamma_0, \gamma_1, \gamma_2,\gamma_3)$ of $K$, we write:
\begin{align*}
[(v_0v_1v_2)v_3, \gamma_1, \gamma_3] &=\prescript{\Gamma_1(\gamma_1)}{}{ \Gamma_2\langle v_0,v_2,v_3\rangle} \circ \left(\prescript{\Gamma_1(\gamma_3)}{}{\Gamma_2\langle v_0,v_1,v_2\rangle}\otimes \Gamma_1\langle v_2, v_3\rangle\right)\\
[v_0(v_1v_2v_3), \gamma_0, \gamma_2] &= \prescript{\Gamma_1(\gamma_2)}{}{\Gamma_2\langle v_0,v_1,v_3\rangle} \circ \left(\Gamma_1\langle v_0, v_1\rangle \otimes \prescript{\Gamma_1(\gamma_0)}{}{\Gamma_2\langle v_1, v_2, v_3\rangle}\right).
\end{align*}

Let  $\tau = (\langle v_0, v_1, v_2, v_3\rangle, \gamma_0, \gamma_1, \gamma_2,\gamma_3)$ be a pointed ordered tetrahedron of $K$. The \emph{positive and negative associator state spaces of $\tau$} are
\begin{align*}
V^+(\Gamma, \tau) &=\operatorname{Hom}^{ \phi_3(\tau)}_\mathcal{C}\! \Big([(v_0v_1v_2)v_3, \gamma_1, \gamma_3], [v_0(v_1v_2v_3), \gamma_0, \gamma_2]\Big)
\end{align*}
and 
\begin{align*}
V^-(\Gamma, \tau) &= \operatorname{Hom}^{ \phi_3(\tau)^{-1}}_\mathcal{C}\!\! \Big([v_0(v_1v_2v_3), \gamma_0, \gamma_2], [(v_0v_1v_2)v_3, \gamma_1, \gamma_3]\Big),
\end{align*}
respectively.  Elements in the associator state spaces are depicted graphically as in Figure \ref{State-spaces},  where we use the shorthand $[ij]$ to denote $\Gamma_1\langle v_i, v_j\rangle$, and $[ijk]$ to denote $\prescript{\Gamma_1(\gamma_l)}{}{\Gamma_2\langle v_i, v_j, v_k\rangle}$, where $\{i,j,k,l\} = \{0,1,2,3\}$. 

\begin{figure}[ht]
\begin{subfigure}{.3\textwidth}
  \centering
  \includegraphics[width=.8\linewidth]{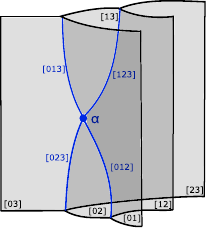}
  \caption*{$\alpha \in V^+(\Gamma, \tau)$}
  \label{fig:sfig1}
\end{subfigure}%
\begin{subfigure}{.3\textwidth}
  \centering
  \includegraphics[width=.8\linewidth]{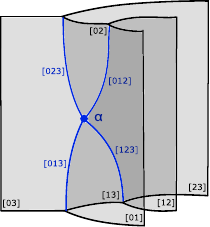}
  \caption*{$\alpha \in V^-(\Gamma, \tau)$}
  \label{fig:sfig2}
\end{subfigure}
\caption{Elements in the graded associator state spaces}
\label{State-spaces}
\end{figure}

\subsection{The canonical associated state}
Let $\Gamma$ be a $\mathcal{C}$-state of $(K,\phi)$ taking values in a simplicial skeleton~$\mathcal{S}$, and let $B_*$ be a choice of ordered simplices of $K$ (see Section \ref{Subsection_Choices of ordered simplices}).  The trace pairing
\begin{equation*}
\langle \cdot , \cdot \rangle_{\Gamma, \tau} \colon V^+(\Gamma, \tau) \otimes V^-(\Gamma, \tau) \to \mathbbm{k}
\end{equation*}
is nondegenerate for any pointed ordered tetrahedron $\tau$ (see Lemma \ref{The trace pairing is nondegenerate}).  Therefore, there is a canonical copairing
\begin{equation*}
\cup_{\Gamma, \tau} \colon \mathbbm{k} \to V^-(\Gamma, \tau) \otimes V^+(\Gamma, \tau). 
\end{equation*}

The \emph{canonical associated state} $\cup_\Gamma$ of $\Gamma$ is the vector
\begin{equation*}
\cup_{\Gamma} = \bigotimes_{\tau \in B_3} \cup_{\Gamma, \tau}(1)\in   \bigotimes_{\tau \in B_3} \left( V^+(\Gamma, \tau) \otimes V^-(\Gamma, \tau) \right),
\end{equation*}
where $\otimes$ denotes the unordered tensor product of~$\mathbbm{k}$-modules. 

\begin{figure}
\begin{subfigure}{.5\textwidth}
  \centering
  \includegraphics[width=0.9\linewidth]{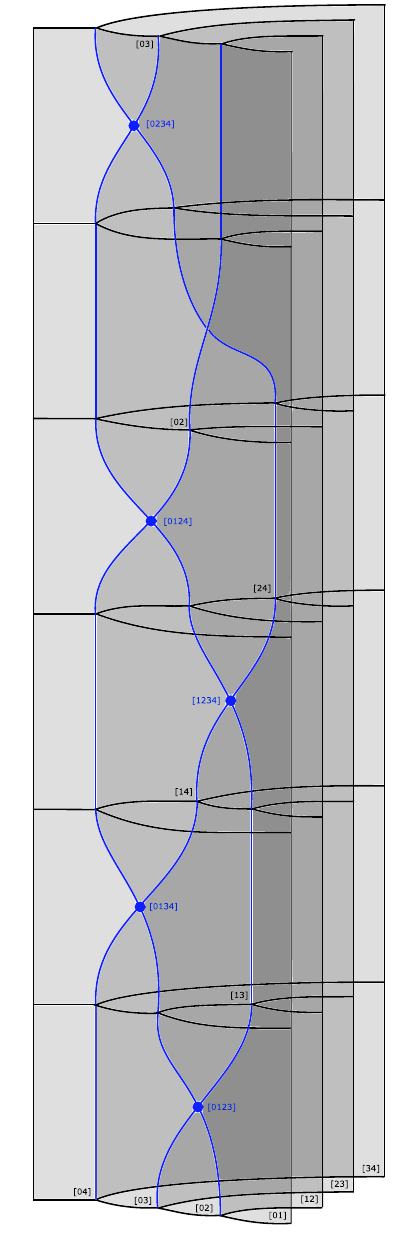}
\end{subfigure}%
\begin{subfigure}{.5\textwidth}
  \centering
  \includegraphics[width=0.9\linewidth]{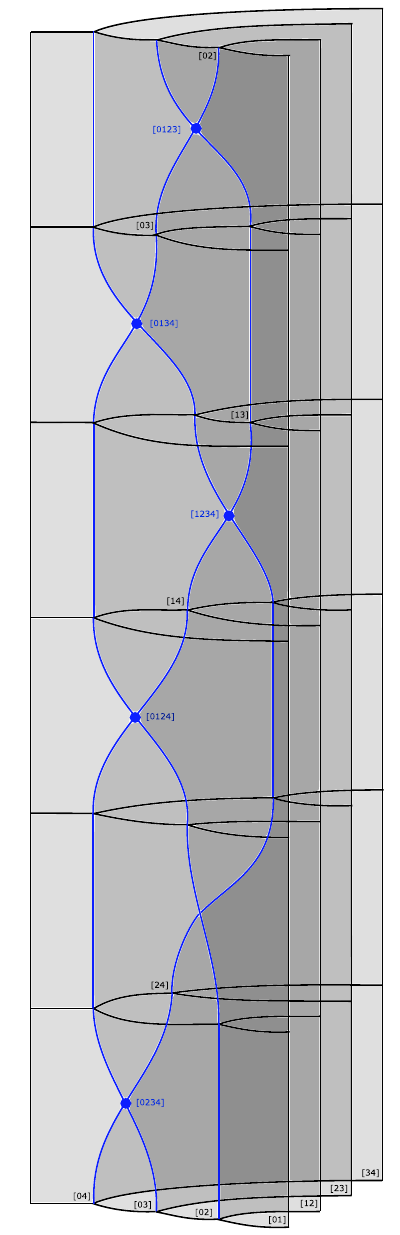}
\end{subfigure}
\caption{The colored 10j-symbols}
\label{10j-symbols}
\end{figure}

\subsection{The colored 10j-symbols}

Let $B_*$ be a choice of ordered simplices of $K$ (see Section \ref{Subsection_Choices of ordered simplices}) and let $\kappa\in B_4$.   We denote by $B_3\cap \kappa$ the set of pointed ordered tetrahedra in~$B_3$ that are included in $\kappa$.  The \emph{sign map}
$$
\varepsilon\colon B_3\cap \kappa \to \{+, -\}
$$
is defined as follows.  Using the linear order on the vertices of $\kappa$, we write $\kappa = \langle 0,1,2,3,4\rangle$.  Let $\tau \in B_3\cap \kappa$ and let~\mbox{$m \in \{0,1,2,3,4\}$} be the vertex of $\kappa$ that is not in $\tau$.  Let $\partial_m(\kappa)$ be the $m$-th face of $\kappa$ with pointing and order induced by restricting the pointing and order of $\kappa$.  Let $\mu$ be the permutation that sends the order of~$\partial_m(\kappa)$ to the order of $\tau$. Then the sign of $\tau$ is defined by
$$
\varepsilon(\tau) = (-1)^m\mathrm{s}(\mu),
$$
where $\mathrm{s}(\mu)$ is the signature of $\mu$.  Moreover, there are isomorphisms of associator state spaces 
$$
\phi_+ \colon V^{\varepsilon (\tau)}(\Gamma, \tau)\ \tilde{\to}\  V^{(-1)^{m}}(\Gamma, \partial_m(\kappa)) \quad \quad \phi_- \colon V^{-\varepsilon (\tau)}(\Gamma, \tau) \ \tilde{\to}\  V^{(-1)^{m+1}}(\Gamma, \partial_m(\kappa))
$$
induced by the isomorphisms in \cite[Lemma 4.2.4]{douglas_fusion_2018} and by the isomorphisms in Equation (\ref{twist_iso}) for the change of pointings. 

Let $\Gamma$ be a $\mathcal{C}$-state of $(K,\phi)$.  The \emph{colored 10j-symbol} associated to a pair $(\Gamma, \kappa)$, where $\kappa \in B_4$, is the linear form $z(\Gamma, \kappa)$ defined as follows.  If the orientation of $\kappa$ (induced by the order on its vertices) coincides with the orientation of $K$, then
$$
z(\Gamma, \kappa) \colon \bigotimes_{\tau \in B_3\cap \kappa} V^{\varepsilon(\tau)}(\Gamma, \tau) \to \mathbbm{k}
$$
is defined as the trace of the 2-morphism depicted on the left in Figure \ref{10j-symbols} composed with the isomorphisms $\phi_+$ described above, where $[ijkl]$ denotes an element of $V^{(-1)^{m}}(\Gamma, \partial_m(\kappa))$.  If the orientation of $\kappa$ (induced by the order on its vertices) does not coincide with the orientation of $K$, then
$$
z(\Gamma, \kappa) \colon \bigotimes_{\tau \in B_3\cap \kappa} V^{-\varepsilon(\tau)}(\Gamma, \tau) \to \mathbbm{k}
$$
is defined as the trace of the 2-morphism depicted on the right in Figure \ref{10j-symbols} composed with the isomorphisms $\phi_-$ described above, where $[ijkl]$ denotes an element of $V^{(-1)^{m+1}}(\Gamma, \partial_m(\kappa))$. 

\subsection{The 10j-action of the canonical associated state}
Let $\Gamma$ be a $\mathcal{C}$-state of $(K, \phi)$ taking values in a simplicial skeleton $\mathcal{S}$, and let $B_*$ be a choice of ordered simplices of $K$. The \emph{10j-action of the canonical associated state} of $\Gamma$ is
\begin{equation*}
Z(\Gamma) = \left( \bigotimes_{\kappa \in B_4} z(\Gamma, \kappa)\right) (\cup_\Gamma ) \in \mathbbm{k}.
\end{equation*}
Notice that for every $\tau\in B_3$, the associator state spaces $V^+(\Gamma, \tau)$ and $V^-(\Gamma, \tau)$ appear exactly once in the domain of $ \otimes_{\kappa \in B_4} z(\Gamma, \kappa)$, and therefore the composition is well defined. 

\section{The state sum for $\sigma$-manifolds}\label{Section. State sum}

\noindent Throughout this section, $\sigma$ is a 2-crossed module of groups (see Section \ref{Section. 2-crossed modules}),  $\mathcal{C}$ is a spherical~$\sigma$-fusion~2-category~(see Section \ref{Section. Graded-fusion 2-categories}) over the nonzero commutative ring $\mathbbm{k}$, and $\mathcal{S}$ is a simplicial skeleton of~$\mathcal{C}$~(see Section \ref{Section. 10j-symbols}). Moreover, we assume that $\mathcal{C}$ has invertible dimensions, which is automatically the case when $\mathbbm{k}$ is an algebraically closed field of characteristic 0 (see Section \ref{Dimensions of 2-categories}). 

We derive from this data an invariant of closed~4-manifolds endowed with a homotopy class of maps to the classifying space of $\sigma$, or equivalently of closed 4-manifolds equipped with a flat 3-bundle for the 3-group modeled by~$\sigma$, see~\cite{Cellot_thesis_2026}. It is defined as a state sum on a triangulation via the colored~10j-symbols.

\subsection{$\sigma$-manifolds}\label{sigma manifolds}

A \emph{$\sigma$-manifold} is a pair $(W,g)$ where $W$ is a closed oriented smooth 4-manifold and $g$ is a homotopy class of maps from $W$ to the classifying space $B\sigma$ (see Section \ref{Section_Classifying space}).  Two $\sigma$-manifolds $(W,g)$ and~$(W',g')$ are \emph{equivalent} if there is an orientation preserving diffeomorphism $\varphi \colon W \to W'$ such that $g' \circ \varphi =g$.  An \emph{invariant} of $\sigma$-manifolds is an invariant of equivalence classes of $\sigma$-manifolds. 

\subsection{Triangulations of $\sigma$-manifolds}\label{Section-Triangulations}

A \emph{combinatorial $4$-manifold} is a finite simplicial complex such that the link of every vertex is PL homeomorphic to the boundary of the standard $4$-simplex.  A \emph{cano\-nical~$4$-manifold} is a canonical simplicial complex (see Section \ref{Canonical simplicial complexes}) whose underlying simplicial complex is a combinatorial \mbox{$4$-mani}\-fold.  An \emph{orientation} of a canonical $4$-manifold $K$ is a choice of orientation of every $4$-simplex of $K$, such that for every tetrahedron of $K$, the orientations induced by the two adjacent~$4$-simplices are opposite.  A canonical $4$-manifold is \emph{oriented} if it is endowed with an orientation.

A map from a simplicial complex $K$ to a smooth manifold $W$ is a \emph{piecewise differentiable homeomorphism} if it is a homeomorphism and its restriction to every simplex of $K$ is a smooth immersion.  A \emph{triangulation} of a smooth manifold $W$ is a pair $(K, \varphi)$, where $K$ is a simplicial complex and $\varphi\colon \lvert K\rvert \to W$ is a piecewise differentiable homeomorphism.  It is a classical result, due to Cairns \cite{cairnsTriangulationManifoldClass1935} and Whitehead \cite{whitehead_c1-complexes_1940}, that any smooth manifold admits a triangulation.  Moreover,  if $(K, \varphi)$ is a triangulation of a smooth manifold, then $K$ is a combinatorial manifold. In fact, the triangulation of a smooth manifold is unique up to PL homeomorphism.  

A \emph{combinatorial $\sigma$-manifold} is a pair $(K,\phi)$, where $K$ is an oriented canonical \mbox{$4$-mani}\-fold, and $\phi$ is a $\sigma$-coloring of $K$. A \emph{$\sigma$-triangulation} (or simply a \emph{triangulation}) of a $\sigma$-manifold $(W,g)$ is a combinatorial $\sigma$-manifold $(K,\phi)$ together with a choice of ordered simplices $B_*$, such that $K$ is a triangulation of $W$ and $B\phi \simeq g$, where~$B\phi$ is the geometric realization of $\phi$ (see Section \ref{geometric realization}).  By Lemma \ref{Lemma Realizability}, every $\sigma$-manifold admits a triangulation.  

\subsection{The state sum}
Let $(W, g)$ be a $\sigma$-manifold. Pick a triangulation $(K,\phi, B_*)$ of $(W,g)$.  Consider the state sum
\begin{align*}
\mathrm{HDR}_\mathcal{C}(W,g) = \dim(\prescript{}{1}{\mathcal{C}}_1^1)^{-\lvert K_0\rvert} \sum_{\Gamma} \Big(\prod_{e \in B_1} \wdim(\Gamma_1(e))^{-1}\Big) \Big(\prod_{t \in B_2}\dim(\Gamma_2(t)) \Big) Z(\Gamma) \in \mathbbm{k},
\end{align*}
where $\Gamma$ runs over the $\mathcal{C}$-states of $(K,\phi)$ taking values in $\mathcal{S}$, $e$ runs over the edges of $K$, and $t$ runs over the triangles of $K$. Here, $\dim(\prescript{}{1}{\mathcal{C}}_1^1)$ denotes the dimension of the neutral component of $\mathcal{C}$ (see Section \ref{Dimensions of 2-categories}), $\wdim(\Gamma_1(e))$ denotes the weighted dimension of the object $\Gamma_1(e)$ (see Section \ref{Weighted dimension}), and $\dim(\Gamma_2(t))$ denotes the dimension of the~1-morphism $\Gamma_2(t)$ (see Section~\ref{Dimensions of objects and 1-morphisms}). 

\begin{thm}\label{maintheorem}
The state sum $\mathrm{HDR}_\mathcal{C}(W,g)$ is an invariant of $\sigma$-manifolds. 
\end{thm}

The invariance of the state sum relies on a $\sigma$-colored version of Pachner's theorem that we show in Section~\ref{Section-Pachner}.  We show in Section \ref{Ex.nontrivial} that this invariant is nontrivial. 

\subsection{Example: the Douglas--Reutter invariant}\label{Ex.Douglas-Reutter}

Let $(W,*)$ be a $\sigma$-manifold with $*$ the homotopy class of a constant map. Recall that the neutral component of $\mathcal{C}$ is the 2-category $\prescript{}{1}{\mathcal{C}_1^1}$ with only degree-1 homogeneous objects, 1-morphisms,  and 2-morphisms (see Section \ref{Section-neutral component}).  Notice that $\prescript{}{1}{\mathcal{C}_1^1}$ is a fusion 2-category in the sense of Douglas and Reutter (up to additive and idempotent completion), and
$$
\mathrm{HDR}_\mathcal{C}(W,*) = \mathrm{DR}_{\prescript{}{1}{\mathcal{C}}^1_1}(W)
$$
where $\mathrm{DR}_{\prescript{}{1}{\mathcal{C}}^1_1}(W)$ is the Douglas--Reutter invariant of $W$ defined using $\prescript{}{1}{\mathcal{C}_1^1}$.  In particular, if $\sigma$ is trivial, then 
$$
\mathrm{HDR}_\mathcal{C}(W,*) = \mathrm{DR}_{\mathcal{C}}(W)
$$
since $\mathcal{C} = \prescript{}{1}{\mathcal{C}_1^1}$.

\subsection{Example: nontriviality of the homotopy invariant}\label{Ex.nontrivial}In this section we fix $\mathbbm{k} = \mathbb{Q}$, and we set 
$$
\begin{tikzcd}
\sigma = (\Z/2\Z \arrow[r, "0"] & 0 \arrow[r, "0"] & 0, \ 0).
\end{tikzcd}
$$
The classifying space $B\sigma$ is a $K(\Z/2\Z, 3)$ space.  Let $W = (S^1)^4$ be the product of four circles.  Then the set~$[W, B\sigma]$ is in bijection with~$\mathrm{H}^3(W; \mathbb{Z}/2\mathbb{Z}) \cong (\mathbb{Z}/2\mathbb{Z})^4$. Let $*$ be the trivial class in $[W, B\sigma]$, and pick a nontrivial class $g$.  We construct a $\sigma$-fusion 2-category $\mathcal{C}$ such that 
$$
\mathrm{HDR}_\mathcal{C}(W,*) \neq \mathrm{HDR}_\mathcal{C}(W,g).
$$
Let 
$$
\begin{tikzcd}
\sigma' = (\mathbb{Z}/2\mathbb{Z} \arrow[r, "\mathrm{Id}"] & \mathbb{Z}/2\mathbb{Z} \arrow[r] & 0), 
\end{tikzcd}
$$
and let $\phi = (\phi_1, \phi_2, \phi_3) \colon \sigma' \to \sigma$ be the 2-crossed module morphism with $\phi_3 = \mathrm{Id}$.  Let $\mathcal{G}_{\sigma'}$ be the linearization of $\sigma'$ (see Section~\ref{Ex. associated fusion}), and let $\mathcal{C} = \phi_*(\mathcal{G}_{\sigma'})$ be the pushforward of $\mathcal{G}_{\sigma'}$ (see Section~\ref{Ex. pushforward fusion}).  

We first compute $\mathrm{HDR}_\mathcal{C}(W,*)$.  We have 
$$
\dim(\mathcal{C}_1^1(\mathbbm{1}, \mathbbm{1})) = 2\quad \text{and} \quad \dim(\prescript{}{1}{\mathcal{C}_1^1}) = \frac{1}{2}.
$$
Let $K$ be a triangulation of $W$,  and let $1$ be the trivial $\sigma$-coloring of $K$. Then $(K,1)$ is a $\sigma$-triangulation of~$(W,*)$. Notice that for any $\mathcal{C}$-state $\Gamma$ of $(K,1)$, we have $Z(\Gamma) =1$ if $\Gamma$ induces a $\rho$-coloring of $K$, where
$$
\begin{tikzcd}
\rho = (0\arrow[r, "0"] & \Z/2\Z\arrow[r, "0"] & 0, \ 0),
\end{tikzcd}
$$
and $Z(\Gamma) = 0$ otherwise. Therefore,
$$
\mathrm{HDR}_\mathcal{C}(W,*) = 2^{\lvert K^{(0)}\rvert - \lvert K^{(1)}\rvert} \#\{\rho\text{-colorings of $K$}\}.
$$
In fact, it is a consequence of Section~\ref{Ex.Douglas-Reutter} that~$\mathrm{HDR}_\mathcal{C}(W,*)$ coincides with the (untwisted) Yetter--Dijkgraaf--Witten invariant of $W$ associated to the crossed module $\chi = (\Z/2\Z \to 0)$.  Therefore, by \cite[Theorem~1.21]{martins_yetters_2007}, 
$$
\mathrm{HDR}_\mathcal{C}(W,*) =\sum_{g \in [W, B\chi]} \frac{\lvert\pi_2(\mathrm{TOP}(W, B\chi), g)\rvert}{\lvert\pi_1(\mathrm{TOP}(W, B\chi), g)\rvert},
$$
where $B\chi$ denotes the classifying space of $\chi$ (see \cite{brown_nonabelian_2011}) and $\mathrm{TOP}(W, B\chi)$ denotes the space of maps $W\to B\chi$ with the compact-open topology. For~$n= 1, 2$ and $g \in [W, B\chi]$, we have
\begin{align*}
\pi_n(\mathrm{TOP}(W, B\chi),g) & \cong \pi_n(\mathrm{TOP}(W, B\chi),*)\\
&\cong [S^n, \mathrm{TOP}(W, B\chi)]_*\\
&\cong [(S^n\times W, \{*\}\times W), (B\chi, *)]\\
&\cong H^2(S^n\times W, \{*\}\times W; \Z/2\Z).
\end{align*}
Using the relative Künneth formula, 
$$
\mathrm{HDR}_\mathcal{C}(W,*) = \lvert \Z/2\Z\rvert^{6}\frac{\lvert\Z/2\Z\rvert}{\lvert \Z/2\Z\rvert^{4}} = 8.
$$

We now compute $\mathrm{HDR}_\mathcal{C}(W,g)$.  Let $(K,\varphi)$ be a $\sigma$-triangulation of~$(W,g)$.  The state sum $\mathrm{HDR}_{\mathcal{C}}(W,g)$ is a sum over the $\mathcal{C}$-states of $(K,\varphi)$. By considering the degrees in $\sigma'$, any $\mathcal{C}$-state of $K$ induces a $\sigma'$-coloring $\widetilde{\varphi}$ of~$K$ such that the diagram 
$$
\begin{tikzcd}
& \sigma' \arrow[d, "\phi"]\\
\Pi_3(K) \arrow[ur, dashed, "\widetilde{\varphi}"] \arrow[r, "\varphi"] & \sigma
\end{tikzcd}
$$
commutes.  Since $B\sigma'$ is contractible and $g$ is not homotopically trivial, no such lifts can exist, and therefore
$$
\mathrm{HDR}_{\mathcal{C}}(W,g) = 0.
$$
Therefore, $\mathrm{HDR}_{\mathcal{C}}(W,g) \neq \mathrm{HDR}_{\mathcal{C}}(W,*)$. Also, all the maps $W \to B\sigma$ are phantom maps (meaning that they induce trivial homomorphisms on homotopy groups), because $\pi_3(W) = 0$ and $B\sigma$  is a $K(\Z/2\Z, 3)$ space.  This shows that the invariant $\mathrm{HDR}_\mathcal{C}$ is nontrivial and can distinguish homotopy classes of phantom maps.

\subsection{Relation with other invariants}

Let $G$ be a discrete group. In \cite{mochidaInvariantsFlatConnections2026}, Mochida uses the data of a finite-type involutory quasitriangular Hopf $G$-algebra to derive invariants of flat $G$-bundles over closed oriented smooth 4-manifolds. These invariants are obtained via a 4-dimensional analogue of the Hennings-type construction using~\mbox{$G$-colored} Kirby diagrams.  Notice that the data of a flat $G$-bundle over a closed oriented smooth 4-manifold is equivalent to that of a $(1\to 1\to G)$-manifold.  We expect that in certain cases,  the delooping of the category of representations of an involutory quasitriangular Hopf $G$-algebra should yield a spherical $\sigma$-fusion 2-category. Then, we conjecture that the invariants of Mochida could  be recovered as the state-sum invariants from this $\sigma$-fusion 2-category. 

The invariants of Mochida have been generalized by Bridges and Cui \cite{bridges_involutory_2025} using the data of a finite-type involutory Hopf $G$-triplet.  Their invariants are constructed via a 4-dimensional analogue of the Kuperberg-type construction via $G$-colored trisection diagrams. We also expect that in certain cases, these should be recovered by our state-sum construction. 

\subsection{Extension to an HQFT}

We expect our state-sum invariants to fit into a 4-dimensional HQFT with target the classifying space $B\sigma$.  The $\mathbbm{k}$-module associated to a closed oriented smooth 3-manifold equipped with a homotopy class of maps to $B\sigma$ remains to be determined.

\section{Colored Pachner moves}\label{Section-Pachner}

\noindent In his seminal paper \cite{pachner_pl_1991}, Pachner states a purely combinatorial theorem that plays a central role in low-dimensional topology:

\begin{theorem}[Pachner Theorem]
Two combinatorial manifolds are PL homeomorphic if and only if they are related by a finite sequence of Pachner moves. 
\end{theorem}

Pachner moves (also called bistellar moves) are local changes in the triangulation of a simplicial complex. For any combinatorial~$n$-manifold, there are only~$n+1$ different types of Pachner moves.  Therefore, in order to check that some quantity is a diffeomorphism invariant of closed 4-manifolds,  it suffices to check that it does not change under the (finite list of) Pachner moves.  

We are interested in invariants of $\sigma$-manifolds (see Section \ref{sigma manifolds}) that are modeled using combinatorial \mbox{$\sigma$-mani}folds (see Section \ref{Section-Triangulations}).  We say that two combinatorial $\sigma$-manifolds $(K,\phi)$, $(L, \psi)$ are \emph{equivalent} if there is a PL homeomorphism $\varphi \colon K \to L$ such that $B\phi$ and $B\psi \circ \varphi$ are homotopic, where $B\phi$ and $B\psi$ are the geometric realizations of $\phi$ and $\psi$, respectively (see Section \ref{geometric realization}).  We show a strengthening of Pachner's theorem:

\begin{thm}[Colored Pachner theorem]\label{Theorem - Colored Pachner}
Two combinatorial $\sigma$-manifolds are equivalent if and only if they are related by a finite sequence of $\sigma$-colored Pachner moves.
\end{thm}

The aim of this section is to prove Theorem \ref{Theorem - Colored Pachner}. We fix\begin{tikzcd}\sigma = (L \arrow[r, "\delta"] & E \arrow[r, "\partial"] & H, \ \omega)\end{tikzcd}a~2-crossed module of groups.

\subsection{4-dimensional Pachner moves} An \emph{$n$-dimensional Pachner move} (or simply a \emph{Pachner move}) replaces a combinatorial $n$-manifold $K$ with a combinatorial manifold $K'$ obtained by replacing a submanifold of $K$ that is isomorphic to a subcomplex $I\subset \partial \Delta^{n+1}$ with the complementary subcomplex $J\subset \partial \Delta^{n+1}$.  There are five types of 4-dimensional Pachner moves denoted by the cardinality of the 4-simplices in $I$ and $J$: the Pachner~$(1,5)$-move, the Pachner $(2,4)$-move, and the Pachner $(3,3)$-move, together with their inverses.  Let~$0,\ldots, 5$ denote the vertices of $\Delta^5$. The following table lists the 4-simplices in $I$ and $J$ for the 4-dimensional Pachner moves:
\begin{center}
\begin{tabular}{c | c | c}
& $I$ & $J$ \\
\hline
Pachner $(1,5)$-move & $(01234)$ & $(01235)$, $(01245)$, $(01345)$, $(02345)$, $(12345)$ \\
\hline 
Pachner $(2,4)$-move & $(01234)$, $(12345)$ & $(01235)$, $(01245)$, $(01345)$, $(02345)$\\
\hline
Pachner $(3,3)$-move & $(01234)$, $(01245)$, $(02345)$ & $(01235)$, $(01345)$, $(12345)$
\end{tabular}
\end{center}
We refer the reader to \cite{lickorish_simplicial_1999} for more details on Pachner moves. 

\subsection{Canonical Pachner moves}

A \emph{canonical Pachner move} on a canonical 4-manifold $K$ is a Pachner move on the underlying combinatorial 4-manifold of $K$ together with an arbitrary choice of distinguished vertex for every newly constructed triangle.  Clearly,  two canonical $4$-manifolds (see Section \ref{Section-Triangulations}) with the same underlying combinatorial 4-manifold are related by a finite sequence of canonical Pachner moves. 

\subsection{Colored Pachner moves}
A \emph{$\sigma$-colored Pachner move} (or simply a \emph{colored Pachner move}) replaces a combinatorial~$\sigma$-manifold $(K,\phi)$ with a combinatorial $\sigma$-manifold $(L,\psi)$, where $L$ is obtained by applying a canonical Pachner move to $K$, and $\psi$ coincides with $\phi$ outside of the new simplices. 

Let $(K, \phi)$ be a combinatorial $\sigma$-manifold, and let $L$ be obtained by applying a canonical Pachner $(1,5)$-move to $K$.  In other words, there is a 4-simplex $(01234)$ in $K$ such that $L$ is obtained from $K$ by replacing~$(01234)$ with the 4-simplices $(01235)$, $(01245)$, $(01345)$, $(02345)$,  and $(12345)$. Notice that all of the 0-, 1-, 2-, and 3-simplices of $K$ are also in $L$.  

\begin{lem}\label{Pachner_(1,5)}
If $L$ is the canonical 4-manifold obtained by applying a canonical Pachner $(1,5)$-move to a canonical 4-manifold $K$ with $\sigma$-coloring $\phi$, then a $\sigma$-coloring $\psi$ of $L$ that coincides with $\phi$ on all pointed ordered 1-, 2-, and 3-simplices of $K$ is uniquely specified by choosing:
\begin{itemize}
\item an element $h\in H$ such that $\psi(\langle 05\rangle) = h$,
\item elements $e_1, e_2, e_3, e_4 \in E$ such that 
\begin{align*}
\psi(\langle 015\rangle) &= e_1, &
\psi(\langle 025\rangle ) &= e_2,\\
\psi(\langle 035\rangle ) &= e_3, &
\psi(\langle 045\rangle ) &= e_4,
\end{align*}
\item elements $l_1, l_2, l_3, l_4, l_5, l_6\in L$ and choices of pointings and orders $\tau_1, \tau_2, \tau_3, \tau_4$,  $\tau_5$, and $\tau_6$ on the simplices $(0125)$, $(0135)$, $(0145)$,  $(0235)$,  $(0245)$ and $(0345)$, respectively, such that $\psi(\tau_i) = l_i$ for all~$i\in \{1,2, 3, 4, 5, 6\}$,
\end{itemize}
with no compatibility conditions. Moreover, there is a bijection between the set of $\sigma$-colorings of $L$ and the set
\[
\{ \text{$\sigma$-colorings of $K$}\}\times H\times E^4\times L^6.
\]
In particular, any $\sigma$-coloring of $L$ induces a $\sigma$-coloring of $K$ by restriction.
\end{lem}

We give a proof of Lemma \ref{Pachner_(1,5)} in Section \ref{Section_Proof of Lemma 8.3.1}. The proofs of Lemmas \ref{Pachner_(2,4)} and \ref{Pachner_(3,3)} are very similar and are left to the reader. 

Suppose $K$ and $L$ are canonical 4-manifolds such that $L$ is obtained from $K$ by applying a canonical Pachner~$(2,4)$-move.  In other words, there are two 4-simplices $(01234)$ and~$(12345)$ in $K$ such that $L$ is obtained from $K$ by replacing them with the 4-simplices $(01235)$, $(01245)$, $(01345)$, and~$(02345)$.  Notice that the only simplices of $K$ that are not in $L$ are $(1234)$, $(01234)$, and $(12345)$. 

\begin{lem}\label{Pachner_(2,4)}
If $L$ is the canonical 4-manifold obtained by applying a canonical Pachner $(2,4)$-move to a canonical 4-manifold $K$ with $\sigma$-coloring $\phi$, then a $\sigma$-coloring $\psi$ of $L$ that coincides with $\phi$ on all pointed ordered 1-, 2-, and 3-simplices common to $K$ and $L$ is uniquely specified by choosing:
\begin{itemize}
\item an element $e\in E$ such that $\psi(\langle 015\rangle) = e$,
\item elements $l_1,l_2,l_3\in L$ and choices of pointings and orders $\tau_1$, $\tau_2$,  and $\tau
_3$ on the simplices $(0125)$, $(0135)$, and $(0145)$, respectively, such that $\psi(\tau_i) = l_i$ for $i\in \{1,2,3\}$,
\end{itemize}
with no compatibility conditions.  Moreover, there is a bijection between the set of $\sigma$-colorings of $L$ and the set 
\[
\{ \text{$\sigma$-colorings of $K$}\}\times E\times L^3.
\]
In particular, any $\sigma$-coloring of $L$ induces a $\sigma$-coloring of $K$ by restriction.
\end{lem}

Suppose $K$ and $L$ are canonical 4-manifolds such that $L$ is obtained from~$K$ by applying a canonical Pachner~$(3,3)$-move.  In other words, there are three 4-simplices $(01234)$,~$(01245)$, and $(02345)$ in $K$ such that $L$ is obtained from $K$ by replacing them with the 4-simplices $(01235)$, $(01345)$, and~$(12345)$. 

\begin{lem}\label{Pachner_(3,3)}
If $L$ is the canonical 4-manifold obtained by applying a canonical Pachner $(3,3)$-move to a canonical 4-manifold $K$ with $\sigma$-coloring $\phi$, then a $\sigma$-coloring $\psi$ of $L$ that coincides with $\phi$ on all simplices common to $K$ and $L$ is uniquely specified by choosing an element $l\in L$ and a pointing and order $\tau$ on $(0135)$ such that~$\psi(\tau) = l$.
\end{lem}

\subsection{Proof of Theorem \ref{Theorem - Colored Pachner}}\label{proof colored Pachner}

This section is devoted to the proof of the Colored Pachner Theorem. Let $(K, \phi)$ and $(L, \psi)$ be combinatorial $\sigma$-manifolds. Our aim is to show that $(K, \phi)$ and $(L, \psi)$ are equivalent if and only if $(L, \psi)$ is obtained by applying a finite sequence of $\sigma$-colored Pachner moves to $(K,\phi)$.  If $(L,\psi)$ is obtained by applying a~$\sigma$-colored Pachner move to $(K, \phi)$, then they are equivalent.  Indeed, in that case there exists a PL homeomorphism~$\varphi\colon K\to L$ such that the induced maps $B\phi$ and $B\psi\circ \varphi$ are homotopic because $\pi_4(B\sigma) = \{0\}$. 

Conversely, let us show that if $(K, \phi)$ and $(L, \psi)$ are equivalent, then $(L,\psi)$ is obtained by applying a finite sequence of $\sigma$-colored Pachner moves to $(K,\phi)$. Notice that by the (classical) Pachner Theorem,  if $(K,\phi)$ and~$(L,\psi)$ are equivalent, then in particular $L$ is obtained by applying a finite sequence of Pachner moves to~$K$. Therefore, without loss of generality, we will suppose that $K = L$.  

Moreover,  the Homotopy Classification Theorem (Theorem \ref{homotopy classification theorem}) tells us that if $(K, \phi)$ and $(K, \psi)$ are equivalent then~$\phi$ and $\psi$ are gauge equivalent.  Therefore, it suffices to show that if $\phi$ and $\psi$ are gauge equivalent then~$(K, \psi)$ is obtained by applying a finite sequence of $\sigma$-colored Pachner moves to $(K, \phi)$.  

\begin{lem}\label{Lemma gauge-groupoid generators}
Let $(r,s,t)$ be a morphism in the gauge groupoid $\mathcal{G}(K,\sigma)$ (see Section \ref{gauge-groupoid}), and let $B_*$ denote a choice of ordered simplices of $K$ (see Section \ref{Subsection_Choices of ordered simplices}). There are non-negative integers $l, m, n$ such that 
$$
(r,s,t) = (r_1,1,1)\cdots (r_l, 1, 1)\cdot  (1, s_1, 1)\cdots (1, s_m, 1)\cdot (1,1,t_1)\cdots (1,1,t_n)
$$
where 
\begin{itemize}
\item for all $i\in \{1, \ldots,  l\}$,  there exists a unique $x\in K_0$ such that $r_i(x) \neq 1$, 
\item for all $j\in \{1, \ldots, m\}$, there exists a unique $b_1 \in B_1$ such that $s_j(b_1) \neq 1$,
\item for all $k\in \{1, \ldots, n\}$,  there exists a unique $b_2\in B_2$ such that $t_k(b_2)\neq 1$.
\end{itemize}
\end{lem}

Lemma \ref{Lemma gauge-groupoid generators} may seem obvious, but is not at all immediate.  Indeed, recall that the composition of morphisms in $\mathcal{G}(K,\sigma)$ makes use of the connecting map $\omega^{(s,s')}$, see Section \ref{gauge-groupoid}.

\begin{proof}
Notice that for all $b_2\in B_2$,  we have $\omega^{(s,1)}(d_1(b_2)) = 1$. Therefore,
$$
(r,s,t) = (r,1,1)\cdot (1, \prescript{r^{-1}}{}{s}, 1)\cdot (1, 1, \prescript{r^{-1}}{}{t}).
$$
Moreover, for all $r, r' \colon K_0 \to H$ and $t, t'\colon B_2\to L$, we have
$$
(r, 1, 1)\cdot (r', 1, 1) = (r r', 1, 1)\quad \text{and}\quad (1,1,t) \cdot (1,1,t') = (1, 1, t t').
$$
Finally, for all $s,s'\colon B_1\to E$, we have
$$
(1,s,1)\cdot (1,s', 1) = (1, s\otimes s', 1)\cdot (1,1,\omega^{(s,s')}(d_1)).
$$
Therefore, we can iteratively apply the previous formulas to write $(r,s,t)$ as a finite composition satisfying the required conditions. 
\end{proof}

We say that a quadratic derivation that is nontrivial on a unique element $b$ of $B_*$ is a \emph{generating quadratic derivation} on $b$.  Lemma \ref{Lemma gauge-groupoid generators} tells us that any quadratic $\phi$-derivation is the composition of generating quadratic derivations. We deduce that it suffices to show that if $\psi = \phi\cdot (r,s,t)$ with $(r,s,t)$ a generating quadratic derivation, then $(K,\psi)$ is obtained by applying a finite sequence of $\sigma$-colored Pachner moves to $(K,\phi)$. 

Suppose $\psi = \phi\cdot (r,s,t)$ with $(r,s,t)$ a generating quadratic derivation on $b\in B_*$. Without loss of generality, we can suppose that $b$ has a vertex $x$ whose link is $\partial \Delta^4$.  Indeed,  since $K$ is a combinatorial~4-manifold, we know that the link of any vertex of $K$ is PL-homeomorphic to $\partial \Delta^4$.  Therefore, by the classical Pachner Theorem,~$\partial \Delta^4$ is obtained by applying a finite sequence of 3-dimensional Pachner moves to the link of $x$.  In fact, these moves can be extended to a sequence of 4-dimensional Pachner moves denoted $M$ such that the link of $x$ in~$M\cdot K$ is~$\partial \Delta^4$, where $M\cdot K$ denotes the combinatorial manifold obtained by applying the moves $M$ to~$K$.  If we arbitrarily extend the moves $M$ to be $\sigma$-colored Pachner moves, our problem amounts to finding a finite sequence of $\sigma$-colored Pachner moves from $(M\cdot K, M\cdot \phi)$ to $(M\cdot K, M\cdot (\phi \cdot (r,s,t)))$, where $M\cdot \phi$ denotes the~$\sigma$-coloring of $M\cdot K$ obtained by applying the moves $M$.  We know that $M\cdot \phi$ and $M\cdot (\phi\cdot (r,s,t))$ are gauge equivalent, therefore there exists a quadratic $(M\cdot \phi )$-derivation $(r',s',t')$ such that 
$$
(M\cdot \phi)\cdot (r',s',t') = M\cdot (\phi \cdot (r,s,t)),
$$
and the maps $r'$, $s'$ and $t'$ are trivial on all basis elements outside of the star of $x$.  Moreover, $(r',s',t')$ is a composition of generating quadratic derivations over simplices that all have a vertex whose link is $\partial \Delta^4$.  Therefore, if we can find a finite sequence of $\sigma$-colored Pachner moves $N$ such that $N\cdot (M\cdot \phi) = (M\cdot \phi) \cdot (r',s',t')$, we will have
$$
M^{-1}NM\cdot \phi = M^{-1}\cdot ((M\cdot \phi) \cdot (r',s',t')) = M^{-1}\cdot (M\cdot (\phi \cdot (r,s,t))) = \phi \cdot (r,s,t).
$$

Now suppose $(r,s,t)$ is a generating quadratic $\phi$-derivation on a simplex $b$ that has a vertex $x$ whose link is $\partial \Delta^4$. We denote the 3-simplices of $\partial \Delta^4$ as $\langle 1234\rangle$, $\langle 1235\rangle$, $\langle 1245\rangle$, $\langle 1345\rangle$, and $\langle 2345\rangle$.  We give an explicit sequence of $\sigma$-colored Pachner moves going from $(K,\phi)$ to $(K,\phi\cdot (r,s,t))$. In the following, we do not specify the distinguished vertices and pointings of simplices to simplify notation.

\textbf{Case 1:} $(r,s,t) = (r,1,1)$. In this case, $r(x') = 1$ for all $x'\in K_0$ such that $x'\neq x$.  
\begin{enumerate}
\item Apply the $\sigma$-colored Pachner $(1,5)$-move to $\langle x1234\rangle$. This yields the 4-simplices $\langle xy123\rangle$, $\langle xy124\rangle$, $\langle xy134\rangle$, $\langle xy234\rangle$, and $\langle y1234\rangle$ with $\sigma$-coloring defined by choosing $\phi\langle xy\rangle = r(x)$, $\phi\langle xyi\rangle =1$ for all $i\in\{1,2,3,4\}$, and $\phi\langle xyij\rangle =1$ for all $i,j \in \{1,2,3,4\}$ such that $i<j$.
\item Apply the $\sigma$-colored Pachner $(2,4)$-move to $\langle x1235\rangle$ and $\langle xy123\rangle$. This yields the 4-simplices $\langle xy125\rangle$, $\langle xy135\rangle$, $\langle xy235\rangle$, and $y1235\rangle$ with $\sigma$-coloring defined by choosing $\phi\langle xy5\rangle =1$ and $\phi\langle xyi5\rangle = 1$ for $i\in\{1,2,3\}$.
\item Apply the $\sigma$-colored Pachner $(3,3)$-move to $\langle x1245\rangle$, $\langle xy124\rangle$, and $\langle xy125\rangle$. This yields the 4-simplices $\langle xy145\rangle$, $\langle xy245\rangle$, and $\langle y1245\rangle$ with $\sigma$-coloring defined by choosing $\phi\langle xy45\rangle = 1$.
\item Apply the $\sigma$-colored Pachner $(4,2)$-move to $\langle xy134\rangle$, $\langle xy135\rangle$, $\langle xy145\rangle$, and $\langle x1345\rangle$. This yields the 4-simplices $\langle xy345\rangle$ and $\langle y1345\rangle$ with no new choices for the $\sigma$-coloring.
\item Apply the $\sigma$-colored Pachner $(5,1)$-move to $\langle x2345\rangle$, $\langle xy234\rangle$, $\langle xy235\rangle$, $\langle xy245\rangle$, and $\langle xy345\rangle$. This yields the 4-simplex $\langle y2345\rangle$ with no new choices for the $\sigma$-coloring.
\end{enumerate}
The resulting combinatorial $\sigma$-manifold is exactly $(K, \phi\cdot (r,1,1))$ (with the vertex $x$ labeled $y$).

\textbf{Case 2:} $(r,s,t) = (1,s,1)$. In this case, $s(b') = 1$ for all $b' \in B_1$ such that $b'\neq b$. We suppose for instance that~$b =\langle x1\rangle$. We apply the same sequence of $\sigma$-colored Pachner moves as in Case 1, with the $\sigma$-colorings determined by the following choices: $\phi\langle xy\rangle =1$, $\phi\langle xy1\rangle = \prescript{\phi\langle x1\rangle}{}{s(\langle x1\rangle)^{-1}}$, $\phi\langle xyi\rangle = 1$ for all $i\in \{2,3,4,5\}$, and~$\phi\langle xyij\rangle = 1$ for all $i,j\in \{1,2,3,4,5\rangle$ such that $i<j$. The resulting $\sigma$-colored combinatorial manifold is exactly $(K, \phi\cdot (1,s,1))$ (with the vertex $x$ labeled $y$).

\textbf{Case 3:} $(r,s,t) = (1,1,t)$. In this case, $t(b') = 1$ for all $b' \in B_2$ such that $b' \neq b$. We suppose for instance that~$b = \langle x12\rangle$.  We apply the same sequence of $\sigma$-colored Pachner moves as in Case 1, with the $\sigma$-colorings determined by the following choices: $\phi\langle xy\rangle = 1$, $\phi\langle xyi\rangle =1$ for all $i\in \{1,2,3,4,5\}$, 
$$
\phi\langle xy12\rangle = \phi\langle x12\rangle \triangleright t(\langle x12\rangle) \quad \text{and}\quad \phi\langle xyij\rangle =1 
$$
for all $i,j\in \{1,2,3,4,5\}$ such that $i<j$ and $(i,j) \neq (1,2)$. The resulting combinatorial $\sigma$-manifold is exactly~$(K, \phi\cdot (1,1,t))$ (with the vertex $x$ labeled $y$).

\section{Proof of Theorem \ref{maintheorem}}\label{Section. Proof of main theorem}

\noindent The aim of this section is to prove Theorem \ref{maintheorem}, that is, to show that the state sum is an invariant of $\sigma$-manifolds.  In Section \ref{Invariance simplicial skeleton} we show that the state sum does not depend on the choice of simplicial skeleton of $\mathcal{C}$. In Section \ref{Invariance ordered simplices} we show that it does not depend on the choice of ordered simplices.  In Section \ref{Section_State sum for manifolds with boundary} we extend the state sum to combinatorial $\sigma$-manifolds with boundary. In Section \ref{Invariance triangulation} we show that it does not depend on the choice of triangulation. 

\subsection{The choice of simplicial skeleton}\label{Invariance simplicial skeleton}

Let $(K,\phi)$ be a combinatorial $\sigma$-manifold with choice of ordered simplices $B_*$.   Two $\mathcal{C}$-states $\Gamma$ and $\Gamma'$ of $(K,\phi)$ are \emph{equivalent} if for every edge $e \in B_1$, there are inverse~$(1,1)$-equivalences
$$
h_e \colon \Gamma(e) \rightleftarrows \Gamma'(e)\colon k_e
$$
and for every triangle $\langle ijk\rangle\in B_2$ with distinguished vertex $i$, there are degree-1 homogeneous 2-isomorphisms
$$
\Gamma\langle ijk\rangle \cong k_{\langle ir\rangle}\circ \Gamma'(t) \circ (h_{\langle ij\rangle } \otimes \Gamma'\langle jk \rangle) \circ (\Gamma\langle ij\rangle\otimes h_{\langle jk \rangle}),
$$
where $h_{\langle ab\rangle }$ denotes the right mate of $k_{\langle ba\rangle}$ if $\langle ba \rangle \in B_1$, and $k_{\langle ab\rangle }$ denotes the right mate of $h_{\langle ba \rangle } $ if $\langle ba \rangle \in B_1$ (see \cite[Definition 2.2.1]{douglas_fusion_2018})

Given two simplicial skeletons $\mathcal{S}$ and $\mathcal{S}'$ of $\mathcal{C}$, for each object $X$ in $\mathcal{S}$, choose inverse $(1,1)$-equivalences
$$
h_X \colon X \leftrightarrows X'\colon k_X
$$
between $X$ and the unique $(1,1)$-equivalent object $X'$ in $\mathcal{S}'$. Similarly, for every 1-morphism $g\colon X\otimes Y \to Z$ in~$\mathcal{S}$, choose a degree-1 homogeneous isomorphism between $h_Z \circ g \circ (k_X\otimes Y) \circ (X'\otimes k_Y)$, and the unique \mbox{degree-1} homogeneous 1-morphism $g'\colon X'\otimes Y' \to Z'$ in $\mathcal{S}$ isomorphic to it.  This construction provides a bijection between the~$\mathcal{C}$-states with values in $\mathcal{S}$ and the~$\mathcal{C}$-states with values in $\mathcal{S}'$, and this bijection takes each state to an equivalent state.  Therefore,  in order to show that the state sum is independent of the choice of simplicial skeleton, it is sufficient to show the following:

\begin{lemref}[\cite{douglas_fusion_2018}]
If two $\mathcal{C}$-states $\Gamma$ and $\Gamma'$ of $(K,\phi)$ are equivalent, then
$$
\Big(\prod_{e\in B_1} \wdim (\Gamma(e))^{-1}\Big) \Big( \prod_{t\in B_2}\dim (\Gamma(t))\Big) Z(\Gamma) = \Big(\prod_{e\in B_1} \wdim (\Gamma'(e))^{-1}\Big) \Big(\prod_{t\in B_2}\dim (\Gamma'(t))\Big) Z(\Gamma'). 
$$
\end{lemref}

\subsection{The choice of ordered simplices}\label{Invariance ordered simplices}

Let $K$ be a canonical 4-manifold and let $\phi$ be a $\sigma$-coloring of $K$.  We would like to show that the state sum does not depend on the choice of ordered simplices of $K$.  It is sufficient to show the following:

\begin{lem}\label{Lemma choice of ordered simplices}
If $B_*$ and $B_*'$ are choices of ordered simplices of $K$ and $\Gamma$ is a $\mathcal{C}$-state of $(K,\phi)$, then
$$
\Big(\prod_{e\in B_1} \wdim (\Gamma(e))^{-1}\Big) \Big( \prod_{t\in B_2}\dim (\Gamma(t))\Big) Z(\Gamma) = \Big(\prod_{e'\in B'_1} \wdim (\Gamma(e'))^{-1}\Big) \Big(\prod_{t'\in B'_2}\dim (\Gamma(t'))\Big) Z(\Gamma). 
$$
\end{lem}
To prove Lemma \ref{Lemma choice of ordered simplices}, it suffices to check that the above quantity is invariant under the change of order and of pointing of 1-, 2-, 3-, and 4-simplices of $K$.  

Suppose that $B_1$ and $B'_1$ are related by the orientation reversal of an edge $\langle ij\rangle \in B_1$. It suffices to check that 
$$
\wdim(\Gamma\langle ij\rangle) = \wdim(\Gamma\langle ji\rangle).
$$
Firstly, by sphericality, $\dim(\Gamma\langle ij\rangle) = \dim(\Gamma\langle ij\rangle^\#) = \dim(\Gamma\langle ji\rangle)$.  Secondly, the fusion category $\mathcal{C}^1_1(\Gamma\langle ij\rangle, \Gamma\langle ij\rangle)$ is equivalent to the category $\mathcal{C}^1_1(\Gamma\langle ij\rangle^\#, \Gamma\langle ij\rangle^\#)^{\mathrm{mop}}$ with opposite monoidal product, therefore 
$$
\dim \mathcal{C}^1_1(\Gamma\langle ij\rangle, \Gamma\langle ij\rangle) = \dim \mathcal{C}^1_1(\Gamma\langle ji\rangle, \Gamma\langle ji\rangle).
$$
Finally, $\mathrm{n}(\Gamma\langle ij\rangle) = \mathrm{n}(\Gamma\langle ji\rangle)$ because simple objects are connected if and only if their duals are connected. 

Suppose that $B_2$ and $B_2'$ are related by the orientation reversal of a triangle $\langle ijk\rangle \in B_2$ with distinguished vertex $i$.  It suffices to check that $\dim(\Gamma\langle ijk \rangle) = \dim(\Gamma\langle ikj\rangle)$.  Recall that 
$$
\Gamma\langle ikj\rangle = (\Gamma\langle ij\rangle \otimes e_{\Gamma\langle jk\rangle})\circ (\Gamma\langle ijk \rangle^*\otimes \Gamma\langle jk\rangle^\#), 
$$
and $\dim(\Gamma\langle ijk\rangle) = \dim(\Gamma\langle ijk \rangle^*)$, therefore $\dim(\Gamma\langle ijk \rangle) = \dim(\Gamma\langle ikj\rangle)$.

Suppose that $B_3$ and $B'_3$ are related by the change of pointing of a 3-simplex $\tau \in B_3$.  For instance, suppose that~$\tau = (\langle ijkl\rangle, \gamma_i, \gamma_j, \gamma_k, \gamma_l)$ and $\tau' = (\langle ijkl\rangle, \gamma_i', \gamma_j, \gamma_k, \gamma_l)$. By induction, it suffices to show the case where
$$
\gamma_i' = \langle jl \rangle \langle lk\rangle \langle kj\rangle \gamma_i.
$$ 
The twist $\theta$ (see \cite[Definition 2.2.4]{douglas_fusion_2018}) induces a 2-isomorphism 
\begin{equation}\label{twist_iso}
\phi\colon \prescript{\Gamma_1(\gamma_0)}{}{\Gamma\langle jkl\rangle} \Rightarrow  \prescript{\Gamma_1(\gamma_0')}{}{\Gamma\langle jkl\rangle}
\end{equation}
that induces isomorphisms between the associator state spaces. The 2-isomorphisms $\phi$ and $\phi^{-1}$ cancel in the expression of $Z(\Gamma)$, which proves the invariance.  The same reasoning shows that $Z(\Gamma)$ is invariant under the change of pointing of a 4-simplex. 

Suppose that $B_3$ and $B'_3$ are related by the change of an order on the vertices of a 3-simplex $\tau\in B_3$.  It suffices to check that $Z(\Gamma)$ is invariant under the transposition of the order of two adjacent vertices of a pointed ordered 3-simplex. An explicit proof of this can be found in \cite[Lemmas 4.2.4 and 4.2.5]{douglas_fusion_2018}. This also proves that $Z(\Gamma)$ is invariant under the change of order on the vertices of a 4-simplex.

\subsection{The state sum for combinatorial $\sigma$-manifolds with boundary}\label{Section_State sum for manifolds with boundary} Following Douglas and Reutter, we extend the state sum to combinatorial $\sigma$-manifolds with boundary.   First, we associate modules to colored closed oriented combinatorial 3-manifolds. Let $T$ be a closed oriented canonical 3-manifold with choice of ordered simplices $B_*$, and let $\phi$ be a $\sigma$-coloring of $T$.  Let~$\varepsilon \colon
 B_3\to \{\pm 1\}$ be such that for any $\tau \in B_3$, $\varepsilon(\tau) = +1$ if and only if the orientation induced by the order on the vertices of $\tau$ coincides with the orientation induced by the orientation of $T$.  To such a colored combinatorial~3-manifold $(T, \phi)$ we assign the following module:
\begin{equation*}
V_\mathcal{C}(T, \phi) = \bigoplus_{\Gamma} \bigotimes_{\tau\in B_3} V^{\varepsilon(\tau)}(\Gamma, \tau),
\end{equation*}
where $\Gamma$ runs over the $\mathcal{C}$-states of $(T,\phi)$ with values in $\mathcal{S}$. 

Let $(-T,\phi)$ be the $\sigma$-colored canonical 3-manifold $(T,\phi)$ with the opposite orientation. We define a nondegenerate pairing $\langle \cdot, \cdot\rangle_T \colon V_\mathcal{C}(-T, \phi) \otimes V_\mathcal{C}(T,\phi) \to \mathbbm{k}$ as follows:
$$
\langle \cdot, \cdot \rangle_{T} = \bigoplus_{\Gamma} \dim(\prescript{}{1}{\mathcal{C}^1_1})^{-\lvert T_0\rvert} \left( \prod_{e\in T^{(1)}}\wdim(\Gamma_1(e))\right)^{-1} \left( \prod_{t\in T^{(2)}}\dim(\Gamma_2(t)) \right)\bigotimes_{\tau\in T^{(3)}}\langle \cdot, \cdot \rangle_{\Gamma, \tau}
$$
where $\Gamma$ runs over the $\mathcal{C}$-states of $(T, \phi)$ with values in $\mathcal{S}$. 

For a combinatorial $\sigma$-manifold with boundary $(K,\phi)$, we define a linear map $$Z_\mathcal{C}(K,\phi) \colon \mathbbm{k}\to V_\mathcal{C}(-\partial K, \phi_{\vert \partial K})$$ as follows
\begin{align*}
Z_\mathcal{C}(K,\phi) & = \bigoplus_{\Sigma} \dim(\prescript{}{1}{\mathcal{C}_1^1})^{-\lvert i(K)_0\rvert}\sum_{\Gamma, \ \Gamma_{\vert \partial K} = \Sigma} \left(\prod_{e\in i(K)^{(1)}}\wdim(\Gamma_1(e)) \right)^{-1}\\
& \quad \quad \left( \prod_{t\in i(K)^{(2)}}\dim(\Gamma_2(t)) \right)\left( \bigotimes_{\kappa \in K^{(4)}}z(\Gamma, \kappa) \right)\circ \left(\bigotimes_{\tau \in K^{(3)}} \cup_{\Gamma, \tau}  \right)
\end{align*}
where $\Sigma$ runs over the $\mathcal{C}$-states of $(\partial K, \phi_{\vert \partial K})$ with values in $\mathcal{S}$,  $\Gamma$ runs over the $\mathcal{C}$-states of $(K,\phi )$ that agree with $\Sigma$ on $\partial K$, and the composition $\circ$ is over all the modules appearing both in the domain of $ \bigotimes_{\kappa \in K^{(4)}}z(\Gamma, \kappa)$ and in the codomain of $\bigotimes_{\tau \in K^{(3)}} \cup_{\Gamma, \tau} $. Therefore, the codomain of $Z_\mathcal{C}(K, \phi)$ agrees with $V_\mathcal{C}(-\partial K, \phi_{\vert \partial K})$

\begin{lem}\label{lemmaboundarysum}
Let $(K,\phi)$ and $(K',\phi')$ be combinatorial $\sigma$-manifolds with boundary and let $f\colon \partial K \to \partial K'$ be an orientation-reversing simplicial isomorphism such that $\phi'\circ f = \phi_{\vert \partial K}$.  Then,
\begin{equation*}
Z_\mathcal{C}(K\cup_f K', \phi\cup_f \phi') = \langle Z_\mathcal{C}(K, \phi) , Z_\mathcal{C} (K', \phi')\rangle_{\partial K}
\end{equation*}
where we used $f$ to identify the modules $V_\mathcal{C}(-\partial K', \phi'_{\vert \partial K'})$ and $V_\mathcal{C}(\partial K, \phi_{\vert \partial K})$. 
\end{lem}

\subsection{The invariance under colored Pachner moves}\label{Invariance triangulation}

Recall that a $\sigma$-colored Pachner move on a combinatorial $\sigma$-manifold $(K,\phi)$ replaces a codimension-0 submanifold of
 $K$ simplicially isomorphic to a subcomplex~$I \subset \partial \Delta^5$ with the complementary subcomplex $J \subset \partial \Delta^5$ and replaces the $\sigma$-coloring $\phi$ with a~$\sigma$-coloring $\psi$ that coincides with $\phi$ outside of $I$. To show invariance of $\mathrm{HDR}_\mathcal{C}$ under such a move, it suffices by Lemma~\ref{lemmaboundarysum} to prove that
\begin{equation*}
Z_\mathcal{C}(I, \phi) = Z_\mathcal{C}(J,\psi),
\end{equation*}
where $\phi$ and $\psi$ are $\sigma$-colorings of $I$ and $J$, respectively, such that $\phi_{\vert \partial I} = \psi_{\vert \partial J}$. Here we suppose that the sub\-complex $I \subset \partial \Delta^5$ carries the orientation induced by the orientation on $\partial \Delta^5$ and that $J$ carries the opposite orientation. 

For any pointed ordered 4-simplex $\kappa$ of a combinatorial $\sigma$-manifold $(K,\phi)$ and for any $\mathcal{C}$-state $\Gamma$ of $(K,\phi)$, recall that the 10j-symbols are defined using the linear maps
$$
z_+(\Gamma, \kappa) \colon \bigotimes_{m = 0}^4 V^{(-1)^m}(\Gamma, \partial_m(\kappa)) \to \mathbbm{k},\quad \quad z_-(\Gamma, \kappa) \colon \bigotimes_{m = 0}^4 V^{(-1)^{m+1}}(\Gamma, \partial_m(\kappa)) \to \mathbbm{k}
$$
depicted in Figure \ref{10j-symbols}. Precomposing with the appropriate maps $\cup \colon \mathbbm{k} \to V^+(\Gamma, \partial_m(\kappa)) \otimes V^-(\Gamma, \partial_m(\kappa))$ (determined by the nondegenerate pairings) leads to linear maps:
\begin{align*}
Z_+(\kappa) &\colon V^+(\Gamma, \partial_4(\kappa)) \otimes V^+(\Gamma, \partial_2(\kappa)) \otimes V^+(\Gamma, \partial_0(\kappa)) \to V^+(\Gamma, \partial_3(\kappa)) \otimes V^+(\Gamma, \partial_1(\kappa)),\\
Z_-(\kappa) &\colon V^+(\Gamma, \partial_3(\kappa)) \otimes V^+(\Gamma, \partial_1(\kappa)) \to V^+(\Gamma, \partial_4(\kappa)) \otimes V^+(\Gamma, \partial_2(\kappa)) \otimes V^+(\Gamma, \partial_0(\kappa)).
\end{align*}
Using these maps, invariance under the $\sigma$-colored Pachner moves can be reexpressed as the following lemmas:

\begin{lem}[Invariance under the colored $(3,3)$-Pachner move]\label{Lemma_Pachner_(3,3)}
Let $I$ and $J$ be the subcomplexes of $\partial \Delta^5$ as defined in Figure \ref{fig:Pachner_(3,3)} and let $\phi$ be a $\sigma$-coloring of $\partial \Delta^5$. Then the following holds for every $\mathcal{C}$-state of $(\partial I,\phi_{\vert \partial I})$:
\begin{equation*}
\sum_{[135]} \dim([135]) Z_+(01235)Z_+(01345)Z_+(12345) = \sum_{[024]} \dim([024])Z_+(02345)Z_+(01245)Z_+(01234),
\end{equation*}
where $[135]$ runs over degree-$\phi(135)$ 1-morphisms in $\mathcal{S}$ from $[13]\otimes [35]$ to $[15]$, and $[024]$ runs over degree-$\phi(024)$ 1-morphisms in $\mathcal{S}$ from $[02]\otimes [24]$ to $[04]$. 
\end{lem}

\begin{lem}[Invariance under the colored $(2,4)$-Pachner move]\label{Lemma_Pachner_(2,4)}
Let $I$ and $J$ be the subcomplexes of $\partial \Delta^5$ as defined in Figure \ref{fig:Pachner_(2,4)} and let $\phi$ be a $\sigma$-coloring of $\partial \Delta^5$. Then the following holds for every $\mathcal{C}$-state of $(\partial I,\phi_{\vert \partial I})$:
\begin{align*}
Z_+&(01235)Z_+(01345) \\
&= \sum_{[24]}\sum_{\substack{[024],  [124], \\ [234], [245]}} \frac{\dim([024])\wdim([124])\dim([234])\dim([245])}{\wdim([24])}Z_+(01234)Z_+(01245)Z_+(02345)Z_+(12345)
\end{align*}
where $[24]$ runs over degree-$\phi(24)$ objects in $\mathcal{S}$,  and $[ijk]$ runs over degree-$\phi(ijk)$ 1-morphisms in $\mathcal{S}$ from $[ij]\otimes [jk]$ to $[ik]$.  
\end{lem}

\begin{lem}[Invariance under the colored $(1,5)$-Pachner move]\label{Lemma_Pachner_(1,5)}
Let $I$ and $J$ be the subcomplexes of $\partial \Delta^5$ as defined in Figure \ref{fig:Pachner_(1,5)} and let $\phi$ be a $\sigma$-coloring of $\partial \Delta^5$. Then the following holds for every $\mathcal{C}$-state of $(\partial I,\phi_{\vert \partial I})$:
\begin{align*}
Z_+(01235) = \dim(\prescript{}{1}{\mathcal{C}^1_1})^{-1}&\sum_{\substack{[04], [14], [24], \\ [34], [45]}} \sum_{\substack{[014], [024], [034], [045],  [124], \\ [134], [145], [234], [245], [345]}} \Big ( \prod\wdim([ij])^{-1}\Big ) \Big ( \prod \dim([ijk])\Big )\\
& \mathrm{Tr}_{V^+(0345)}\big (Z_+(02345)Z_+(01245)Z_+(01234)Z_-(12345)Z_-(01345) \big ),
\end{align*}
where $[ij]$ runs over degree-$\phi(ij)$ objects in $\mathcal{S}$, and $[ijk]$ runs over degree-$\phi(ijk)$ 1-morphisms in $\mathcal{S}$ from $[ij]\otimes [jk]$ to $[ik]$.  The trace $\mathrm{Tr}_{V^+(0345)}$ is over the module $V^+(0345)$. 
\end{lem}

\begin{figure}[h]
\begin{center}
\begin{tabular}{c | c | c}
Pachner $(3,3)$-move& $I$ & $J$ \\
\hline
4-simplices & $(01234)$, $(01245)$, $(02345)$ & $(01235)$, $(01345)$, $(12345)$\\ \hline 
Tetrahedra & $(0124)$, $(0234)$, $(0245)$& $(0135)$, $(1235)$, $(1345)$ \\ \hline
Triangles & $(024)$ &$(135)$ \\
\end{tabular}
\end{center}
\caption{The simplices in the interior of $I$ and $J$ for the $\sigma$-colored Pachner $(3,3)$-move}
\label{fig:Pachner_(3,3)}
\end{figure}

\begin{figure}[h]
\begin{center}
\begin{tabular}{c | c | c}
Pachner $(2,4)$-move& $I$ & $J$ \\
\hline
4-simplices & $(01235)$, $(01345)$ & $(01234)$, $(01245)$, $(02345)$, $(12345)$ \\ \hline 
Tetrahedra & $(0135)$  & $(0124)$, $(0234)$, $(0245)$, $(1234)$, $(1245)$, $(2345)$  \\ \hline
Triangles &   & $(024)$, $(124)$, $(234)$, $(245)$  \\ \hline
Edges & & $(24)$
\end{tabular}
\end{center}
\caption{The simplices in the interior of $I$ and $J$ for the $\sigma$-colored Pachner $(2,4)$-move}
\label{fig:Pachner_(2,4)}
\end{figure}

\begin{figure}[h]
\begin{center}
\begin{tabular}{c | c | c}
Pachner $(1,5)$-move& $I$ & $J$ \\
\hline
4-simplices & $(01235)$ & $(01234)$, $(01245)$, $(01345)$, $(02345)$, $(12345)$  \\ \hline 
Tetrahedra & & $(0124)$, $(0134)$, $(0145)$, $(0234)$, $(0245)$, $(0345)$, $(1234)$, $(1245)$, $(1345)$, $(2345)$  \\ \hline
Triangles &   & $(014)$, $(024)$, $(034)$, $(045)$, $(124)$, $(134)$, $(145)$, $(234)$, $(245)$, $(345)$  \\ \hline
Edges & & $(04)$, $(14)$, $(24)$, $(34)$, $(45)$ \\ \hline
Vertices &  & $(4)$
\end{tabular}
\end{center}
\caption{The simplices in the interior of $I$ and $J$ for the $\sigma$-colored Pachner $(1,5)$-move}
\label{fig:Pachner_(1,5)}
\end{figure}

The proofs of Lemmas \ref{Lemma_Pachner_(3,3)},  \ref{Lemma_Pachner_(2,4)}, and \ref{Lemma_Pachner_(1,5)} follow directly from the proofs of \cite[Lemmas 4.3.4, 4.3.5, and 4.3.6]{douglas_fusion_2018} by using the formulas shown in Section \ref{Section_Dimension_formulas}.  The details are left as an exercise for the reader.

\section{Proof of Lemma \ref{Pachner_(1,5)}}\label{Section_Proof of Lemma 8.3.1}

\noindent In this section we prove Lemma \ref{Pachner_(1,5)}, which is a key step in the proof of the Colored Pachner Theorem. Let~$L$ be the canonical 4-manifold obtained by applying a Pachner $(1,5)$-move to a canonical 4-manifold $K$ with $\sigma$-coloring $\phi$.  Let us show that the map $\psi$ specified above defines a unique $\sigma$-coloring of $L$. 

First we show that these choices determine $\psi$. Recall that $\psi_{\vert K} = \phi$, therefore we only need to check that $\psi$ is uniquely defined on simplices containing the vertex $\langle 5\rangle$. The triangle conditions applied to $\langle 015\rangle$, $\langle 025\rangle$, $\langle 035\rangle$, and $\langle 045\rangle$ determine the values of $\psi\langle 15\rangle$, $\psi\langle 25\rangle$, $\psi\langle 35\rangle$, and $\psi\langle 45\rangle$, respectively. Therefore $\psi$ is uniquely defined on all edges of $L$.  Similarly, the tetrahedron conditions applied to $\langle 0125\rangle$, $\langle 0135\rangle$, $\langle 0145\rangle$, $\langle 0235\rangle $,  $\langle 0245\rangle$,  and~$\langle 0345\rangle $ determine the values of $\psi\langle 125\rangle$,  $ \psi\langle 135\rangle$,  $ \psi\langle 145\rangle$, $\psi\langle 235\rangle $,  $\psi\langle 245\rangle$, and $\psi\langle 345\rangle$, respectively.  Therefore,~$\psi$ is uniquely defined on all triangles of $L$.  The 4-simplex conditions applied to $\langle 01235\rangle$, $\langle 01245\rangle$, $\langle 01345\rangle$, and $\langle 02345\rangle$ determine the values of $\psi\langle 1235\rangle$, $\psi\langle 1245\rangle$, $\psi\langle 1345\rangle$, and $\psi\langle 2345\rangle$, respectively. Therefore $\psi$ is uniquely defined on all $1$-, $2$- and $3$-simplices of $L$. 

Now we check that $\psi$ is a $\sigma$-coloring of $L$.  Since $\psi_{\vert K} = \phi$ and $\phi$ is a $\sigma$-coloring of $K$, $\psi$ verifies the simplex conditions for all $2$-, $3$- and $4$-simplices of $K$.  By construction, $\psi$ already verifies the simplex conditions for $\langle 015\rangle$, $\langle 025\rangle$, $\langle 035\rangle$, and $\langle 045\rangle$.  For any choice of pointing:
\begin{itemize}
\item $\gamma_0$ in $\langle 125\rangle$ from $\langle 1\rangle$ to the distinguished vertex of $\langle 125\rangle$,
\item $\gamma_1$ in $\langle 015\rangle $ from $\langle 0\rangle$ to the distinguished vertex of $\langle 015\rangle$,
\item $\gamma_2$ in $\langle 025\rangle$ from $\langle 0\rangle$ to the distinguished vertex of $\langle 025\rangle$, 
\item $\gamma_3$ in $\langle 012\rangle$ from $\langle 0\rangle$ to the distinguished vertex of $\langle 012\rangle$,
\end{itemize}
we get
\begin{align*}
\partial (\psi\prescript{\gamma_0}{}{\langle 125 \rangle}) &= \phi\langle 01\rangle^{-1} \partial(\prescript{\gamma_1}{}{e_1})^{-1}  \partial (\prescript{\gamma_2}{}{e_2})^{-1} \partial (\phi( \prescript{\gamma_3}{}{\langle 012\rangle}) \phi\langle 01\rangle\\
&= \psi\langle 15\rangle \psi\langle 25\rangle^{-1} \psi\langle 12\rangle^{-1}.
\end{align*}
Therefore $\psi$  verifies the triangle condition for $\langle 125\rangle$. Similar computations show that $\psi$ verifies the triangle conditions for $\langle 135\rangle$, $\langle 145\rangle$, $\langle 235\rangle$, $\langle 245\rangle$,  and $\langle 345\rangle$.  This proves that $\psi$ verifies the triangle conditions for all triangles of $L$. 

By construction, $\psi$ already verifies the tetrahedron conditions for $\langle 0125\rangle$, $\langle 0135\rangle$, $\langle 0145\rangle$, $\langle 0235\rangle$, $\langle 0245\rangle$  and $\langle 0345\rangle$.  Let $(\langle 01235\rangle, \bar{\gamma})$ be a pointed ordered 4-simplex, with 
$$
\bar{\gamma} = (\gamma_{01}, \gamma_{02}, \gamma_{03}, \gamma_{04}, \gamma_{12}, \gamma_{13}, \gamma_{14}, \gamma_{23}, \gamma_{24}, \gamma_{34}).
$$
For vertices $i,j,k,l\in \{0,1,2,3,5\}$, we denote by $(\langle ijkl\rangle, \bar{\gamma})$ the pointed ordered tetrahedron $\langle ijkl\rangle$ induced by those of $(\langle 01235\rangle, \bar{\gamma})$.  By definition,
\begin{multline*}
\psi\prescript{\langle 01\rangle}{}{(\langle 1235\rangle}, \bar{\gamma}) = \psi(\prescript{\gamma_{23}}{}{\langle 015\rangle})^{-1} \triangleright \left[ \psi(\langle 0125\rangle, \bar{\gamma})d(\prescript{\gamma_{13}}{}{\langle 025\rangle})\triangleright \omega(\psi(\prescript{\gamma_{35}}{}{\langle 012\rangle} ) , \psi(\prescript{\langle 01\rangle \langle 12\rangle \gamma_{01}}{}{\langle 235\rangle} ))\psi(\langle 0235\rangle, \bar{\gamma})\right. \\
\left.  \psi(\prescript{\gamma_{12}}{}{\langle 035\rangle }) 
\triangleright \psi( \langle 0123\rangle, \bar{\gamma})^{-1} \psi(\langle 0135\rangle, \bar{\gamma})^{-1} \right].
\end{multline*}
Therefore, 
\begin{equation*}
\begin{split}
\delta (\psi\prescript{\langle 01 \rangle }{}{(\langle 1235\rangle, \bar{\gamma}}) & = \psi(\prescript{\gamma_{23}}{}{\langle 015\rangle})^{-1}\delta \psi(\langle 0125\rangle, \bar{\gamma}) \psi(\prescript{\gamma_{13}}{}{\langle 025\rangle})\delta\omega(\psi(\prescript{\gamma_{35}}{}{\langle 012\rangle} ) , \psi(\prescript{\langle 01\rangle \langle 12\rangle \gamma_{01}}{}{\langle 235\rangle} ))\psi(\prescript{\gamma_{13}}{}{\langle 025\rangle})^{-1}\\
& \quad\quad  \delta \psi(\langle 0235\rangle, \bar{\gamma})\psi(\prescript{\gamma_{12}}{}{\langle 035\rangle})\delta \psi(\langle 0123\rangle, \bar{\gamma})^{-1}\psi(\prescript{\gamma_{12}}{}{\langle 035\rangle })^{-1}\delta \psi(\langle 0135 \rangle, \bar{\gamma})^{-1}\psi(\prescript{\gamma_{23}}{}{\langle 015\rangle}) \\
&= \psi(\prescript{\langle 01\rangle \gamma_{03}}{}{\langle 125 \rangle})\psi(\prescript{\langle 01 \rangle\langle 12\rangle\gamma_{01}}{}{\langle 235\rangle})\psi(\prescript{\langle 01 \rangle \gamma_{05}}{}{\langle 123\rangle })^{-1}\psi(\prescript{\langle 01\rangle \gamma_{02}}{}{\langle 135\rangle})^{-1}.
\end{split}
\end{equation*}
Similarly, $\psi$ verifies the tetrahedron conditions for $\langle 1245\rangle $, $\langle 1345\rangle$,  and $\langle 2345\rangle$. 

By construction, $\psi$ verifies the 4-simplex conditions for $\langle 01235\rangle$, $\langle 01245\rangle$, $\langle 01345\rangle$, and $\langle 02345\rangle$.  It only remains to check that $\psi$ verifies the 4-simplex condition for $\langle 12345\rangle$.  Let 
$$
\bar{\gamma} = (\gamma_{01}, \gamma_{02}, \gamma_{03}, \gamma_{04}, \gamma_{12}, \gamma_{13}, \gamma_{14}, \gamma_{23}, \gamma_{24}, \gamma_{34})
$$
denote a choice of pointing such that $(\langle 12345\rangle, \bar{\gamma})$ is a pointed ordered 4-simplex.  For ease of reading, we introduce the following notations:
\begin{itemize}
\item for any vertices $i,j \in \{0, 1, 2, 3, 4, 5\}$ such that $i< j$, we write $[ij] = \psi\langle ij\rangle$,
\item for any vertices $i,j,k, l, m\in \{0,1,2,3,4,5\}$ such that $i<j<k$ and $l<m$, we write $$[ijk] = \prescript{\psi(\gamma_{lm})}{}{\psi\langle ijk\rangle},$$
\item for any vertices $i,j,k,l\in \{0,1,2,3,4,5\}$ such that $i<j<k<l$, we write $$[ijkl] = \psi(\langle ijkl\rangle, \bar{\gamma}),$$ where $(\langle ijkl\rangle, \bar{\gamma})$ denotes the pointed ordered tetrahedron $\langle ijkl\rangle$ induced by $\bar{\gamma}$. 
\end{itemize}
The 4-simplex condition amounts to $E = 1$, with 
\begin{equation*}
E  = [1235] [135]\triangleright \omega([123], \prescript{[12][23]}{}{[345]})[1345][145]\triangleright [1234]^{-1}[1245]^{-1} [125]\triangleright \prescript{[12]}{}{[2345]}^{-1}.
\end{equation*}
By construction,
\begin{align*}
\prescript{[01]}{}{[1235]} &= [015]^{-1}\triangleright \left([0125][025]\triangleright \omega([012], \prescript{[01][12]}{}{[235]} )[0235][035]\triangleright [0123]^{-1}[0135]^{-1} \right),\\
\prescript{[01]}{}{[1245]} &= [015]^{-1}\triangleright \left([0125][025]\triangleright \omega([012], \prescript{[01][12]}{}{[245]} )[0245][045]\triangleright [0124]^{-1}[0145]^{-1}  \right),\\
\prescript{[01]}{}{[1345]} &= [015]^{-1}\triangleright \left([0135][035]\triangleright \omega([013], \prescript{[01][13]}{}{[345]} )[0345][045]\triangleright [0134]^{-1}[0145]^{-1}  \right),\\
\prescript{[01]}{}{[2345]} &= [025]^{-1}\triangleright \left([0235][035]\triangleright \omega([023], \prescript{[02][23]}{}{[345]} )[0345][045]\triangleright [0234]^{-1}[0245]^{-1}  \right).
\end{align*}
Therefore, 
\begin{multline*}
E = \prescript{[01]^{-1}}{}{[015]^{-1}}\triangleright \prescript{[01]^{-1}}{}{\Big (} [0125][025]\triangleright \omega([012], \prescript{[01][12]}{}{[235]} )[0235][035]\triangleright [0123]^{-1}[0135]^{-1} \\
([015]\prescript{[01]}{}{[135]}) \triangleright \prescript{[01]}{}{\omega\left([123], \prescript{[12][23]}{}{[345]}\right)}[0135][035]\triangleright \omega \left( [013], \prescript{[01][13]}{}{[345]}\right) [0345][045]\triangleright [0134]^{-1}\\
[0145]^{-1}([015]\prescript{[01]}{}{[145]})\triangleright \prescript{[01]}{}{[1234]^{-1}}[0145][045]\triangleright [0124][0245]^{-1}[025]\triangleright \omega\left( [012], \prescript{[01][12]}{}{[245]}\right)^{-1}\\
[0125]^{-1} ([015]\prescript{[01]}{}{[125]}\prescript{[01][12][02]^{-1}}{}{[025]^{-1}})\triangleright \prescript{[01][12][02]^{-1}}{}{\Big (} [0245][045]\triangleright [0234][0345]^{-1}\\
[035] \triangleright \omega\left( [023], \prescript{[02][23]}{}{[345]}\right)^{-1}[0235]^{-1} \Big ) \Big ) .
\end{multline*}
Recall that $\delta [0135] = [015]\prescript{[01]}{}{[135]}[013]^{-1}[035]^{-1}$, therefore 
\begin{equation*}
[015]\prescript{[01]}{}{[135]} = \delta [0135] [035][013].
\end{equation*}
Similarly,
\begin{align*}
[015]\prescript{[01]}{}{[145]} & = \delta [0145] [045] [014],\\
[015]\prescript{[01]}{}{[125]} & = \delta [0125] [025] [012]. 
\end{align*}
Using the fact that for all $l,m \in L$, $(\delta l) \triangleright m = lml^{-1}$, we get
\begin{multline*}
E = \prescript{[01]^{-1}}{}{[015]}^{-1}\triangleright \prescript{[01]^{-1}}{}{\Big (} [0125] [025]\triangleright \omega\left([012], \prescript{[01][12]}{}{[235]} \right) [0235] [035]\triangleright [0123]^{-1} \\([035][013]) \triangleright \prescript{[01]}{}{\omega} \left([123], \prescript{[12][23]}{}{[345]} \right)
[035]\triangleright \omega\left([013], \prescript{[01][13]}{}{[345]} \right)[0345][045]\triangleright [0134]^{-1}\\
([045][014])\triangleright \prescript{[01]}{}{[1234]^{-1}} [045]\triangleright [0124][0245]^{-1}[025]\triangleright \omega\left([012], \prescript{[01][12]}{}{[245]} \right)^{-1}\\
([025][012]\prescript{\partial [012]^{-1}}{}{[025]^{-1}}) \triangleright \prescript{\partial [012]^{-1}}{}{\Big (}[0245] [045]\triangleright [0234][0345]^{-1}[035] \triangleright \omega\left([023], \prescript{[02][23]}{}{[345]} \right)^{-1}\\ 
 [0235]^{-1}\Big ) [0125]^{-1}\Big ).
\end{multline*}
The equation $E = 1$ is equivalent to $E_1 = 1$ with 
\begin{multline*}
E_1 = [025]\triangleright \omega\left([012], \prescript{[01][12]}{}{[235]} \right) [0235] [035]\triangleright [0123]^{-1} ([035][013]) \triangleright \prescript{[01]}{}{\omega} \left([123], \prescript{[12][23]}{}{[345]} \right)\\
[035]\triangleright \omega\left([013], \prescript{[01][13]}{}{[345]} \right)[0345][045]\triangleright [0134]^{-1} ([045][014])\triangleright \prescript{[01]}{}{[1234]^{-1}} [045]\triangleright [0124]\\
 [0245]^{-1}[025]\triangleright \omega\left([012], \prescript{[01][12]}{}{[245]} \right)^{-1} ([025][012]\prescript{\partial [012]^{-1}}{}{[025]^{-1}}) \triangleright \prescript{\partial [012]^{-1}}{}{\Big (}[0245] [045]\triangleright [0234]\\
 [0345]^{-1}[035] \triangleright \omega\left([023], \prescript{[02][23]}{}{[345]} \right)^{-1} [0235]^{-1}\Big )
\end{multline*}
Recall that for all $e,f \in E$, we have $\delta \omega(e,f) = efe^{-1}\prescript{\partial e}{}{f^{-1}}$.  Therefore, 
\begin{equation*}
[025][012]\prescript{\partial [012]^{-1}}{}{[025]^{-1}} = [012] \delta \omega\left( [012]^{-1}, [025]\right).
\end{equation*}
Moreover, recall from \cite[Lemma 2.2]{martins_fundamental_2011} that for all $e\in E$ and $l\in L$, we have
\begin{equation*}
\prescript{\partial e}{}{l} = \left(e\triangleright l\right) \omega(e, \delta( l^{-1})). 
\end{equation*}
Therefore,
\begin{align*}
([025]&[012]\prescript{\partial [012]^{-1}}{}{[025]^{-1}}) \triangleright \prescript{\partial [012]^{-1}}{}{\Big (}[0245] [045]\triangleright [0234]
 [0345]^{-1}[035] \triangleright \omega\left([023], \prescript{[02][23]}{}{[345]} \right)^{-1} [0235]^{-1}\Big )\\ 
&= [012]\triangleright \Big ( \omega ([012]^{-1}, [025]) \prescript{\partial [012]^{-1}}{}{\Big (}0245] [045]\triangleright [0234]
 [0345]^{-1}[035] \triangleright \omega\left([023], \prescript{[02][23]}{}{[345]} \right)^{-1} [0235]^{-1}\Big )\\
 & \quad \quad \omega ([012]^{-1}, [025])^{-1}\Big )\\
 &= [012] \triangleright \omega ([012]^{-1}, [025]) [0245][045]\triangleright [0234][0345]^{-1}[035] \triangleright \omega\left([023], \prescript{[02][23]}{}{[345]} \right)^{-1} [0235]^{-1}\\
 & \quad \quad [012]\triangleright \omega \left([012]^{-1}, [025]\prescript{[02]}{}{[235]}\prescript{[02][23]}{}{[345]}\prescript{[02]}{}{[234]^{-1}}\prescript{[02]}{}{[245]^{-1}} [025]^{-1} \right) [012]\triangleright \omega ([012]^{-1}, [025])^{-1}
\end{align*}
We substitute in $E_1$: 
\begin{multline*}
E_1 = [025]\triangleright \omega\left([012], \prescript{[01][12]}{}{[235]} \right) [0235] [035]\triangleright [0123]^{-1} ([035][013]) \triangleright \prescript{[01]}{}{\omega} \left([123], \prescript{[12][23]}{}{[345]} \right)\\
[035]\triangleright \omega\left([013], \prescript{[01][13]}{}{[345]} \right)[0345][045]\triangleright [0134]^{-1} ([045][014])\triangleright \prescript{[01]}{}{[1234]^{-1}} [045]\triangleright [0124]\\
 [0245]^{-1}[025]\triangleright \omega\left([012], \prescript{[01][12]}{}{[245]} \right)^{-1} [012] \triangleright \omega ([012]^{-1}, [025]) [0245][045]\triangleright [0234][0345]^{-1}\\
 [035] \triangleright \omega\left([023], \prescript{[02][23]}{}{[345]} \right)^{-1} [0235]^{-1} [012]\triangleright \omega \Big ([012]^{-1}, [025]\prescript{[02]}{}{[235]}\prescript{[02][23]}{}{[345]}\prescript{[02]}{}{[234]^{-1}} \\
\prescript{[02]}{}{[245]^{-1}} [025]^{-1} \Big ) [012]\triangleright \omega ([012]^{-1}, [025])^{-1}
\end{multline*}
From the fact that $\phi$ is a $\sigma$-coloring of $K$, we deduce that 
\begin{equation*}
1 = [0124][024]\triangleright \omega\left([012], \prescript{[01][12]}{}{[234]} \right)[0234][034]\triangleright [0123]^{-1}[0134]^{_1}[014]\triangleright \prescript{[01]}{}{[1234]^{-1}}.
\end{equation*}
Therefore, 
\begin{align*}
[045]\triangleright [0134]^{-1}([045][014]) \triangleright \prescript{[01]}{}{[1234]^{-1}}[045]\triangleright [0124] &= ([045][034]) \triangleright [0123][045]\triangleright [0234]^{-1}\\
& \quad \quad ([045][024]) \triangleright \omega \left( [012], \prescript{[01][12]}{}{[234]} \right)^{-1}.
\end{align*}
We substitute in $E_1$:
\begin{multline*}
E_1 = [025]\triangleright \omega\left([012], \prescript{[01][12]}{}{[235]} \right) [0235] [035]\triangleright [0123]^{-1} ([035][013]) \triangleright \prescript{[01]}{}{\omega} \left([123], \prescript{[12][23]}{}{[345]} \right)\\
[035]\triangleright \omega\left([013], \prescript{[01][13]}{}{[345]} \right)[0345]([045][034]) \triangleright [0123][045]\triangleright [0234]^{-1} \\
([045][024]) \triangleright \omega \left( [012], \prescript{[01][12]}{}{[234]} \right)^{-1} [0245]^{-1}[025]\triangleright \omega\left([012], \prescript{[01][12]}{}{[245]} \right)^{-1} \\
 [012] \triangleright \omega ([012]^{-1}, [025]) [0245][045]\triangleright [0234][0345]^{-1}  [035] \triangleright \omega\left([023], \prescript{[02][23]}{}{[345]} \right)^{-1} [0235]^{-1} \\
 [012]\triangleright \omega \Big ([012]^{-1}, [025]\prescript{[02]}{}{[235]}\prescript{[02][23]}{}{[345]}\prescript{[02]}{}{[234]^{-1}} \prescript{[02]}{}{[245]^{-1}} [025]^{-1} \Big ) [012]\triangleright \omega ([012]^{-1}, [025])^{-1}.
\end{multline*}
Recall that for all $l,m\in L$, we have $lm = (\delta l\triangleright m) l$, therefore
\begin{equation*}
[0345]([045][034])\triangleright [0123] = ([035]\prescript{[03]}{}{[345]}) \triangleright [0123] [0345].
\end{equation*}
This yields
\begin{multline*}
E_1 = [025]\triangleright \omega\left([012], \prescript{[01][12]}{}{[235]} \right) [0235] [035]\triangleright [0123]^{-1} ([035][013]) \triangleright \prescript{[01]}{}{\omega} \left([123], \prescript{[12][23]}{}{[345]} \right)\\
[035]\triangleright \omega\left([013], \prescript{[01][13]}{}{[345]} \right)([035]\prescript{[03]}{}{[345]}) \triangleright [0123] [0345][045]\triangleright [0234]^{-1} \\
([045][024]) \triangleright \omega \left( [012], \prescript{[01][12]}{}{[234]} \right)^{-1} [0245]^{-1}[025]\triangleright \omega\left([012], \prescript{[01][12]}{}{[245]} \right)^{-1} \\
 [012] \triangleright \omega ([012]^{-1}, [025]) [0245][045]\triangleright [0234][0345]^{-1}  [035] \triangleright \omega\left([023], \prescript{[02][23]}{}{[345]} \right)^{-1} [0235]^{-1} \\
 [012]\triangleright \omega \Big ([012]^{-1}, [025]\prescript{[02]}{}{[235]}\prescript{[02][23]}{}{[345]}\prescript{[02]}{}{[234]^{-1}} \prescript{[02]}{}{[245]^{-1}} [025]^{-1} \Big ) [012]\triangleright \omega ([012]^{-1}, [025])^{-1}.
\end{multline*}
Recall from \cite[Lemma 2.2]{martins_fundamental_2011} that for all $e,f,g\in E$ we have 
$\omega(e,fg) =\left( (efe^{-1})\triangleright \omega(e,g)\right)\omega(e,f)$, therefore 
\begin{multline*}
\omega \Big ([012]^{-1}, [025]\prescript{[02]}{}{[235]}\prescript{[02][23]}{}{[345]}\prescript{[02]}{}{[234]^{-1}} \prescript{[02]}{}{[245]^{-1}} [025]^{-1} \Big ) = [012]^{-1}[025][012])\triangleright \omega \Big ([012]^{-1}, \prescript{[02]}{}{[235]}\\
\prescript{[02][23]}{}{[345]}\prescript{[02]}{}{[234]^{-1}} \prescript{[02]}{}{[245]^{-1}} [025]^{-1} \Big )\omega([012]^{-1}, [025]).
\end{multline*}
This yields
\begin{multline*}
E_1 = [025]\triangleright \omega\left([012], \prescript{[01][12]}{}{[235]} \right) [0235] [035]\triangleright [0123]^{-1} ([035][013]) \triangleright \prescript{[01]}{}{\omega} \left([123], \prescript{[12][23]}{}{[345]} \right)\\
[035]\triangleright \omega\left([013], \prescript{[01][13]}{}{[345]} \right)([035]\prescript{[03]}{}{[345]}) \triangleright [0123] [0345][045]\triangleright [0234]^{-1} \\
([045][024]) \triangleright \omega \left( [012], \prescript{[01][12]}{}{[234]} \right)^{-1} [0245]^{-1}[025]\triangleright \omega\left([012], \prescript{[01][12]}{}{[245]} \right)^{-1} \\
 [012] \triangleright \omega ([012]^{-1}, [025]) [0245][045]\triangleright [0234][0345]^{-1}  [035] \triangleright \omega\left([023], \prescript{[02][23]}{}{[345]} \right)^{-1} [0235]^{-1} \\
([025][012])\triangleright \omega \Big ([012]^{-1}, \prescript{[02]}{}{[235]} \prescript{[02][23]}{}{[345]}\prescript{[02]}{}{[234]^{-1}} \prescript{[02]}{}{[245]^{-1}} [025]^{-1} \Big )
\end{multline*}
Recall that by definition,
\begin{equation*}
\prescript{[03]}{}{[345]}\triangleright [0123] = [0123] \omega\Big ( [023][012]\prescript{[01]}{}{[123]^{-1}}[013]^{-1}, \prescript{[03]}{}{[345]}  \Big ).
\end{equation*}
Therefore,
\begin{multline*}
E_1 = [025]\triangleright \omega\left([012], \prescript{[01][12]}{}{[235]} \right) [0235] [035]\triangleright [0123]^{-1} ([035][013]) \triangleright \prescript{[01]}{}{\omega} \left([123], \prescript{[12][23]}{}{[345]} \right)\\
[035]\triangleright \omega\left([013], \prescript{[01][13]}{}{[345]} \right)[035]\triangleright [0123] [035]\triangleright \omega\Big ( [023][012]\prescript{[01]}{}{[123]^{-1}}[013]^{-1}, \prescript{[03]}{}{[345]}  \Big )\\
  [0345][045]\triangleright [0234]^{-1} ([045][024]) \triangleright \omega \left( [012], \prescript{[01][12]}{}{[234]} \right)^{-1} [0245]^{-1}[025]\triangleright \omega\left([012], \prescript{[01][12]}{}{[245]} \right)^{-1} \\
 [012] \triangleright \omega ([012]^{-1}, [025]) [0245][045]\triangleright [0234][0345]^{-1}  [035] \triangleright \omega\left([023], \prescript{[02][23]}{}{[345]} \right)^{-1} [0235]^{-1} \\
([025][012])\triangleright \omega \Big ([012]^{-1}, \prescript{[02]}{}{[235]} \prescript{[02][23]}{}{[345]}\prescript{[02]}{}{[234]^{-1}} \prescript{[02]}{}{[245]^{-1}} [025]^{-1} \Big )
\end{multline*}
Recall that for all $e,f,g\in E$, we have $\omega(e,fg) =\left( (efe^{-1})\triangleright \omega(e,g)\right) \omega(e,f)$, therefore
\begin{multline*}
\omega \Big ([012]^{-1}, \prescript{[02]}{}{[235]}\prescript{[02][23]}{}{[345]}\prescript{[02]}{}{[234]^{-1}} \prescript{[02]}{}{[245]^{-1}} [025]^{-1} \Big ) = ([012]^{-1}\prescript{[02]}{}{[235]}[012]) \triangleright \omega\Big ( [012]^{-1}, \prescript{[02][23]}{}{[345]}\\
\prescript{[02]}{}{[234]^{-1}} \prescript{[02]}{}{[245]^{-1}} [025]^{-1} \Big )\omega\Big ( [012]^{-1}, \prescript{[02]}{}{[235]}\Big ).
\end{multline*}
Moreover, recall from \cite[Lemma 2.2]{martins_fundamental_2011} that for all $e,f\in E$, we have $\omega(e,f) = e\triangleright \omega(e^{-1}, \prescript{\partial e}{}{f})^{-1}$, therefore,
\begin{equation*}
\omega\Big ( [012]^{-1}, \prescript{[02]}{}{[235]}\Big ) = [012]^{-1}\triangleright \omega\Big ( [012], \prescript{[01][12]}{}{[235]}\Big )^{-1}.
\end{equation*}
By substituting, we get that $E_1 = 1$ is equivalent to $E_2 = 1$ with 
\begin{multline*}
E_2 = [0235] [035]\triangleright [0123]^{-1} ([035][013]) \triangleright \prescript{[01]}{}{\omega} \left([123], \prescript{[12][23]}{}{[345]} \right) [035]\triangleright \omega\left([013], \prescript{[01][13]}{}{[345]} \right)\\
[035]\triangleright [0123] [035]\triangleright \omega\Big ( [023][012]\prescript{[01]}{}{[123]^{-1}}[013]^{-1}, \prescript{[03]}{}{[345]}  \Big )[0345][045]\triangleright [0234]^{-1} \\
  ([045][024]) \triangleright \omega \left( [012], \prescript{[01][12]}{}{[234]} \right)^{-1} [0245]^{-1}[025]\triangleright \omega\left([012], \prescript{[01][12]}{}{[245]} \right)^{-1} [012] \triangleright \omega ([012]^{-1}, [025]) \\
 [0245][045]\triangleright [0234][0345]^{-1}  [035] \triangleright \omega\left([023], \prescript{[02][23]}{}{[345]} \right)^{-1} [0235]^{-1} \\
([025]\prescript{[02]}{}{[235]}[012])\triangleright \omega \Big ([012]^{-1},  \prescript{[02][23]}{}{[345]}\prescript{[02]}{}{[234]^{-1}} \prescript{[02]}{}{[245]^{-1}} [025]^{-1} \Big ).
\end{multline*}
Recall that for all $l,m \in L$, we have $lm = (\delta l \triangleright m)l$, therefore
\begin{align*}
[0235]^{-1} ([025]\prescript{[02]}{}{[235]}&[012])\triangleright  \omega \Big ([012]^{-1},  \prescript{[02][23]}{}{[345]}\prescript{[02]}{}{[234]^{-1}} \prescript{[02]}{}{[245]^{-1}} [025]^{-1} \Big ) \\
&= ([035][023][012])\triangleright \omega \Big ([012]^{-1},  \prescript{[02][23]}{}{[345]}\prescript{[02]}{}{[234]^{-1}} \prescript{[02]}{}{[245]^{-1}} [025]^{-1} \Big )[0234]^{-1}.
\end{align*}
By substituting, we get that $E_2 = 1$ is equivalent to $E_3 = 1$ with 
\begin{multline*}
E_3 = [035]\triangleright [0123]^{-1} ([035][013]) \triangleright \prescript{[01]}{}{\omega} \left([123], \prescript{[12][23]}{}{[345]} \right) [035]\triangleright \omega\left([013], \prescript{[01][13]}{}{[345]} \right)\\
[035]\triangleright [0123] [035]\triangleright \omega\Big ( [023][012]\prescript{[01]}{}{[123]^{-1}}[013]^{-1}, \prescript{[03]}{}{[345]}  \Big )[0345][045]\triangleright [0234]^{-1} \\
  ([045][024]) \triangleright \omega \left( [012], \prescript{[01][12]}{}{[234]} \right)^{-1} [0245]^{-1}[025]\triangleright \omega\left([012], \prescript{[01][12]}{}{[245]} \right)^{-1} [012] \triangleright \omega ([012]^{-1}, [025]) \\
 [0245][045]\triangleright [0234][0345]^{-1}  [035] \triangleright \omega\left([023], \prescript{[02][23]}{}{[345]} \right)^{-1} \\
 ([035][023][012])\triangleright \omega \Big ([012]^{-1},  \prescript{[02][23]}{}{[345]}\prescript{[02]}{}{[234]^{-1}} \prescript{[02]}{}{[245]^{-1}} [025]^{-1} \Big ).
\end{multline*}
After applying the formula $lm = (\delta l \triangleright m)l$ with $l = [035]\triangleright [0123]^{-1}$ and 
\begin{equation*}
m = ([035][013]) \triangleright \prescript{[01]}{}{\omega} \left([123], \prescript{[12][23]}{}{[345]} \right) [035]\triangleright \omega\left([013], \prescript{[01][13]}{}{[345]} \right),
\end{equation*}
we get 
\begin{multline*}
E_3 = ([035][023][012]\prescript{[01]}{}{[123]^{-1}}) \triangleright \prescript{[01]}{}{\omega} \left([123], \prescript{[12][23]}{}{[345]} \right) ([035][023][012]\prescript{[01]}{}{[123]^{-1}}\\
[013]^{-1})\triangleright \omega\left([013], \prescript{[01][13]}{}{[345]} \right) [035]\triangleright \omega\Big ( [023][012]\prescript{[01]}{}{[123]^{-1}}[013]^{-1}, \prescript{[03]}{}{[345]}  \Big )[0345][045]\triangleright [0234]^{-1} \\
([045][024]) \triangleright \omega \left( [012], \prescript{[01][12]}{}{[234]} \right)^{-1}
  [0245]^{-1}[025]\triangleright \omega\left([012], \prescript{[01][12]}{}{[245]} \right)^{-1} [012] \triangleright \omega ([012]^{-1}, [025]) [0245]\\
  [045]\triangleright [0234][0345]^{-1}
  [035] \triangleright \omega\left([023], \prescript{[02][23]}{}{[345]} \right)^{-1} ([035][023][012])\triangleright \omega \Big ([012]^{-1},  \prescript{[02][23]}{}{[345]}\\
  \prescript{[02]}{}{[234]^{-1}} \prescript{[02]}{}{[245]^{-1}} [025]^{-1} \Big ).
\end{multline*}
Recall from \cite[Lemma 2.2]{martins_fundamental_2011} that for all $e,f,g \in E$, we have $\omega(ef,g) = \left( e\triangleright \omega(f,g)\right) \omega(e,\prescript{\partial f}{}{g})$ and that for all $e,f \in E$, we have $\omega (e,f) = e\triangleright \omega(e^{-1}, \prescript{\partial e}{}{f})^{-1}$, therefore 
\begin{multline*}
\omega\Big ( [023][012]\prescript{[01]}{}{[123]^{-1}}[013]^{-1}, \prescript{[03]}{}{[345]}  \Big ) = ([023][012]\prescript{[01]}{}{[123]^{-1}}[013]^{-1}) \triangleright \omega\Big ( [013], \prescript{[01][13]}{}{[345]}\Big )^{-1}\\
\omega\Big ( [023][012]\prescript{[01]}{}{[123]^{-1}}, \prescript{[01][13]}{}{[345]} \Big ).
\end{multline*}
This yields
\begin{multline*}
E_3 = ([035][023][012]\prescript{[01]}{}{[123]^{-1}}) \triangleright \prescript{[01]}{}{\omega} \left([123], \prescript{[12][23]}{}{[345]} \right) [035]\triangleright \omega\Big ( [023][012]\prescript{[01]}{}{[123]^{-1}}, \prescript{[03]}{}{[345]}  \Big )\\
[0345][045]\triangleright [0234]^{-1}
([045][024]) \triangleright \omega \left( [012], \prescript{[01][12]}{}{[234]} \right)^{-1}
  [0245]^{-1}[025]\triangleright \omega\left([012], \prescript{[01][12]}{}{[245]} \right)^{-1}\\
   [012] \triangleright \omega ([012]^{-1}, [025]) [0245]  [045]\triangleright [0234][0345]^{-1}
  [035] \triangleright \omega\left([023], \prescript{[02][23]}{}{[345]} \right)^{-1} \\
  ([035][023][012])\triangleright \omega \Big ([012]^{-1},  \prescript{[02][23]}{}{[345]} \prescript{[02]}{}{[234]^{-1}} \prescript{[02]}{}{[245]^{-1}} [025]^{-1} \Big ).
\end{multline*}
Similarly, we have 
\begin{multline*}
\omega\Big ( [023][012]\prescript{[01]}{}{[123]^{-1}}, \prescript{[03]}{}{[345]}  \Big ) = ([023][012]\prescript{[01]}{}{[123]^{-1}}) \triangleright \omega\Big ( \prescript{[01]}{}{[123]}, \prescript{[01][13]}{}{[345]}\Big )^{-1}\\
\omega\Big ([023][012], \prescript{[01][12][23]}{}{[345]}\Big ).
\end{multline*}
By substituting, we get 
\begin{multline*}
E_3 = [035] \triangleright \omega\Big ([023][012], \prescript{[01][12][23]}{}{[345]}\Big ) [0345][045]\triangleright [0234]^{-1} ([045][024]) \triangleright \omega \left( [012], \prescript{[01][12]}{}{[234]} \right)^{-1}
 [0245]^{-1}\\
  [025]\triangleright \omega\left([012], \prescript{[01][12]}{}{[245]} \right)^{-1} [012] \triangleright \omega ([012]^{-1}, [025]) [0245]  [045]\triangleright [0234][0345]^{-1}\\
[035] \triangleright \omega\left([023], \prescript{[02][23]}{}{[345]} \right)^{-1} ([035][023][012])\triangleright \omega \Big ([012]^{-1},  \prescript{[02][23]}{}{[345]} \prescript{[02]}{}{[234]^{-1}} \prescript{[02]}{}{[245]^{-1}} [025]^{-1} \Big ).
\end{multline*}
Recall that for all $e,f,g \in E$, we have $\omega(ef,g) = \left( e\triangleright \omega(f,g) \right)\omega(e, \prescript{\partial f}{}{g})$, therefore
\begin{equation*}
 \omega\Big ([023][012], \prescript{[01][12][23]}{}{[345]}\Big ) = [023]\triangleright \omega\Big ( [012], \prescript{[01][12][23]}{}{[345]}\Big ) \omega\Big ([023], \prescript{[02][23]}{}{[345]}\Big ).
\end{equation*}
Recall that for all $e,f,g\in E$, we have $\omega(e,fg) = \left((efe^{-1}) \triangleright \omega(e,g)\right) \omega(e,f)$, therefore
\begin{multline*}
\omega \Big ([012]^{-1},  \prescript{[02][23]}{}{[345]} \prescript{[02]}{}{[234]^{-1}} \prescript{[02]}{}{[245]^{-1}} [025]^{-1} \Big ) = ([012]^{-1}\prescript{[02][23]}{}{[345]}[012])\triangleright \omega\Big ([012]^{-1}, \prescript{[02]}{}{[234]^{-1}}\\
\prescript{[02]}{}{[245]^{-1}}[025]^{-1}\Big ) \omega\Big ( [012]^{-1}, \prescript{[02][23]}{}{[345]}\Big ).
\end{multline*}
Recall that for all $e,f\in E$, we have $\omega(e,f) = e\triangleright \omega(e^{-1}, \prescript{\partial e}{}{f})^{-1}$, therefore
\begin{equation*}
\omega\Big ([012]^{-1}, \prescript{[02][23]}{}{[345]}\Big ) = [012]^{-1}\triangleright \omega\Big ( [012], \prescript{[01][12][23]}{}{[345]}\Big )^{-1}.
\end{equation*}
By substituting the above formulas, we get that $E_3 = 1$ is equivalent to $E_4 = 1$ with 
\begin{multline*}
E_4 = [035] \triangleright \omega\Big ([023], \prescript{[02][23]}{}{[345]}\Big ) [0345][045]\triangleright [0234]^{-1} ([045][024]) \triangleright \omega \left( [012], \prescript{[01][12]}{}{[234]} \right)^{-1}
 [0245]^{-1}\\
  [025]\triangleright \omega\left([012], \prescript{[01][12]}{}{[245]} \right)^{-1} [012] \triangleright \omega ([012]^{-1}, [025]) [0245]  [045]\triangleright [0234][0345]^{-1}\\
[035] \triangleright \omega\left([023], \prescript{[02][23]}{}{[345]} \right)^{-1} ([035][023]\prescript{[02][23]}{}{[345]}[012])\triangleright \omega \Big ([012]^{-1},  \prescript{[02]}{}{[234]^{-1}} \prescript{[02]}{}{[245]^{-1}} [025]^{-1} \Big ).
\end{multline*}
We apply the formula $lm = (\delta l \triangleright m) l$ with 
\begin{align*}
l &= [045]\triangleright [0234] [0345]^{-1}[035]\triangleright \omega\Big ([023], \prescript{[02][23]}{}{[345]}\Big )^{-1} \\
m&= ([035][023]\prescript{[02][23]}{}{[345]}[012])\triangleright \omega\Big ( [012]^{-1}, \prescript{[02]}{}{[234]^{-1}}\prescript{[02]}{}{[245]^{-1}}[025]^{-1}\Big ).
\end{align*}
This yields that $E_4 = 1$ is equivalent to $E_5 = 1$ with 
\begin{multline*}
E_5 = ([045][024]) \triangleright \omega \left( [012], \prescript{[01][12]}{}{[234]} \right)^{-1}
 [0245]^{-1} [025]\triangleright \omega\left([012], \prescript{[01][12]}{}{[245]} \right)^{-1} [012] \triangleright \omega ([012]^{-1}, [025])\\
   [0245] ([045][024]\prescript{[02]}{}{[234]}[012])\triangleright \omega\Big ( [012]^{-1}, \prescript{[02]}{}{[234]^{-1}}\prescript{[02]}{}{[245]^{-1}}[025]^{-1}\Big ) .
\end{multline*}
Recall that for all $e,f,g \in E$, we have $\omega(e,fg) = \left((efe^{-1})\triangleright \omega(e,g)\right) \omega(e,f)$, therefore 
\begin{multline*}
\omega\Big ( [012]^{-1}, \prescript{[02]}{}{[234]^{-1}}\prescript{[02]}{}{[245]^{-1}}[025]^{-1}\Big )  = ([012]^{-1}\prescript{[02]}{}{[234]^{-1}}[012])\triangleright \omega\Big ([012]^{-1}, \prescript{[02]}{}{[245]}^{-1}[025]^{-1}\Big ) \\
\omega\Big ([012]^{-1}, \prescript{[02]}{}{[234]^{-1}}\Big ).
\end{multline*}
Recall that for all $e,f\in E$, we have $\omega(e,f) = (efe^{-1})\triangleright \omega(e,f^{-1})^{-1}$, therefore 
\begin{equation*}
    \omega\Big ([012]^{-1}, \prescript{[02]}{}{[234]^{-1}}\Big ) = ([012]^{-1}\prescript{[02]}{}{[234]^{-1}}[012])\triangleright \omega\Big ([012]^{-1}, \prescript{[02]}{}{[234]}\Big )^{-1}.
\end{equation*}
Recall that for all $e,f\in E$, we have $\omega(e,f)^{-1} = e\triangleright \omega(e^{-1}, \prescript{\partial e}{}{f})$, therefore
\begin{equation*}
    \omega\Big ([012]^{-1}, \prescript{[02]}{}{[234]}\Big )^{-1} = [012]^{-1}\triangleright \omega\Big ( [012], \prescript{[02]}{}{[234]}\Big ).
\end{equation*}
By substituting, we get
\begin{multline*}
E_5 = ([045][024]) \triangleright \omega \left( [012], \prescript{[01][12]}{}{[234]} \right)^{-1}
 [0245]^{-1} [025]\triangleright \omega\left([012], \prescript{[01][12]}{}{[245]} \right)^{-1} [012] \triangleright \omega ([012]^{-1}, [025])\\
   [0245] ([045][024][012])\triangleright \omega \Big ([012]^{-1}, \prescript{[02]}{}{[245]^{-1}}[025]^{-1}\Big ) ([045][024])\triangleright \omega\Big ( [012], \prescript{[02]}{}{[234]}\Big ) .
\end{multline*}
Therefore $E_5 = 1$ is equivalent to $E_6 = 1$ with 
\begin{multline*}
E_6 = 
 [0245]^{-1} [025]\triangleright \omega\left([012], \prescript{[01][12]}{}{[245]} \right)^{-1} [012] \triangleright \omega ([012]^{-1}, [025])\\
   [0245] ([045][024][012])\triangleright \omega \Big ([012]^{-1}, \prescript{[02]}{}{[245]^{-1}}[025]^{-1}\Big ) .
\end{multline*}
By applying the formula $lm = (\delta l\triangleright m)l$ with 
\begin{align*}
    l &= [0245]\\
    m &= ([045][024][012])\triangleright \omega\Big ([012]^{-1}, \prescript{[02]}{}{[245]^{-1}}[025]^{-1}\Big),
\end{align*}
we get that $E_6 = 1$ is equivalent to $E_7 = 1$ with 
\begin{equation*}
E_7 = 
 [025]\triangleright \omega\left([012], \prescript{[01][12]}{}{[245]} \right)^{-1} [012] \triangleright \omega ([012]^{-1}, [025])
   ([025]\prescript{[02]}{}{[245]}[012])\triangleright \omega \Big ([012]^{-1}, \prescript{[02]}{}{[245]^{-1}}[025]^{-1}\Big ) .
\end{equation*}
Recall that for all $e,f,g\in E$, we have $\omega(e,fg) =\left( (efe^{-1})\triangleright \omega(e,g)\right)\omega(e,f)$, therefore
\begin{align*}
\omega\Big ( [012]^{-1}, \prescript{[02]}{}{[245]^{-1}}[025]^{-1}\Big ) &= ([012]^{-1}\prescript{[02]}{}{[245]^{-1}}[025]^{-1}[012])\triangleright \omega \Big ([012]^{-1}, [025]^{-1}\Big ) \omega\Big ([012]^{-1}, \prescript{[02]}{}{[245]^{-1}}\Big ) .
\end{align*}
Recall that for all $e,f\in E$, we have $\omega(e,f) = (efe^{-1})\triangleright \omega(e, f^{-1})^{-1}$, therefore
\begin{multline*}
\omega\Big ( [012]^{-1}, \prescript{[02]}{}{[245]^{-1}}[025]^{-1}\Big ) = ([012]^{-1}\prescript{[02]}{}{[245]^{-1}}[025]^{-1}[012])\triangleright \omega \Big ( [012]^{-1}, [025]^{-1}\Big )\\
 ([012]^{-1}\prescript{[02]}{}{[245]^{-1}}[025]^{-1})\triangleright \omega \Big ([012]^{-1}, \prescript{[02]}{}{[245]}\Big )^{-1}.
\end{multline*}
Recall that for all $e,f\in E$, we have $\omega(e,f)^{-1} = e\triangleright \omega(e^{-1}, \prescript{\partial e}{}{f})$, therefore
\begin{multline*}
\omega\Big ( [012]^{-1}, \prescript{[02]}{}{[245]^{-1}}[025]^{-1}\Big ) = ([012]^{-1}\prescript{[02]}{}{[245]^{-1}}[025]^{-1}[012])\triangleright \omega \Big ( [012]^{-1}, [025]\Big )^{-1}\\
([012]^{-1}\prescript{[02]}{}{[245]^{-1}})\triangleright \omega\Big ( [012], \prescript{[01][12]}{}{[245]}\Big ).
\end{multline*}
By substituting, we get
\begin{align*}
E_7 &= 
 [025]\triangleright \omega\left([012], \prescript{[01][12]}{}{[245]} \right)^{-1} [012] \triangleright  \omega \Big ([012]^{-1}, [025]\Big )
  [012]\triangleright  \omega \Big ( [012]^{-1}, [025]\Big )^{-1}\\
& \quad \quad \quad \quad [025]\triangleright \omega\Big ( [012], \prescript{[01][12]}{}{[245]}\Big )\\
& = 1.
\end{align*}
Therefore the 4-simplex condition is verified for $\langle 12345\rangle$, and $\psi$ is a $\sigma$-coloring of $L$.

\newpage

\bibliographystyle{alpha}
\bibliography{Bibliography.bib}

\end{document}